\documentclass[10pt]{amsart}
\usepackage{amsmath,amssymb,enumerate}
\usepackage{epsf,epsfig,amsfonts,graphicx,color,url}  
 \usepackage[applemac]{inputenc} 

 \usepackage{mathdots}
\usepackage{multirow}

\usepackage[all]{xy}
\usepackage{graphicx}

\definecolor{blue-violet}{rgb}{0.54, 0.17, 0.89}

 \date{30 3 2015}   
  
 \numberwithin{equation}{section}   
\def\cplxi{{\mskip2mu\sf{i}\mskip2mu}}
\newcommand{\trace}{\mathrm{trace}}

\newcommand{\lb}{\llbracket}
\newcommand{\rb}{\rrbracket}
\newcommand{\vvol}{\boldsymbol{\epsilon}}

\DeclareMathOperator{\const}{const}%
 
\usepackage{stmaryrd}  
 	\definecolor{antiquefuchsia}{rgb}{0.57, 0.36, 0.51}

\newtheorem{theorem}{Theorem}[section]
\newtheorem{lemma}[theorem]{Lemma}
\newtheorem{proposition}[theorem]{Proposition}
\newtheorem{corollary}[theorem]{Corollary}

\theoremstyle{definition}
\newtheorem{definition}[theorem]{Definition}

\theoremstyle{remark}
\newtheorem{remark}[theorem]{\rm\bf Remark}
\newtheorem*{definition*}{\rm\bf Definition}

\newcommand{\weg}[1]{}
\newcommand{\be}{\begin{equation}}
\newcommand{\ee}{\end{equation}}

\newcommand{\ol}{\overline}

\newcommand{\si}{\sigma}

\newcommand{\ttt}{t}
\newcommand{\bm}{\bf m}
\newcommand{\bp}{\bf p}

\newcommand{\End}{\operatorname{End}}

\newcommand{\lpl}{
  \mbox{$
  \begin{picture}(12.7,8)(-.5,-1)
  \put(2,0.2){$+$}
  \put(6.2,2.8){\oval(8,8)[l]}
  \end{picture}$}}

\newcommand{\ce}{\mathcal{E}}
\renewcommand{\P}{\mbox{\sf P}}
\newcommand{\J}{\mbox{\sf J}}
\newcommand{\Sc}{R}
\newcommand{\cT}{\mathcal{T}}

\newcommand{\nn}[1]{(\ref{#1})}

\def\sideremark#1{\ifvmode\leavevmode\fi\vadjust{\vbox to0pt{\vss
 \hbox to 0pt{\hskip\hsize\hskip1em
 \vbox{\hsize3cm\tiny\raggedright\pretolerance10000
 \noindent #1\hfill}\hss}\vbox to8pt{\vfil}\vss}}}%
                        
\title[Projective geometry and metrics]{Projectively related metrics,
  Weyl nullity, and metric projectively invariant equations} \date{}
\author{A. Rod Gover and Vladimir S. Matveev }
\thanks{ARG gratefully acknowledges support from the Royal
  Society of New Zealand via Marsden Grant 13-UOA-018; VSM thanks DFG,  DAAD and the University of Jena  for partial financial support}

\address{A.R.G.: Department of Mathematics, University of Auckland, Private Bag 
92019, Auckland 1042, New Zealand. \\
V.S.M.: Institute of
  Mathematics, Friedrich-Schiller-Universit\"at Jena, 07737 Jena
  Germany}
 \email{r.gover@auckland.ac.nz}
 \email{vladimir.matveev@uni-jena.de}



\begin{document}

\begin{abstract}
A metric projective structure is a manifold equipped with the
unparametrised geodesics of some pseudo-Riemannian metric. We make a
comprehensive treatment of such structures in the case that there is a
certain (well-motivated) algebraic restriction on the projective Weyl
curvature, a nullity condition. The analysis is simplified by a 
fundamental and canonical 2-tensor invariant that we discover. 
It leads to a new canonical tractor connection for these geometries
which is defined on a rank $(n+1)$-bundle. We show this connection is
linked to the metrisability equations that govern the existence of
metrics compatible with the structure.
The fundamental 2-tensor also
leads to a new class of invariant linear differential operators that
are canonically associated to these geometries; included is a third order
equation studied by Gallot et al.\ and via the new connection we show
its equivalence (on suitable geometries and classes of solutions) to
the first order metrisability equation.

 We apply the results to study the metrisability
   equation, in the nullity setting described. We obtain strong local
   and global results on the nature of solutions and also on the nature of the
   geometries admitting such solutions, obtaining classification
   results in some cases.  We show that closed Sasakian and K\"ahler
   manifold do not admit nontrivial solutions. We also
   prove that, on a closed manifold, two nontrivially projectively
   equivalent metrics cannot have the same tracefree Ricci
   tensor.  We show that on a closed manifold a metric
     having a nontrivial solution of the metrisablity equation cannot
     have a two dimensional nullity space at every point. In these
   statements the meaning of trivial solution is dependent on the
   context.

There is a function $B$ naturally appearing if a metric projective
structure has nullity.  We analyse in detail the case when this is not
a constant, and describe all nontrivially projectively equivalent
Riemannian metrics on closed manifolds with nonconstant $B$. 
\end{abstract}

\maketitle

\section{Introduction}\label{intro}

Affine connections $\nabla$ and $\bar \nabla$ are said to be
projectively equivalent if each $\nabla$-geodesic is, after
reparametrisation, a $\bar \nabla$-geodesic.  Two connections
differing only by torsion evidently have the same geodesics, and so,
within this article, we shall say that a {\em projective structure} on
a manifold is an equivalence class ${\bf p}=[ \nabla]$ of all torsion
free connections projectively equivalent to some given torsion-free
affine connection $\nabla$.

A generic projective structure $\bp$ does not contain the Levi-Civita
connection of any metric.  However on a manifold $M$, any metric $g$
(of arbitrary signature) canonically determines a torsion-free
connection $\nabla^g$ that preserves $g$; this is the Levi-Civita
connection. Thus $g$ determines a projective structure ${\bf
  p}=[\nabla^g]$. Within this equivalence class there is a
distinguished non-empty subset of connections, namely the subset
characterised by property that each connection contained therein is the
Levi-Civita connection of some metric.

In the (pseudo-)Riemannian setting it is typically most convenient to
work directly with metrics. So we say two metrics $g$ and $\bar g$   of arbitrary signature are
{\it projectively equivalent} if they share the same geodesics up to
re-parametrisation, that is if their Levi-Civita connections
$\nabla^g$ and $\nabla^{\bar{g}}$ are projectively equivalent; we say
further that $g$ and $\bar g$ are {\em affinely} projectively
equivalent if $\nabla^g = \nabla^{\bar{g}}$.  Projective equivalence is an equivalence relation on
metrics and we make the following definition.
\begin{definition}
On a manifold $M$, a {\em metric projective structure} ${\bf m}=
\llbracket g\rrbracket$ is an equivalence class of projectively
related metrics.
\end{definition}

For a generic metric, its projective class determines the metric, up
to multiplication by constants, so this distinquished subset of
connections contains only one representative up to dilation, see
e.g.\ \cite{relativity}. In this case questions regarding the geometry
of the metric projective structure $\bm$ can be reduced to questions
concering the metric. On the other hand, there exist non-proportional
metrics $g$ and $\bar g$ that have the same Levi-Civita connection,
i.e.\ are affinely equivalent. In the Riemannian case (i.e.\ $g$
positive definite)  this situation is well covered by the classical
literature \cite{Eisenhart53A45,deRham}. However there also exist
examples of projectively equivalent metrics that are not affinely
equivalent. The investigation of such non-affinely projectively
equivalent metrics is also a classical topic in differential geometry,
and was studied by Beltrami, Levi-Civita, Painleve, Weyl, Eisenhart,
and Thomas, for example, see \cite{EisenhartBook}. In fact this topic
was one of the main topics in Japan in the period of the 1950s through
to the 1960s, and also in the former USSR over the period from 1960s
through to the 1990s, see \cite{mikesSur}. Presently there is a strong
revival of interest in this area owing to new ideas entering from the
directions of integrable systems and parabolic geometries and via this
some of the earlier open questions have been resolved
\cite{bryantDE,eastwood1,hyperbolic,dedicata}.

Thus we are motivated to study the geometry of a manifold $M$ equipped
with a metric projective structure ${\bf m}$, and we will call
$(M,{\bf m})$ a {\em metric projective geometry}. Evidently this
uniquely determines a projective structure $\bp$ on $M$, but there is
of course more information in $\bm$. Metric projective geometry is an
analogue of the idea of a conformal geometry where one considers a
manifold equipped with a conformal equivalence class of metrics.
However it is important to emphasize that, in the detail, there are
significant differences.  For example for a given metric projective
geometry $(M,{\bf m})$ it can be that, as mentioned above, up to
constant dilation, there is only one metric in the equivalence class,
and a generic metric has this property. Alternatively for a different
metric projective structure $\bm$ there can be metrics $g$ and $\bar
g$ in $\bm$ that have different signatures.  A key focus of the
current article is to understand metric projective geometries $(M,{\bf
  m})$ where ${\bf m}$ includes metrics that are not affinely
equivalent.

Associated to any projective geometry $(M,{\bf p})$ is a tensor
invariant ${W^i}_{jk\ell}$ known as the projective Weyl tensor, see
\nn{weylt} below. A rather interesting algebraic condition that we can
impose on this is that it is has {\em nullity}, meaning that there 
exists a nonzero vector field $v \in
\Gamma(TM)$ such that
\begin{equation}\label{nullc}
{W^i}_{jk\ell}v^j=0 .
\end{equation}
The projective Weyl tensor is a natural invariant of metric projective
structures $(M,{\bf m})$, since these have a canonical projective
structure. We will show that in this setting the Weyl nullity condition
is extremely important and leads to subtle and unexpected phenomema.

In dimension $2$, the projective Weyl curvature vanishes identically
and so all structures trivially have nullity. However in dimension
$\ge 3$, the projective Weyl curvature of a generic metric does not
have projective Weyl nullity. Even if a metric non-affinely admits a
projectively equivalent metric then generically the Weyl tensor does
not necessarily have nullity.  For example, for a generic choice of functions
$X_1(x^1)$, $X_2(x^2)$, and $X_3(x^3)$ of the indicated variables, the
following 3-dimensional metric from \cite{Levi-Civita} {\tiny$$
\begin{array}{cl}
  (X_1(x^1 )- X_2(x^2))(X_1(x^1 )- X_3(x^3))(dx^1)^2 &+ 
(X_1(x^1 )- X_2(x^2))(X_2(x^1 )- X_3(x^3))(dx^2)^2\\ &  +  (X_1(x^1 )- X_3(x^2))(X_2(x^1 )- X_3(x^3))(dx^3)^2  
\end{array}$$}
 is projectively equivalent to 
{\small$$\begin{array}{cl}
 \tfrac{(X_1(x^1 )- X_2(x^2))(X_1(x^1 )- X_3(x^3))}{X_1(x^1)^2 X_2(x^2)X_3(x^3)} (dx^1)^2 & + 
\tfrac{(X_1(x^1 )- X_2(x^2))(X_2(x^1 )- X_3(x^3))}{X_1(x^1) X_2(x^2)^2X_3(x^3)}(dx^2)^2 \\ & + \tfrac{(X_1(x^1 )- X_3(x^2))(X_2(x^1 )- X_3(x^3))}{X_1(x^1) X_2(x^2)X_3(x^3)^2}(dx^3)^2  
\end{array}. $$}
 If at least one of the functions $X_i$ is not constant, the two
 displayed metrics are not affinely equivalent. It can be shown that
 for such metrics the existence of projective Weyl nullity is
 equivalent to the property that the metrics have constant curvature,
 which is clearly not the case for a generic choice of functions
 $X_i$. (A direct way to see this it is to calculate the Weyl
 curvature and check; an equivalent calculation was done by Fubini and
 Bolsinov et al.\ in \cite{kiosak,Fubini}.)  Note that the metrics are
 defined if at all points $(x^1, x^2, x^3)$ the values of the
 functions $X_1, X_2, X_3$ are distinct and are different from zero,
 which is of course a generic condition.  The example can be
 generalized for all dimensions $n \ge 3$.

Although such examples exist, we will show that Weyl nullity is
necessarily present for a very broad class of geometries and cases of
interest in the literature; interestingly in these studies this link to the
Weyl tensor has previously gone unnoticed.
Indeed in Theorem
\ref{WeqK}, we will show that, for metric projective geometries, the
nullity condition on the projective Weyl tensor naturally generalises
the so-called curvature constancy condition from \cite{Otsuki,Gray},
which in turn was a generalisation of the curvature nullity condition
suggested by Chern et al.\ in \cite{chern}. This condition was
extensively studied in the literature, under different names (for
example $K$-nullity in \cite{Ta,rosental, rosental2, MS}, or simply
nullity in \cite{ferus}), and without the realisation of the
connection to nullity of the projective Weyl tensor that we establish
here in Theorem \ref{WeqK} (and see Remark \ref{rem:nullity2}).  The
difference between the curvature constancy condition of Otsuki and
Gray and the Weyl nullity (as above) is the following: at a point they
coincide but most authors have studied curvature constancy assuming
that the object that we call $B$ in Section \ref{weylB} is a constant,
while for us it may be any function.

That the non-trivial Weyl nullity appears naturally in many other
well-known geometric structures and constructions is seen by
combining our Theorem \ref{WeqK} with the existing literature.  For
example, a warped product metric  $\pm dt^2 + f(t)^2 h(x)_{ij}dx^idx^j$
always has projective Weyl nullity (and in this case the function $B$
of Proposition \ref{mweyl} is generically not a constant). In
particular cone metrics, which are warped product metrics with $f(t)=
t^2$ have nullity (and in this case $B= 0$).  Sasakian metrics also
always have Weyl nullity (and $B= 1$). Moreover, the existence of
sufficiently many solutions of certain geometric PDE implies
nullity. For example any of the following conditions implies Weyl
nullity: the existence of a non-trivial special conformal Killing
vector field, see \cite{devries} and \cite{Couty}; or a non-trivial
concircular vector field (see e.g. \cite{Tashiro}); or two
non-proportional Einstein metrics in the same conformal class, or two
projectively equivalent metrics which are not affinely equivalent such
that one of them is Einstein (see e.g. \cite{einstein}); or three
pointwise linearly independent projectively equivalent metrics that
are not affinely equivalent (see e.g. \cite{kiosak}); or two
projectively, but not affinely, equivalent metrics with the same
stress-energy tensor \cite{KM2014}.

A key observation is that Weyl nullity as in \nn{nullc}
allows a completely different approach to metric projective geometry.
Our first main focus is the development and application of this as a 
conceptual and calculationally effective framework.  The first
main result is Theorem \ref{fth} which constructs a new fundamental
invariant of such manifolds. This invariant is a symmetric two-tensor
$\phi_{ab}$ and is of direct interest because it is determined by the
metric projective geometry $(M,\bm)$ (with Weyl nullity) but is not in
general the restriction of a pseudo-Riemannian or projective geometry
invariant. Importantly it also yields remarkable simplifications to
the equations governing many of the natural problems. One might hope
that the $\phi$-invariant is usefully available on some class of
metric projective geometries that do not have Weyl nullity. However
this is not the case, as we show in Theorem \ref{converse}.

As an immediate application of the $\phi$-invariant, we show in
Section \ref{iops} that this immediately leads to new linear
differential operators that are canonical and invariant on metric
projective structures with nullity but which are not in general the
restriction of projective or (pseudo-)Riemannian invariant operators.
In two examples these are shown to provide a bridge between so called
first BGG equations, of current interest in parabolic geometry (see
e.g.\ \cite{CGHduke}), and certain classically studied equations
including the Gallot-Obata-Tanno equation \nn{one}, see in particular
Theorem \ref{G-Obthm}.

The next direction of application is in the treatment of Einstein
metrics. The $\phi$-invariant leads to extremely simple proofs of the
Beltrami Theorem (see Corollary \ref{Bel}) and the result that if $g$
is an Einstein metric, then any projectively related metric is also
Einstein (in the presence of Weyl nullity), see Corollary \ref{1ein2}.
Proposition \ref{1ein22} explains that the explicit mention of Weyl
nullity may be dropped if we require that the metrics are non-affinely
equivalent.

On projective manifolds there is a canonical invariant calculus
associated to the Cartan connection. An early version of this is due
to Thomas \cite{Thomas}, while the modern treatment was founded in
\cite{BEG}. This {\em tractor calculus} is based around an invariant
linear connection on a natural rank $(n+1)$ vector bundle that extends
(a density twisting of) the tangent bundle, see Section \ref{trS}.
Returning to the development of theory, in Section \ref{newcon} we
show that, on metric projective geometries, projective Weyl nullity
leads to a 1-parameter family of alternative canonical tractor
connections and a distinguished connection $\nabla^{\cT_1}$ within
this family; see Proposition \ref{alphaf} and the subsequent
discussion. This distinguished connection turns out to be useful for
the problem of treating non-affinely projectively related metrics, as
we explain in later sections. As a more immediate application, we
explain in Theorem \ref{GOTthm} that sections of $S^2\cT^*$ (where
$\cT^*$ is the dual tractor bundle) that are parallel for this
connection are equivalent to solutions of the
Gallot-Obata-Tanno equation \nn{one}.

The final Section in the first part of the paper, Section \ref{prol},
is concerned with analyzing Sinjukov's metrisability equation
\nn{metr} and its consequences; this is the projectively invariant
equation governing the existence of a Levi-Civita in a projective
class, and is an instance of a first BGG equation in the sense of
\cite{CGHjlms}. The first main result there begins with the
prolongation tractor connection found by Eastwood et al \cite{EM} (and
recalled in Theorem \ref{EMthm}) and shows that, on metric projective
structures with Weyl nullity, this leads to a simpler prolonged system
that is canonical and invariant on metric projective structures with
Weyl nullity. In the case of $B$ constant this agrees precisely with
the connection $\nabla^{\cT_1}$ (of Proposition \ref{alphaf}) on
$S^2\cT$, while for $B$ not constant it is a simple (and generalising)
modification of this, see Theorem \ref{4.7} and Theorem
\ref{mconnthm}.  There are immediate applications of these
results. For example in Corollary \ref{GOT=m} we show that suitably
generic solutions of the Gallot-Obata-Tanno equation \nn{one} are the
same as similarly generic solutions of the metrisability equation on
metric projective manifolds with Weyl nullity and a metric with $B$
constant. In fact using the full information at hand more general
results are available, see Proposition \ref{GOT=met} and Proposition
\ref{mettoGOT}.  Next in Section \ref{solS} we apply the machinery mentioned to obtain strong results for solutions of the
metrisability equation: In Theorem \ref{Btypefix} we show that if $B$
is constant (or non-constant) for one metric in $\bm$ then it is
constant (resp.\ non-constant) for all metrics in $\bm$. In Theorem
\ref{mformthm} we describe a severe constraint of the form of
solutions, in the case that $B$ is not constant.

Part of the local analysis for the development in Section
\ref{newconn} of the simpler prolonged system, and tractor connection on $S^2\cT$, relies on Theorem
\ref{mid2n}; the proof of this requires some non-trivial linear
algebra, and so is deferred  to Section
\ref{Vsec}. The technical tools developed in Section \ref{Vsec} and
the earlier sections are then used to establish some immediate
consequences: Theorem \ref{converse} has been mentioned already while
Theorem \ref{new} shows that on a metric projective manifold $(M,\bm)$
with Weyl nullity, there can only be strictly non-proportional metrics
(as defined above the Theorem) in $\bm$ if each has constant
curvature.

Sections \ref{BnC} and \ref{BC} are concerned with describing, on
closed manifolds, the nature of the metric for metric projective
structures with nullity, and where there is also another solution of
the metrisability equation \ref{metr}.  Section \ref{BnC} assumes that
$B$ is not constant; under this assumption we describe such metrics
$g$ locally. As a consequence, Theorem \ref{globalBnC} classifies such
metric projective structures on closed manifolds under the additional
assumption that one of the metrics is Riemannian.  Section \ref{BC}
considers the case of closed manifolds, $B$ constant, and two
non-affinely projectively related metrics.  The result is Theorem
\ref{Bcon} which states that in this case any metric $\bm$ is, up to
dilation with a possibly negative coefficient, the Riemannian metric
of constant positive sectional curvature $+1$.  We apply this result
to show that closed Sasakian manifolds of nonconstant curvature do not
admit projectively but non-affinely equivalent metrics, see Corollary
\ref{sasy}. We also show that if a K\"ahler metric admits a
projectively but non-affinely equivalent metric then it has a nullity
with $B=0$; this implies that on a closed manifold any metric
projectively equivalent to a K\"ahler metric is affinely equivalent to
it.  In Section \ref{finalS} we show in Corollary
  \ref{ccccc} that if a metric projective structure $\bm$, on a closed
  manifold $M$, contains two non-affinely equivalent metrics that
  have the same trace-free part of the Ricci tensor, then all metrics
  in this projective metrics structure have constant sectional
  curvature.  In Section \ref{kathi} we show, in Corollary \ref{cccccc},
  that if a metric projective structure $\bm$ on a closed manifold $M$ is 
  such that, first, the Weyl nullity is at least two-dimensional at every
  point and, second, it contains two non-affinely equivalent metrics, then all metrics
  in this metric projective structure have constant sectional
  curvature.

\subsection{Some conventions}

Throughout any manifold $M$ will be assumed smooth, connected, and of
dimension $n\geq 2$. Similarly the standard structures on this (such as metrics) will be assumed smooth. The term smooth here and throughout means
$C^\infty$. Unless otherwise stated indices on tensors and bundle will
be abstract indices in sense of Penrose.  For example the tangent
bundle $TM$ will often be denoted $\ce^a$. Then its symmetric tensor
power $S^2TM$ is denoted $\ce^{(ab)}$. Connections on the tangent
bundle its dual and tensor powers will be torsion-free.  Statements
will be said to hold {\em almost everywhere} if they are true on an
open dense set.

\section{Background}\label{backg}

Let $(M,\nabla)$ be a special affine manifold (of dimension $n\geq
2$), meaning here that $\nabla$ is a torsion-free affine connection
and that locally everywhere this preserves a volume density; we do not
assume that $M$ is oriented.  The curvature $R^j{}_{i\ell k}$ of the
connection $\nabla$ is given by
$$
[\nabla_\ell,\nabla_k]v^j=R^j{}_{i\ell k} v^i .
$$ The Ricci curvature is defined by  $R_{jk} = {R^a}_{jak}$ and this
is symmetric.  


The projective Weyl tensor is
\begin{equation} \label{weylt} 
{W^i}_{jk\ell}:={R^i}_{jk\ell}+\tfrac{1}{n-1}\left({\delta^{i}_\ell}
\, R_{jk}-{\delta^{i}_k}\, R_{j\ell}\right) ={R^i}_{jk\ell}+\left({\delta^{i}_\ell}
\, \P_{jk}-{\delta^{i}_k}\, \P_{j\ell}\right), 
\end{equation}
where $\P_{jk}=\frac{1}{n-1}R_{jk}$ is called the {\em projective
  Schouten tensor}. The Weyl tensor is the trace-free part of the
curvature $R^i{}_{j\ell k}$; any contraction of the index $i$ (of $ {W^i}_{jk\ell}$) with a
lower index results in $0$. It is easily verified that in dimension 2 
the Weyl tensor is identically zero. 

From the differential Bianchi identity we obtain the identity
$$
\nabla_c W^c{}_{dab}{}=(n-2)C_{dab},  
$$ 
where
\begin{equation}\label{Cotton}
C_{cab}:= \nabla_a \P_{bc}-\nabla_b \P_{ac}
\end{equation} 
is called the {\em
  projective Cotton tensor}.

\subsection{The conformal Weyl tensor and Einstein metrics}\label{conf}

On a pseudo-Riemannian manifold $(M,g)$ the {\em conformal Weyl tensor}
$C^a{}_{bcd}$ is the completely metric trace-free part of the
curvature $R^a{}_{bcd}$ of the Levi-Civita connection. So
$C^a{}_{bad}=0$ and also $C^a{}_{bcd}g^{bc}=0$. 
We may compare this with the projective
Weyl tensor:
\begin{proposition}{} \cite[Proposition 5.5]{EM} \label{EMprop}
On a pseudo-Riemannian manifold $(M,g)$ one has 
\begin{equation}\label{EMform}
(n-2)W^{i}{}_{jkl}= (n-2)C^{i}{}_{jkl} + (\delta^i_k\mathring{\P}_{lj}-\delta^i_l\mathring{\P}_{kj})+ (n-1)(\mathring{\P}^i{}_k g_{jl}-\mathring{\P}^i{}_l g_{jk})
\end{equation}
where $\mathring{\P}_{kl} $ is the trace-free part of the Projective
Schouten tensor.
\end{proposition}

Recall that in the case of dimension $n\geq 3$, we say a
pseudo-Riemannian metric $g$ is {\em Einstein} if its Ricci tensor, or
equivalently projective Schouten tensor $\P_{ab}$, is proportional to
the metric: $\P_{ab}=\lambda g_{ab}$.  It then follows from the
contracted Bianchi identity that the function $\lambda$ is constant.
In the case of dimension $n=2$ we take $\P_{ab}=\lambda g_{ab}$ {\em
  with $\lambda$ constant} to be the definition of an {\em Einstein
  metric}.

With these definitions we have the following consequences of Proposition \ref{EMprop} (cf.\ \cite{GMac,Nur} ):
\begin{corollary}\label{Eins}
On an Einstein pseudo-Riemannian manifold $(M,g)$ the conformal and
projective Weyl tensors agree, that is
$$
W^{i}{}_{jkl}= C^{i}{}_{jkl}.
$$ 
If $n\geq 3$ then, conversely, agreement of the two
pseudo-Riemannian Weyl tensors implies the metric is Einstein.
\end{corollary}
\begin{proof}
In dimension 2 the first statement is true trivially. In other dimensions 
both statements are immediate from \nn{EMform}, as Einstein
is equivalent to $\mathring{\P}_{ab}=0$.
\end{proof}

\subsection{Projective geometry}\label{pg} Recall the notion of projective structure was introduced in Section \ref{intro}, as follows.  A projective structure
$(M^n,\bp)$, $n\geq 2$, is a smooth manifold equipped with an
equivalence class $\bp$ of torsion-free affine connections. The class is
characterised by the fact that two connections $\nabla$ and
$\widehat\nabla$ in $\bp$ have the same geodesics up to
parametrisation. Explicitly the transformation relating these
connections on $TM$ and $T^*M$ are given by
 \begin{equation} \label{ptrans}
 \widehat\nabla_a Y^b = \nabla_a Y^b +\Upsilon_a Y^b + \Upsilon_c Y^c 
 \delta^b_a, \quad \mbox{and} \quad \widehat\nabla_a u_b = \nabla_a u_b - \Upsilon_a u_b -\Upsilon_b u_a,
\end{equation}
where $\Upsilon$ is some smooth section of $ T^*M$.  In the setting of
a projective structure $(M,\bp)$ any connection $\nabla\in \bp$ is called
a {\em Weyl connection} or {\em Weyl structure} on $M$.  In the
following we shall consider only special affine connections from $\bp$,
in which case the corresponding Weyl structure is often called a
{\em choice of scale}. If $\nabla $ and $\widehat{\nabla}$ are special affine connections then $\Upsilon_b$ is exact, meaning $\Upsilon_b=\nabla_b f$ for some function $f$ (see e.g. \cite{GMac}).

Under a projective transformation of connection, as in \nn{ptrans}, it
is easily verified (and well known) that the projective Weyl curvature
$W^a{}_{bcd}$ is unchanged. Thus it is an invariant of the projective
structure $(M,\bp)$. In dimension 2 the Cotton tensor $C_{abc}$ is
projectively invariant.  These are the complete obstructions to
projective flatness: in dimensions $n\geq 3$ there is a flat
connection $\nabla\in \bp$ if and only if $W^a{}_{bcd}=0$
everywhere. In the dimension 2 the same statement is true with
$W^a{}_{bcd}$ replaced by the Cotton tensor $C_{abc}$. Thus, via Corollary \ref{Eins}, there is another 
immediate consequence of Theorem \ref{EMform}:
\begin{corollary}\label{Etof}
In dimensions 2 and 3 Einstein manifolds are projectively flat.
\end{corollary}
\begin{proof}
For dimension 2 this is immediate from the definition of the Cotton
tensor and the meaning of Einstein.  In the case of dimension 3 it
follows from the fact that in dimension three the conformal Weyl
tensor is identically zero.
\end{proof}

The Schouten tensor $\P_{ab}$ plays an important role in projective
differential geometry, even though it is far from invariant under
projective transformations. Under the transformations \nn{ptrans} we have,
\begin{equation}\label{pcurvtrans}
\P^{\widehat{\nabla}}_{ab}= \P^{\nabla}_{ab}-\nabla_a
\Upsilon_b+\Upsilon_a\Upsilon_b~.
\end{equation}

As is usual in projective geometry, we write $\ce(1)$ for the positive 
$(2n+2)$nd root  of the oriented bundle
$(\Lambda^n T M )^2$.  Observe that any connection $\nabla \in \bp$
determines a connection on $\ce(1)$ as well as its real powers
$\ce(w)$, $w\in \mathbb{R}$; we call $\ce(w)$ the bundle of projective
densities of weight $w$. Given any bundle $\mathcal{B}$ we shall write
$\mathcal{B}(w)$ as a shorthand notation for $\mathcal{B}\otimes
\ce(w)$.

\subsection{The projective tractor bundle and connection} \label{trS}
By the definition of a projective geometry $(M,\bp)$, there is no
distinguished connection on $TM$. However there is a canonical connection
(due to Cartan and Thomas \cite{Cartan,Thomas}) on a closely related
rank $(n+1)$ natural bundle. This connection, now known as the tractor
connection, enables an invariant calculus \cite{BEG}. Since the
details of this will be important for us, we recall briefly the
definition and construction of this connection following
\cite{BEG,CGMac} (and see \cite{CapGoTAMS} for recovering from this
the equivalent Cartan bundle and connection  \cite{Cartan}).

Canonically associated to the projective density bundle $\ce(1)$ is
its first jet prolongation $J^1\ce(1)\to M$. By definition, its fiber
over $x\in M$ consists of all one--jets $j^1_x\sigma$ of local smooth
sections $\sigma\in\Gamma(\ce(1))$ defined in a neighborhood of $x$.
Mapping $j^1_x\sigma$ to $\sigma(x)$ defines a surjective
bundle map $J^1\ce(1)\to\ce(1)$, called the \textit{jet
  projection}. If $j^1_x\sigma$ lies in the kernel of this projection,
so $\sigma(x)=0$, then the value $\nabla\sigma(x)\in
T^*_xM\otimes\ce_x(1)$ is the same for all linear connections $\nabla$
on the vector bundle $\ce(1)$. This identifies the kernel of the jet
projection with the bundle $T^*M\otimes\ce(1)$. (See for example
\cite{palais} for a general development of jet bundles.)

We will write $\cT^*$, or $\ce_A$ in abstract index notation, for
$J^1\ce(1)$ and $\cT$, or $\ce^A$ in abstract index notation, for the
dual vector bundle. Then observe that the jet projection is a canonical
section $X^A$ of the bundle $\ce^A\otimes\ce(1)=\ce^A(1)$. Similarly,
the inclusion of the kernel of this projection can be viewed as a
canonical bundle map $\ce_a(1)\to\ce_A$, which we denote by
$Z_A{}^a$. In this notation we have a canonical sequence  
\begin{equation}\label{euler}
0\to \ce_a(1)\stackrel{Z_A{}^a}{\to} \ce_A \stackrel{X^A}{\to}\ce(1)\to 0, 
\end{equation}
which is the well-known  jet exact sequence (at 1-jets) for the bundle $\ce(1)$.

We write $\ce_A=\ce(1)\lpl \ce_a(1)$ to summarise the composition
structure in \nn{euler}.  As mentioned earlier, any connection $\nabla
\in \bp$ determines a connection on $\ce(1)$. On the other hand a
linear connection on $\ce(1)$ is the same as a splitting the 1-jet
sequence \nn{euler}.  Thus given such a choice $\nabla\in \bp$ we have the direct sum
decomposition $\ce_A \stackrel{\nabla}{=} \ce(1)\oplus \ce_a(1) $ with
respect to which we define a connection by
\begin{equation}\label{pconn}
\nabla^{\mathcal{T}^*}_a \binom{\si}{\mu_b}
:= \binom{ \nabla_a \si -\mu_a}{\nabla_a \mu_b + \P_{ab} \si},
\end{equation}
where, recall, $\P_{ab}$ denotes the projective Schouten tensor of
$\nabla$.  An easy calculation shows that the connection \nn{pconn} is
independent of the choice $\nabla \in p$, and so
$\nabla^{\mathcal{T}^*}$ is determined canonically by the projective
structure $\bp$. This is the {\em cotractor} connection of \cite{BEG},
and is the {\em normal} connection, i.e. it is equivalent to the
normal tractor Cartan connection \cite{CapGoTAMS}.  Thus we call
$\cT^*=\ce_A$ the {\em cotractor bundle}, and we note the dual {\em
  tractor bundle} $\cT=\ce^A$ has canonically the dual {\em tractor
  connection}. In terms of a splitting dual to that above this is
given by
\begin{equation}\label{tconn}
\nabla^\cT_a \left( \begin{array}{c} \nu^b\\
\rho
\end{array}\right) =
\left( \begin{array}{c} \nabla_a\nu^b + \rho \delta^b_a\\
\nabla_a \rho - \P_{ab}\nu^b
\end{array}\right).
\end{equation}

Now consider $\ce^{( BC)}=S^2\mathcal{T}$. It follows immediately that
this has the composition series
$$
\ce^{(bc)}(-2)\lpl\ce^b(-2)\lpl \ce(-2),
$$
and the normal tractor connection is given on $S^2\mathcal{T}$ by
\begin{equation}\label{s2conn-orig}
\nabla_a^{\cT}\left( \begin{array}{c}
\si^{bc}\\
\mu^b\\
\rho
\end{array}
\right)=
\left( \begin{array}{c}
\nabla_a \si^{bc} + \delta^b_a \mu^c + \delta^c_a \mu^b \\
\nabla_a \mu^b + \delta^b_a\rho - \P_{ac}\si^{bc}  \\
\nabla_a \rho - 2 \P_{ab}\mu^b
\end{array}
\right).
\end{equation}
This connection on $S^2\mathcal{T}$ will be important to us below. It
closely linked to the geometry of metric projective structures. For
example a parallel section of
$S^2\mathcal{T}$ with $\si$  non-degenerate (as a bilinear tractor form
on $T^*M$) determines and, is equivalent to, an 
Einstein metric $g$ such that $\nabla^g\in \bp$, see
\cite{Armstrong,CGMac,GMac}.

\subsection{Weyl nullity}\label{Wn}

We shall say that the projective Weyl tensor 
$$
{W^i}_{jk\ell}:={R^i}_{jk\ell}
-\tfrac{1}{n-1}\left({\delta^{i}_\ell} \, R_{jk}-{\delta^{i}_k}\, R_{j\ell}\right)
$$ has {\em nullity} at a point $x$, if there exists a nonzero $v \in
T_xM$ such that
$$
{W^i}_{jk\ell}v^j=0 \quad \mbox{at} \quad x.
$$ 

Note that since the Weyl curvature is an invariant of projective manifolds 
(i.e.\ it is unchanged by the projective transformation \nn{ptrans})
 Weyl nullity is a property of the projective structure;
it is not dependent on the choice of any connection from the projective
class. It is detected precisely by the projective invariant 
$$
\stackrel{m}{\mathcal{W}}:=W^i{}_{a_1 k p}\cdots  W^j{}_{a_m \ell q} \quad m\in \{1,\cdots ,m\}
$$ 
where the sequentially labelled indices $a_1\cdots a_n$ are skewed
over. For example a projective structure has Weyl nullity at $x$ if
and only if for $m=n$ this invariant vanishes at $x$. The Weyl
nullity space has dimension at least 2 if and only if the invariant
with $m=n-1$ vanishes, and so forth.

\begin{remark}\label{disc}
Note that if $v$ is in the nullity of the projective Weyl curvature at $x$ then 
$$
v^jC_{jkl} 
$$ 
is projectively invariant at $x$.
This is because under the projective transformation
$\nabla\mapsto \widehat{\nabla}$ as in \nn{ptrans} the projective
transformation of $C_{jkl}$ is by the addition of $\Upsilon_i
W^{i}{}_{jkl}$ (up to a constant multiple).

This observation only has significance in dimensions $n\geq 3$: Any
projective manifold $(M,\bp)$ of dimension 2 has Weyl nullity
trivially, since the projective Weyl tensor is identically zero in
this case. On the other hand, as mentioned earlier,  
 in dimension 2 the projective Cotton tensor is
projectively invariant.
\end{remark}

\subsection{Projective differential equations and Weyl nullity}\label{BGG}

It is natural to ask what projectively invariant differential
equations have solutions implying Weyl nullity. One easy case is available. If 
$$
V^B \stackrel{\nabla}{=}  \left( \begin{array}{c} \nu^b\\
\rho
\end{array}\right), \qquad \nabla\in \bp,
$$ is a section of $\cT$ that is parallel for the standard tractor
connection then clearly $V^B$ is annihilated by the tractor curvature.
But then using the formula for the latter, as in e.g. \cite{BEG} it
follows that $\nu^b$ is everywhere in the Weyl nullity.

Now $V^B$ parallel implies that 
$$
\nabla_a\nu^b+ \rho\delta^b_a=0; 
$$
and
this is by definition a (projective geometry) first BGG equation in
the sense of \cite{CGHjlms,CGHduke}.  Conversely if $\nu^b$ solves
the displayed equation then $\rho=
-\frac{1}{n}\nabla_a\nu^a$ and two facts follow. First, differentiating the display,  one easily
shows that $\nabla_a\rho=\P_{ab}\nu^b$. So $V^B:=(\nu^b,~\rho)$ is
parallel for the (normal) standard tractor connection. This exactly
means all solutions of this first BGG equation are {\em normal}, again
in the sense of \cite{CGHjlms,CGHduke}. Second if $V^B$ is parallel
and non-trivial then $\nu^b$ must be non-zero on an open dense set,
since $V^B$ is part of the data of the 1-jet of $\nu^b$. By
continuity, it follows there is Weyl nullity everywhere.  Thus we have
the following result.
\begin{proposition}\label{psysth}
On a connected projective manifold $(M,\bp)$, a non-trivial solution of the
projectively invariant system 
\begin{equation}\label{psys1}
\nabla_a\nu^b+ \rho\delta^b_a=0 
\end{equation}
implies Weyl nullity everywhere, and $W^a{}_{bcd}\nu^b=0$ everywhere. 
\end{proposition}

\begin{remark}
Solutions of the equation \nn{psys1} are sometimes called {\em
  concircular} vector fields (at least in the pseudo-Riemannian
setting) \cite{Tashiro}.  Using the curvature formula from \cite{BEG},
one also sees that $C_{dab}\nu^d=0$ for any solution of \nn{psys1}.
\end{remark}

\section{Weyl nullity on a pseudo-Riemannian manifold}\label{weylB}

Note that the Levi-Civita connection of a metric is obviously a
special affine connection, since it preserves the volume density of
the metric.  Here we work on a metric projective structure $(M,\bm )$
of dimension $n\geq 2$, with no restriction on the signature of
possible metrics in $\bm$.  We wish to understand the implications of
projective Weyl nullity in this setting.

\begin{proposition}\label{mweyl} Suppose that $(M,\bm )$ has a 
projective Weyl nullity at a point $x$, i.e., assume there exists a
vector $v\in T_x M\setminus \{0\}$  such that
$$ 
W^i_{\ jk\ell} v^j=0 \quad \mbox{at} \quad x.
$$ 
Then, for any metric $g\in \bm $, $v$ is an eigenvector for
the projective Schouten tensor $\P_{ij}$ of $g$. That is 
\begin{equation}\label{Pe}
\P^i{}_j v^j=B v^i \quad \mbox{at} \quad x ,
\end{equation}
for some $B\in \mathbb{R}$, where $P^i{}_j=g^{ik}\P_{kj}$.
\end{proposition}
\begin{proof} Consider the curvature decomposition  \eqref{weylt} as calculated in the scale of the Levi-Civita connection $\nabla$ of $g$. 
 Working at the point $x$, we contract \eqref{weylt} with $v^j$ and
 $v_i$.  The term $v^jv_i{W^i}_{jk\ell}$ from the right hand side
 vanishes because of the assumption of Weyl nullity. On the other hand
 the term $v^j v_i {R^i}_{jk\ell}$ vanishes because of the symmetries
 of the Riemann curvature tensor so we end up with the equality
$$
v_\ell   \, \P_{jk} v^j   = {v_k}\, \P_{j\ell} v^j.
$$ This equality implies that $\P_{j\ell} v^j $ is proportional to
$v_\ell$. So $v$ is an eigenvector of $P^i_ {\ j}$ as claimed (and we
denote the eigenvalue by $B$, so yielding (\ref{Pe})).
\end{proof}
\begin{remark}
From the Proposition (and with the assumptions there) we have that 
$$
Bv_iv^i= \P_{ij}v^iv^j \quad \mbox{at} \quad x .
$$
We can replace $v_x$ by a parallel vector $\hat{v}$ of length $\pm 1$, then 
$B=\P_{ij}\hat{v}^i\hat{v}^j $, at $x\in M$. 
\end{remark}

Weyl nullity is equivalent to an interesting condition on the
curvature, as follows.
\begin{theorem}\label{WeqK}
Let $(M,\bm)$ be a metric projective geometry. Suppose that this
has projective Weyl nullity at a point $x$,
and $v\neq 0$ is a vector in the nullity subspace of $T_xM$. Then, for the 
curvature $R^i_{ \ j k \ell}$ of any metric $g\in \bm $ we have
\begin{equation} \label{Knullity} 
R^i_{ \ j k \ell} v^j = - B\cdot v^j K^i_{\ j k\ell},
\quad \mbox{at} \quad x,
\end{equation} 
where $B\in \mathbb{R}$ is determined by \nn{Pe}
and $K^i_{\ jk\ell } $ is the `constant
curvature tensor'
\begin{equation}
 K^i_{\ jk\ell}:= \delta^i_\ell g_{jk} - \delta^i_k g_{j\ell}. 
\end{equation} 

Conversely, if (\ref{Knullity}) holds, for some $g\in \bm$ and number $B$, then
$v$ is in the Weyl nullity at $x$.
 \end{theorem}
\begin{proof} Suppose that $(M,\bm )$ has a projective Weyl nullity. We 
contract the equation \eqref{weylt} with $v^j$.  The
left hand side of the result vanishes, so we obtain
$$
v^j {R^i}_{jk\ell}= - \left({\delta^{i}_\ell} \, \underbrace{v^j \P_{jk}}_{\eqref{Pe}} -{\delta^{i}_k}\, \underbrace{  v^j \P_{j\ell}}_{\eqref{Pe}}\right) = 
-B\left({\delta^{i}_\ell} \, v_k-{\delta^{i}_k}\, v_\ell\right) = -B \cdot v^j K^i_{\ j k\ell}
$$
as claimed.

For the converse we suppose that 
(\ref{Knullity}) holds for $g\in \bm$. The
result is deduced from \eqref{weylt} in two
steps. Assuming that (\ref{Knullity}) holds, the trace obtained by
contracting with $\delta^{\ell}_i$ shows that $v$ is an eigenvector of
the Schouten tensor as in (\ref{Pe}). Then using this with (\ref{Knullity})
establishes that $v$ is in the Weyl nullity at $x$.   
\end{proof}

\begin{remark} \label{rem:nullity2}
On a pseudo-Riemannian manifold $(M,g)$ the condition \eqref{Knullity} can be equivalently written as
\begin{equation} \label{nullity2} 
Z^i_{\ j k m} v^j = 0, \textrm{ \ where 
$Z^{i}_{\ jkm}:= R^{i}_{\ jkm} + B \cdot K^{i}_{\ jkm}.$} 
\end{equation}
Note that $Z$ has the same algebraic symmetries as the Riemann
curvature tensor. 

As mentioned in the introduction, for the case of
$B$ constant this condition has been studied in the literature under
different names, but to the best of our knowledge the link to Weyl
nullity was not made.
\end{remark}

On a metric projective structure it can be that the Weyl nullity
subspace of $T_x M$ has dimension greater than 1.  We next
observe that the eigenvalue $B$ arising \eqref{Pe} is
independent of the choice of vector $v$ in the nullity subspace. 
\begin{proposition}\label{B!} On a metric projective manifold $(M,\bm)$, let
$g\in \bm$.  Suppose that $u,v$ are two non-zero vectors in the
  Weyl nullity at $x$, then they belong to the same eigenspace of $\P^i{}_j$, at
  $x$, where $\P^i{}_j$ is the projective
  Schouten tensor of $g$.
\end{proposition}
\begin{proof} 
Let $u,v$ be two non-zero vectors in the Weyl nullity at $x$. Then by
Theorem \ref{Knullity}, and the symmetries of $R$ and $K$, we have 
$$
R^i_{ \ j k \ell} v^j = - B_v \cdot v^j K^i_{\ j k\ell},
 \quad \mbox{and} \quad
R^i_{ \ j k \ell} u^k = - B_u \cdot u^k K^i_{\ j k\ell}
$$ 
at $x$. Here $B_u$ indicates the eigenvalue determined by $(g,u)$
according to Proposition \ref{mweyl}, while $B_v$ is the eigenvalue
similarly determined by $(g,v)$.

Now the Riemannian curvature $R_{i j k \ell}$ is alternating on its
first two indices, and $R_{i j k \ell} = R_{k\ell ij}$.  Thus
considering $ v^j u^k R^i_{\ jkm} $ we have:
 $$ v^j u^k R^i_{\ jkm} = -B_v\left( v ^k u_k
\delta^i_m- u^i v_m \right)= -B_u\left( v ^k u_k \delta^i_m- u^i v_m
\right)
$$
implying that $B_u=B_v$, since $u\ne 0\ne v $ and $n\geq 2$. 
\end{proof}

\begin{proposition}\label{Bfn} Suppose that a metric projective structure
$(M,\bm)$ 
has Weyl nullity at all $x\in M$. Then any metric
  $g\in \bm$ determines a canonical function $B:M\to \mathbb{R}$ via
  \nn{Pe}.
Furthermore, the function $B$ is  smooth on any open set $U$ where the dimension of nullity space of the Weyl curvature is constant.  
\end{proposition}
\begin{proof}
The first statement follows immediately from Theorem \ref{mweyl} and
Proposition \ref{B!}. Next since the geometry is smooth, the Weyl
curvature is smooth. Using this, it is easily shown that about each
point $x$ in $U$, there an open neighbourhood $V$ on which there is a
nowhere zero smooth vector field $v$ that is everywhere in the Weyl
nullity. Thus since $\P^a{}_b$ is smooth (on $M$ and hence) on $V$ it
follows from \nn{Pe} that $B$ is smooth on $V$. From this and
Proposition \ref{B!}, we conclude that $B$ is smooth on $U$.
\end{proof}

In the remainder of the article we will say {\em a metric
  projective geometry $(M,\bm)$ has Weyl nullity} to mean that
$(M,\bm)$ has Weyl nullity (in the sense of Proposition \ref{mweyl})
at all $x\in M$. 

It is possible that on a metric projective structure $(M,\bm)$ with
Weyl nullity the field $B$ is necessarily smooth everywhere, in this
case the results of Proposition \ref{Bfn} could be
strengthened. However whether this is true or not is unclear at this
point. At various stages in the following we will investigate the
consequences of having $B$ smooth.

\subsection{The fundamental projectively invariant 2-tensor}\label{phi}

We observe here that if the projective Weyl tensor has nullity then
there is a fundamental 2-tensor that is an invariant of the metric
projective structure.
\begin{theorem}\label{fth}
Suppose that $(M,\bm)$ is a metric projective structure with Weyl
nullity at $x\in M$.  Let $g$, $\bar{g}$ be two metrics in the
equivalence class $\bm$. Write $\P_{ab}$, $\bar{\P}_{ab}$ for the
respective Schouten tensors and $B$, $\bar{B}$ for the respective
$B$-scalars (at $x$). Then
\begin{equation}\label{phidef}
\phi_{ab}(x):=\P_{ab}- Bg_{ab}
= \bar\P_{ab}- \bar{B}\bar{g}_{ab} \quad \mbox{at } x. 
\end{equation}
Thus $\phi_{ab}(x)$ is canonically determined by the metric projective
structure $(M,\bm)$. 

\smallskip

\noindent If $v\in T_xM$ is in the Weyl nullity at $x\in M$, then
$v^a\phi_{ab}=0$ at $x$.

\smallskip

\noindent 
The tensor field $\phi_{ab}$ is smooth on any open set where $B$
is smooth. In particular if the Weyl tensor has constant nullity on an
open set $U$, then $\phi_{ab}$ is a smooth tensor field on $U$.
If the dimension $n=2$ then $\phi=0$, and 
$$
\P_{ab}=Bg_{ab},
$$
with $B$ smooth everywhere.
\end{theorem}
\begin{proof}
We calculate at $x\in M$ and in the scale of $g$. Let $v\neq 0$ be a
vector in the Weyl nullity at $x$. From \nn{weylt}, \nn{Pe}, and
\nn{Knullity}, we have
\begin{equation}\label{vW}
v^\ell W^i{}_{jk\ell}= (\P_{jk}- Bg_{jk})v^i
\end{equation} where indices have been lowered using $g_{jk}$. But the
left-hand-side is independent of the metric from the equivalence class
$\bm$. So we similarly have
$$
v^\ell W^i{}_{jk\ell}= (\bar\P_{jk}- \bar{B}\bar{g}_{jk})v^i .
$$

The next claim follows from \nn{Pe} or by contracting both sides of
\nn{vW} with $v^k$ and using that the projective Weyl tensor is skew
on its last index pair.

The statements in final paragraph of the Theorem now follow
immediately from \nn{vW}, \nn{Pe}, and Proposition \ref{Bfn}.
\end{proof}

\medskip 

Another important invariant is the tensor ${Z^i}_{jk\ell}$
introduced earlier.
\begin{proposition}\label{Zprop} 
Let $(M,\bm)$ be a metric projective geometry of dimension $n\geq 2$
with Weyl nullity. The tensor ${Z^i}_{jk\ell}$ defined in \nn{nullity2}
of Remark \ref{rem:nullity2} is independent of $v$ and is an invariant
of the geometry $(M,\bm)$. It is smooth where $B$ is smooth.
\end{proposition}
\begin{proof}
It is easily verified that 
$$
{Z^i}_{jk\ell}={W^i}_{jk\ell}-\left({\delta^{i}_\ell}
\, \phi_{jk}-{\delta^{i}_k}\, \phi_{j\ell}\right), 
$$
(cf.\ \nn{weylt}) and thus the result is immediate from Theorem \ref{fth}.  
\end{proof}

One consequence of the Theorem \ref{fth} is that as we move between metrics in
the equivalence class $\bm$, $Bg_{ab}$ ``transforms like a Schouten
tensor''. To be more precise we elaborate as follows:
\begin{proposition}\label{Btransp}
Let $(M,\bm)$ be a metric projective structure with projective Weyl nullity. 
Let $\Upsilon_a$ be
the exact 1-form relating the Levi-Civita connections for metrics $g$
and $\bar{g}$ in $\bm$.  Then
\begin{equation} \label{Btrans}
\bar{B}\bar{g}_{ab} = Bg_{ab} - \nabla_a\Upsilon_b +\Upsilon_a\Upsilon_b. 
\end{equation}
\end{proposition}
\begin{proof}
From \nn{pcurvtrans} and the invariance
of $\phi$ \nn{phidef} we have
$$
\P_{ab}- Bg_{ab}
=\bar{\P}_{ab} - \bar{B}\bar{g}_{ab}=\P_{ab}- \nabla_a\Upsilon_b +\Upsilon_a\Upsilon_b - \bar{B}\bar{g}_{ab}.
$$ 
\end{proof}

\begin{remark}
{ Throughout the remainder of the article we will often
  assume a metric projective structure $(M,\bm)$ with Weyl nullity. In
  fact many of the results then obtained hold in the apparently more
  general setting where one does not assume Weyl nullity but just the
  existence on $(M,\bm)$ of an invariant 2-tensor $\phi_{ab}$
  satisfying \nn{phidef} for all metrics in $\bm$  (and for  certain $B$ that depend on the metric). 
 However assuming
  such a structure is certainly ``close to'' assuming Weyl nullity,
  see Theorem \ref{converse} below. }
\end{remark}

\subsection{Invariant differential operators on metric projective structures} 
\label{iops}

We work here on a metric projective manifold $(M,\bm)$ of dimension
$n\geq 2$ with projective Weyl nullity. Furthermore we shall assume,
in this subsection, that the field $B$ is smooth.

We show here that Theorem \ref{fth}, or equivalently, Proposition
\nn{Btransp} leads to new linear differential operators that are
canonically determined by metric projective manifolds with Weyl
nullity.  An important point being that these are not the restriction
to metric projective geometries of projectively invariant linear
differential operators. While it is straightforward to see that there
is a large class of such operators our aim here is to highlight the
idea with two simple but important cases.
\begin{proposition} On $(M,\bm)$ there are canonical invariant linear differential operators 
$$
E_{ab}:\ce(1)\to \ce_{(ab)}(1) \quad \mbox{and} \quad S_{abc}: \ce(2)\to \ce_{(ab)}(2)
$$ with,in a scale $g\in \bm$, $E_{ab} (\si)$
given by
$$
\nabla_a \nabla_b \si +  Bg_{ab}\si
$$
and $S_{abc} (\tau ) $  given by
\begin{equation}\label{Sform}
 \nabla_{(a}\nabla_{b}\nabla_{c)}\tau + 4B g_{(ab}\nabla_{c)}\tau+ 2g_{(ab}\big(\nabla_{c)}B  \big)\tau .
\end{equation}
\end{proposition}
\begin{proof}
There are canonical sequences of invariant linear differential
operators on projective manifolds (and more generally parabolic
geometries) known as BGG sequences, see e.g.\ \cite{BEbook,CSSann,CD}. The
first operators in such sequences are often called first BGG operators
and form an important class of invariant overdetermined
operators. Among the most well known (see e.g. \cite[Section
  3]{CGHjlms}) in this class are operators on $\ce(1)$ and $\ce(2)$
given (in a scale) respectively by
\begin{equation}\label{two}
\si\mapsto \nabla_a \nabla_b \si + \P_{ab} \si
\end{equation}
 and 
\begin{equation}\label{three}
\tau \mapsto \nabla_{(a}\nabla_{b}\nabla_{c)}\tau + 4\P_{(ab}\nabla_{c)}\tau+ 2\big(\nabla_{(a}\P_{bc)}\big)\tau .
\end{equation}
These are are canonically determined on any projective manifold, and
so also invariant and canonical upon restriction to metric projective
manifolds $(M,\bm)$.

The invariance of $E_{ab}$ follows at once from the formula \nn{two}
and that, according to Theorem \ref{fth}, $ \phi_{ab}$ is invariant
and so $\nabla_a \nabla_b \si + \si (\P_{ab} - \phi_{ab})$ is
invariant on $(M,\bm)$.

Similarly the invariance of $S_{abc}$ is immediate from the formulae \nn{three} and the fact that for $\tau\in \Gamma \ce(2)$ 
$$
\tau \nabla_{(a}\phi_{bc)}+2 \phi_{(ab}\nabla_{c)}\tau
$$
is invariant on $(M,\bm)$. The latter  follows easily from the transformation formulae \nn{ptrans}:
If $\bar{g}$ and $g$ are metrics in $\bm$, and $\bar\nabla$ and $\nabla$ denote their respective Levi-Civita connections,  then
$$
\bar\nabla_a\phi_{bc}=\nabla_a\phi_{bc}-2\Upsilon_a\phi_{bc}-\Upsilon_b\phi_{ac}-\Upsilon_c\phi_{ba}
$$
for some exact 1-form $\Upsilon$. On the other hand 
$$
\bar\nabla_a \tau = \nabla_a\tau+2 \Upsilon_a \tau.
$$
\end{proof}

\begin{remark} On projective manifolds $(M,\bp)$ 
the equations \nn{two} and \nn{three} have important geometric
interpretations linked to the Einstein equations,
\cite{CapGoCrelle,CGHjlms,CGMac}.  For example a nowhere zero solution of
\nn{two} is equivalent to the existence of a Ricci-flat affine
connection in the projective class $\bp$. Similarly a special class of
solutions to \nn{three} (solutions which are normal and suitably
non-degenerate) is in 1-1 correspondence with non-Ricci flat Einstein metrics
with Levi-Civita in the projective class $\bp$.
\end{remark}

All of the invariant linear differential operators on projective
manifolds $(M,\bp)$ (between irreducible weighted tensor bundles) can
be given in a scale by universal formulae involving only the affine
connection $\nabla$ of the scale and the corresponding Schouten tensor
and its $\nabla $ derivatives \cite{CDS}. In fact, as pointed out in
\cite{CDS}, essentially the same formulae govern a huge class of
so-called {\em standard} operators on other parabolic differential
geometries; the formulae were first found in conformal geometry \cite{Gqjm}.
It seems likely that for each projectively invariant linear
differential operator between irreducible weighted tensor bundles
there is on metric projective structures, with Weyl nullity and $B$
smooth, a corresponding invariant operator constructed using only a
metric $g\in \bm$, its Levi-Civita connection $\nabla$ and the field
$B$ determined by $g$.
  
\medskip

Finally here we note that the operator $S_{abc}$ is nicely linked to
the {\em Gallot-Obata-Tanno equation} which, on a pseudo-Riemannian manifold
$(M,g)$, may be written in the form
\begin{equation}\label{one}
\nabla_{a}\nabla_{b}\nabla_{c} f + B_\circ \cdot( 2g_{bc}\nabla_{a} f +g_{ab}\nabla_{c} f + g_{ca}\nabla_{b} f )=0 ,
\end{equation}
where $B_\circ$ is constant, $f$ is a function and $\nabla$ is the Levi-Civita connection for $g$.  
If this equation holds then  
$$
 0= [\nabla_{a},\nabla_{b}]\nabla_c f + B_\circ( \tau g_{cb}\nabla_a f-g_{ca}\nabla_b f)=-R^{d}{}_{cab}\nabla_d f+B_\circ( g_{cb}\nabla_a f-g_{ca}\nabla_b f)   ,
$$ 
as follows by skewing \nn{one} over the index pair ``$ab$''. (Here $[\cdot ,\cdot]$ is the commutator bracket.)
So 
\begin{equation}\label{in}
R^{d}{}_{cab}\nabla_d f= B_\circ( g_{cb}\nabla_a f-g_{ca}\nabla_b f)
\end{equation}
which means that $df$ is in the projective Weyl nullity, according to
Theorem \ref{WeqK}. Even more simply a function $ f$ satisfying
\nn{one} obviously also satisfies \nn{Sform} with $B$ set to the same
constant $B_\circ$.  In the converse direction if, on a pseudo-Riemannian
manifold $(M,g)$, a function $f$ satisfies $S^g_{abc}(f)=0$ and that
$g^{ab}\nabla_a f$ it is in the Weyl nullity then, if $B$ is constant,
$f$ also satisfies the Gallot-Obata-Tanno equation. More than this we see
that the nullity condition and the equation of $S^g_{abc}$, with $B$ constant and compatible, are each equivalent to
irreducible parts of the Gallot-Obata-Tanno equation. Note that the metric
determines a volume density and hence, by taking a root of this
(noting the bundle of volume densities is oriented), a canonical
non-zero section of $\ce(2)$ that is parallel for the Levi-Civita
connection (cf.\ the discussion below surrounding \nn{taudef}). 
Thus by multiplying or dividing by this we see that $f$ is
canonically related to an equivalent projective density of weight 2
that we might denote $\tau^f$.  In summary:
\newcommand{\grad}{\operatorname{grad}}
\begin{theorem}\label{G-Obthm}
Let $(M,g)$ be  a pseudo-Riemannian manifold $(M,g)$ of dimension $n\geq 2$. Then a function  $f$ is
a solution of the Gallot-Obata-Tanno equation \nn{one} if and only if: (i) $\grad f$ is
in the projective Weyl nullity; (ii) $B^g$ is constant with $B^g=B_\circ$
on any open set where $\grad f$ is non-zero; and (iii) $\tau^f$ satisfies
\nn{Sform} with $B:=B_\circ $. 
\end{theorem}

The equation \eqref{one} appeared and has been studied in different a
priori unrelated branches of differential geometry.  The motivation of
Gallot and Tanno 
 to study this equation came from the spectral
geometry: it is well-known (see for example \cite{Ga}) that, on the
standard sphere $S^n \subset \mathbb{R}^{n+1}$ of dimension $n>1$, all
eigenfunctions corresponding to the second biggest eigenvalue (namely $-n$) of
the Laplacian satisfy the equation
\begin{equation} \label{obata1}
\nabla_a\nabla_b f  +   g_{ab} f  =0. 
\end{equation} 
 The
eigenfunctions corresponding to the third biggest eigenvalue $-2(n+1)$
satisfy \eqref{one} with $B_\circ=1$. 

Obata has shown \cite[Theorem A]{obata1} that, on closed Riemannian
manifolds, the existence of a nonconstant solution of \eqref{obata1}
implies that the metric has constant curvature $1$.  Later, he
\cite{Obata}, and, according to Gallot \cite{Ga}, Lichnerowicz, asked
the question whether the same holds for the equation \eqref{one}
(assuming $c=1$). The affirmative answer was given in \cite{Ga,Ta}.

The equation \eqref{one} naturally appears in the study of the
geometry of the metric cones, see Gallot \cite{Ga} or Alekseevsky et
al \cite{Leist}.
This equation also appears in the context of projective
equivalence. In particular, Tanno has shown that, for any solution
$f $, the vector field $\grad f$ is a non-trivial projective
vector field provided $B\ne 0$.

\subsection{Einstein and related conditions}
The Einstein condition was defined in Section \ref{conf}. We
investigate here some consequences of the incidence of this with
projective Weyl nullity (cf.\ \cite{CGMac,GMac} where different aspects are treated). First we observe that for Einstein metrics
this incidence is not restrictive in the lowest dimensions. If $n=2$
then, in any case, the Weyl curvature is zero. While for next
dimension we have the following.
\begin{proposition}\label{3spec}
 Let $(M,\bm)$ be a metric projective structure of dimension 3. If
 $g\in \bm $ is Einstein then $(M,\bm)$ is projectively flat,
 i.e.\ $W^a{}_{bcd}=0$.
\end{proposition}
\begin{proof}
This is immediate from Corollary \ref{Eins}, as the conformal Weyl tensor 
$C^a{}_{bcd}$ is identically zero on 3-manifolds. 
\end{proof}

In the following $\Sc^g:=g^{ij}R_{ij}$ is the scalar curvature
determined by a metric $g\in \bm$.
\begin{proposition}\label{ein1} 
Suppose that $(M,\bm)$ is a metric projective structure with Weyl
nullity at $x\in M$. If $g\in \bm $ is Einstein then 
$$
B^g=\frac{1}{n}g^{ij}\P^g_{ij}=\frac{1}{2n(n-1)}\Sc^{g}\quad \mbox{at}\quad x.
$$
\end{proposition}
\begin{proof}
This is immediate from \nn{Pe} and definition of Einstein. 
\end{proof}

Given a metric $g$ in the projective class $\bm$,  let us write $\J:=g^{ab} \P_{ab}$. Note that contracting
\nn{phidef} with $g^{ab}$ gives
\begin{equation}\label{JB}
g^{ab}\phi_{ab}= \J-n B .
\end{equation}
So, in general, the metric trace of $\phi$ measures the failure of ($n \times$) $B$ to agree
with the metric trace of Schouten.  
Now
$$
\phi_{ab}= \P_{ab}-Bg_{ab}= \mathring{\P}_{ab}+\frac{1}{n}\J g_{ab}-Bg_{ab}
= \mathring{\P}_{ab}+ \frac{1}{n} (g^{cd}\phi_{cd})g_{ab},
$$
and so 
\begin{equation}\label{tfphi}
\mathring{\phi}_{ab}= \mathring{\P}_{ab} , 
\end{equation}
or equivalently 
\begin{equation}\label{tfphi2}
\mathring{\phi}^{a}{}_b = \frac{1}{n} W^{a}{}_{cbd}g^{cd} .
\end{equation}

There are some obvious consequences of Theorem \ref{fth}. First it
provides an easy route to the well known Beltrami Theorem:
\begin{corollary}\label{Bel}
Let $M$ be a smooth manifold of dimension $n\geq 2$ and let $g$ be a
pseudo-Riemannian metric on $M$ such that the projective structure
determined by $g$ is locally projectively flat. Then g has constant
sectional curvature.
\end{corollary}
\begin{proof}
Since $W^i{}_{jk\ell}=0$ on $M$ we have (Weyl nullity and) that
$\phi_{jk}=0$, equivalently $\P_{jk}=Bg_{jk}$, from \nn{vW}. So, if
$n\geq 3$,  $g$ is Einstein and $B$ is constant. In dimension 2, 
\nn{vW} implies $\P_{jk}- Bg_{jk}$, so  the
vanishing of the Cotton tensor is equivalent to  $B$ constant.
\end{proof}

Partly generalising this, we have the following:
\begin{corollary}\label{1ein2}
Let $(M,\bm)$ be a smooth metric projective geometry with Weyl nullity. 
If the dimension of $M$ satisfies 
$n\geq 3$ then there is an Einstein metric
$g\in \bm$ if and only if $\phi_{ab}=0$. If  $n\geq 2$
and $g\in
\bm$ is Einstein, then any metric $\bar{g}\in \bm$ is  Einstein.
\end{corollary}
\begin{proof}
Suppose that $n\geq 3$ and
$g\in \bm$ is Einstein. Using \nn{tfphi} we see that
$\mathring{\phi}_{ab}=0$. On the other
hand from \nn{JB} and Proposition \nn{ein1} we see that 
the trace part of $\phi$ is also zero, $g^{ab}\phi_{ab}=0$. So if 
$g\in \bm$ is Einstein then $\phi_{ab}=0$. 

For the converse, observe that if $\phi_{ab}=0$ then for any $g\in
\bm$ we have $\mathring{\P}_{ab}=0$ from \nn{tfphi}. Thus if $n\geq 3$
we have immediately that $g$ is Einstein, and this also proves the
last statement.

It remains to treat the last statement in the case that $n=2$.  Assume
that $n=2$ and $g$ is Einstein. Then by Theorem \nn{fth}, and our definition of Einstein in
dimension 2, the corresponding $B$ is constant. Since $\P_{ab}=Bg_{ab}$
the Cotton tensor vanishes.  But the Cotton tensor is independent of
the metric in $\bm$. Calculating in the scale of any other metric
$\bar{g}\in \bm$ we see that Cotton zero and
$\bar{\P}_{ab}=\bar{B}\bar{g}_{ab}$ implies that $\bar{B}$ is
constant.
  \end{proof}
\begin{remark}
In dimensions 2 and 3 the last statement of Corollary \ref{1ein2}
holds with the explicit assumption of Weyl nullity. In dimension 2
this is obvious, while for dimension 3 it follows from Proposition
\ref{3spec}. In either case we are then in the setting of Corollary \ref{Bel}.
\end{remark}

In fact, on any metric projective manifold $(M,\bm)$, if $\bm$
contains two non-affinely equivalent metrics $g$ and $\bar g$ and one
is Einstein then $(M,\bm)$ has nullity in the sense of \nn{Knullity},
see \cite{einstein}, and hence Weyl nullity by Theorem
\ref{WeqK}. Thus from Corollary \ref{1ein2} we recover a simpler proof of the following result from \cite{einstein}:
\begin{proposition}\label{1ein22}
Let $(M,\bm)$ be a smooth metric projective geometry of dimension
$n\geq 2$ and suppose that $g,\bar{g}\in \bm$ with $g$ Einstein and
$\bar{g}$ not affinely equivalent to $g$. Then any metric in $ \bm$ is
Einstein.

\end{proposition}

\begin{remark}
Let $(M,\bm)$ be a metric projective structure of dimension 4 with
projective Weyl nullity everywhere. If there is a Riemannian signature
Einstein metric $g\in \bm$ then the structure is projectively
flat (and so again we are in the setting of Corollary \ref{Bel}). 
This result arises as follows. If $v$ is in the projective Weyl
nullity at $x$ then from Corollary \ref{Eins} we have
$$
C^a{}_{bcd}v^b=0,
$$ 
But in dimension 4 we have $4C^{abcd}C_{ebcd}=\delta^a_e |C|^2$, where
$|C|^2:=C^{abcd}C_{abcd}$.

We can drop the Riemannian signature requirement if we insist that
the nullity vector field is almost everywhere non-null, since there are
no non-trivial algebraic Weyl tensors in dimension 3.
\end{remark}

\subsection{Tractor connections on metric projective 
structures with nullity}\label{newcon} 
On a metric projective
structure with Weyl nullity there is a family of canonical tractor
connections parametrised by $\ttt\in \mathbb{R}$. We see this as
follows. In this subsection we assume that the field $B$ is smooth.

Suppose that any projective manifold is equipped with a fixed smooth
$(0,2)$-tensor field $\phi_{ab}$. Then we canonically obtain a
corresponding projectively invariant 1-form taking values in the
bundle of tractor endomorphisms $\End{\cT}$ by forming
\begin{equation}\label{Phi}
\Phi^C{}_{Ba}:= X^C Z_B{}^b\phi_{ab} .
\end{equation}
Thus for each $\ttt\in \mathbb{R}$ we may modify the tractor
connection $\nabla^\cT$ to
$$
\nabla^\cT+ \ttt \Phi . 
$$ This notation means that, for a tractor field $V^C$, its covariant
derivative by this connection is
$$
\nabla^\cT_a V^C + \ttt V^B \Phi^C{}_{Ba}. 
$$

Thus by Theorem \ref{fth} a metric projective manifold admitting Weyl nullity 
has a family of such connections. We summarise:
\begin{proposition}\label{alphaf}
Let $(M,\bm)$ be a smooth metric projective manifold of dimension
$n\geq 2$ with Weyl nullity and $B$ smooth. Then there is a 1-parameter family of canonical tractor connections 
$$
\nabla^{\cT_\ttt}:=\nabla^\cT+ \ttt \Phi 
$$  via \nn{Phi}, where $\phi_{ab}$ is the fundamental 2-tensor of
Theorem \ref{fth}.
\end{proposition}

For reasons that will shortly be clear, we are especially interested in
the case that $\ttt=1$ is chosen. In this case the explicit
appearance of $\P$ is replaced altogether in the tractor connection:
\begin{equation}\label{tconn1}
\nabla^{\cT_1}_a \left( \begin{array}{c} \nu^b\\
\rho
\end{array}\right) \stackrel{g}{=}
\left( \begin{array}{c} \nabla_a\nu^b + \rho \delta^b_a\\
\nabla_a \rho - Bg_{ab}\nu^b
\end{array}\right),
\end{equation}
where $g\in \bm$.

 On $S^2\mathcal{T}$ this tractor connection is given by
\begin{equation}\label{s2conn}
\nabla_a^{\cT_1}\left( \begin{array}{c}
\si^{bc}\\
\mu^b\\
\rho
\end{array}
\right)\stackrel{g}{=}
\left( \begin{array}{c}
\nabla_a \si^{bc} + \delta^b_a \mu^c + \delta^c_a \mu^b \\
\nabla_a \mu^b + \delta^b_a\rho - Bg_{ac}\si^{bc}  \\
\nabla_a \rho - 2 Bg_{ab}\mu^b
\end{array}
\right).
\end{equation}
which will be useful for our later developments.

\begin{proposition}\label{splitB}
Let $(M,\bm)$ be a smooth metric projective manifold of dimension
$n\geq 2$ with Weyl nullity and $B$ smooth. Then there is an invariant differential splitting operator
$$
L^\phi: S^2TM(-2)\to S^2\cT,
$$
given by 
\begin{equation}\label{Lphi-op}
\si^{ab}\mapsto L^\phi(\si):\stackrel{g}{=}
\left( \begin{array}{c}
\si^{bc}\\
-\frac{1}{n+1}\nabla_b\si^{ba}\\
\frac{1}{n}(\nabla_b\nabla_c\si^{bc}+ B g_{bc}\si^{bc})
\end{array}
\right),
\end{equation}
in any scale $g\in \bm$. Any section of $ S^2\cT$ that, on an open set $U$, is parallel for $\nabla^{\cT_1}$ is in the image of $L^\phi$.
\end{proposition}
\begin{proof}
In the fixed scale $g$ the last statement is easily verified. Then the invariance of $L^\phi$ follows from that of the connection.
\end{proof}

\begin{remark}
Let $(M,\bm)$ be a metric projective manifold of dimension $n\geq 2$ and let us calculate with respect to some $g\in \bm$. 
Suppose that if $\nu^b$ is a solution of the projectively invariant
system \nn{psys1}, so $V^B= (\nu^b,~\rho)$ is parallel for the normal
tractor connection. Then by Proposition \ref{psysth} $\nu^b$
is in the Weyl nullity. Thus $\P_{ab}\nu^b=B\nu_a.$ So $V^B$ is parallel also for the connection $\nabla^{\cT_{1}} $. (In fact $V^B$ is parallel for $\nabla^{\cT_{t}} $ for any $t\in \mathbb{R}$, as $\nu^b\phi_{ab}=0$.)

Conversely if $V^B= (\nu^b,~\rho)$ is a section parallel for a
connection of the form $\nabla^{\cT_{1}}$ then there is Weyl nullity
everywhere, $\nu^b$ is in the Weyl nullity and $V^B$ is parallel for
the normal tractor connection. 
\end{remark}

Next, consider $\nabla^{\cT_1}$ on the symmetric power of the dual
tractor bundle $S^2\cT^*$. 
In a scale $g$ this is given by
\begin{equation}\label{s2-dual-conn}
\nabla_a^{\cT_1}\left( \begin{array}{c}
\tau\\
\mu_b\\
\rho_{bc}
\end{array}
\right)\stackrel{g}{=}
\begin{pmatrix}
  \nabla_a\tau-2\mu_a  \\ \nabla_a\mu_b+ Bg_{ab}\tau-\rho_{ab} \\ 
\nabla_a\rho_{bc}+2Bg_{a(b}\mu_{c)}
\end{pmatrix},
\end{equation}
from which an analogue of Proposition \ref{splitB} is evident. In particular,
calculating with respect to a metric $g\in \bm$: Any section $(\tau,
\mu_c, \rho_{bc})$ of $S^2\cT^*$ that is parallel for $\nabla^{\cT_1}$
has $\mu_c=\frac{1}{2}\nabla_c \tau$, and
$\rho_{bc}=\frac{1}{2}\nabla_b\nabla_c \tau + B g_{bc}\tau $. Furthermore
the differential operator $\bar{L}^\phi:\ce(2)\to S^2\cT^*$ given (in the scale $g$) by 
\begin{equation}\label{Lbar}
\Gamma(\ce(2))\ni\tau\mapsto
(\tau,\frac{1}{2}\nabla_c \tau, \frac{1}{2}\nabla_b\nabla_c \tau + B
g_{bc}\tau) \in \Gamma(S^2\cT^*)
\end{equation} 
is metric projectively invariant. Then using the explicit formula \nn{s2-dual-conn}, and the
 the map between
functions $f$ on pseudo-Riemannian manifolds and corresponding
projective densities $\tau =\tau^f\in\Gamma(\ce(2)) $, as described for
Theorem \ref{G-Obthm}, one easily verifies the following result.
\begin{theorem}\label{GOTthm}
{ On a pseudo-Riemannian manifold $(M,g)$ solutions of
  the Gallot-Obata-Tanno equation \nn{one} are equivalent to solutions
  of the system on the right-hand-side of \nn{s2-dual-conn} with
  $B=B_\circ$ constant. On any open set where $\nabla_a\tau$ is nowhere zero, it lies in the Weyl nullity and $B^g=B_\circ$.}

{ In particular, on a  pseudo-Riemannian manifold $(M,g)$ 
with Weyl nullity and $B^g$ constant,}
solutions of the Gallot-Obata-Tanno equation \nn{one} are in
one-to-one correspondence with sections of $S^2\cT^*$ that
are parallel for $\nabla^{\cT_1}$. 
\end{theorem}
\begin{proof} 
The first statement is just the observation that the right-hand-side
of \nn{s2-dual-conn} coincides with equation \nn{one} rewritten as a
linear first order system. Then the statement that $\nabla_a\tau$
annihilates the Weyl curvature follows from Theorem \ref{G-Obthm}. The final result uses the discussion above the Theorem here.
\end{proof}

In particular we have the standard first consequence of such results: 
\begin{corollary}\label{GOTcor}
If $f$ is a non-zero solution of the Gallot-Obata-Tanno equation
\nn{one} then $f$ is non-zero on an open dense set. 
\end{corollary}
\begin{proof}
Recall we assume $M$ is connected. Suppose $f$ is a solution of the
Gallot-Obata-Tanno equation. Let $\tau=\tau^f$ be the section of
$\ce(2)$ corresponding to $f$, as discussed above Theorem
\ref{G-Obthm}.  From the splitting operator $\tau\mapsto
\bar{L}(\tau):=(\tau,\frac{1}{2}\nabla_c \tau,
\frac{1}{2}\nabla_b\nabla_c \tau + B_\circ g_{bc}\tau)$ it follows
that if $\tau=0$ in an open neighbourhood then $L(\tau)$=0 on the same
neighbourhood. Since $\bar{L}(\tau)$ is parallel, for the connection
given by the right-hand-side of \nn{s2-dual-conn} with $B=B_\circ$, it
follows that it is zero everywhere, and hence $\tau =0$ (and so also
$f=0$) everywhere on $M$.
\end{proof}

By a very similar argument we can also show the following stronger result: 
\begin{theorem} \label{Tae}
 Suppose a function $f$ satisfies the Gallot-Obata-Tanno equation
 \nn{one}.  If the differential of the function $f$ is not zero at a
 point, then it is not zero at each point of a certain everywhere
 dense open subset of $M$. Thus a pseudo-Riemannian manifold $(M,g)$ with a
 non-constant solution of the equation \nn{one} has Weyl nullity 
and $B^g=B_\circ$ (and so, in particular, $B^g$ constant) everywhere.
\end{theorem}  
\begin{proof}
Recall that we assume $M$ connected.
Consider the equation
\begin{equation}\label{teq}
\nabla_{(a}\nabla_b u_{c)}+4B_\circ g_{(ab} u_{c)}=0 ,
\end{equation}
on a 1-form field $u_b$. This is an overdetermined linear homogeneous
geometric PDE of finite type. Thus by the general prolongation theory
in \cite{BCEG}, solutions correspond to parallel sections for a linear
tractor-type connection and the solutions, if not trivial, can only
vanish on a closed nowhere dense set.  This implies the first claim
immediately, as we may view the view the Gallot-Obata-Tanno equation
\nn{one} as the combined system consisting of $u_a=\nabla_a f$, \nn{teq}, and the nullity
equation $R^{d}{}_{cab}u_d= B_\circ( g_{cb}u_a-g_{ca}u_b)$ from
\nn{in}. The second claim then follows from Theorem \ref{G-Obthm} and
continuity.
\end{proof}

\begin{remark} 
Concerning the Theorem \ref{Tae} here, note that there is an simpler
argument in the case that $B_\circ=0$.  Indeed, for $B_\circ=0$ the
Gallot-Tanno-Obata equation can be rewritten as
$\nabla_a\bigl(\nabla_b\nabla_c f\bigr)= 0$, so the (0,2) tensor field
$\nabla_b\nabla_c f$ is parallel. If it is not zero at a point, then
it is not zero at every point and we are done. If it zero everywhere,
then $\nabla_b\bigl(\nabla_c f\bigr)$ is parallel, so if $df=\nabla_c
f$ is not zero at a point, then it is not zero at every point.
\end{remark}

\section{The prolonged system for a second metric}\label{prol}

Here we first review the prolonged system corresponding to the
existence of a Levi-Civita connection in the projective class. Then we
find the simplifications  that are available when we restrict
to the metric projective setting with nullity. This reveals a nice
link with the connection $\nabla^{\cT_1}$ found earlier.

\subsection{The prolonged system for the metrisability equation} 
We work first in the setting of a general projective manifold $(M,\bp)$ and 
let $\nabla\in \bp$.
Consider the differential operator
$$
D_a: \ce^{(bc)}(-2) \to (\ce_a{}^{(bc)})_0(-2), \quad \mbox{given by
}\quad \si^{bc}\mapsto \operatorname{trace-free}\left( \nabla_a
  \si^{bc}\right).
$$ 
It is an easy exercise to verify that $D$ is a projectively
invariant differential operator, meaning  that it is independent of the
choice $\nabla\in \bp$.  Part of the importance of $D$ derives from the
following result due to Sinjukov \cite{Sinjukov}.
\begin{theorem}\label{EM1}
Suppose that $n\geq 2$ and $\nabla$ is a special torsion-free
connection on $M$. Then $\nabla$ is projectively equivalent to a
Levi-Civita connection if and only if there is a non-degenerate
solution $\si$ to the equation
\begin{equation}\label{metr}
D\si=0.
\end{equation}
\end{theorem}
\noindent Here $\si$ {\em non-degenerate} means that it is
non-degenerate as a bilinear form on $T^*M (1)$. Our presentation of
the Theorem here follows the treatment \cite{EM}. In view of the Theorem we shall call \nn{metr} the {\em metrisability equation}; note that the trace terms can be included
into a new variable $\mu^a\in \Gamma(\ce^a)$ and so this equation can
be written
\begin{equation}\label{ef}
\nabla_a \si^{bc} + \delta^b_a \mu^c + \delta^c_a \mu^b =0.
\end{equation}

To simplify the discussion we assume in this section that $M$ is oriented.
Let us write $\vvol_{a_1a_2\cdots a_n}$ for the canonical section of
$\Lambda^nT^*M(n+1)$ which gives the tautological bundle map
$\Lambda^nTM\to \ce(n+1)$.  Observe that each section $\si^{ab}$ in
$\ce^{(ab)}(-2)$ canonically determines a section $\tau^\si\in
\ce(2)$, by taking its determinant using $\vvol$:
\begin{equation} \label{taudef}
\si^{ab}\mapsto \tau^\si:=\si^{a_1b_1}\cdots \si^{a_nb_n}\vvol_{a_1\cdots a_n}\vvol_{b_1\cdots b_n}
\end{equation}
For simplicity in the following we fix $\si$ and write simply $\tau=\tau^\si$.
We may form
\begin{equation} \label{invm}
\tau \si^{ab}
\end{equation}
and in the case that $\si^{ab}$ is non-degenerate taking the inverse
of this yields a metric that we shall denote $g^\si_{ab}$. This
construction is clearly invertible and a metric $g_{ab}$ determines a
non-degenerate section $\si^{ab}\in \ce^{(ab)}(-2)$.  We are
interested in the metric $g^\si$ when $\si $ is a solution to
\nn{metr}. Indeed, the Levi-Civita connection mentioned in the Theorem
is the Levi-Civita connection for $g^\si$.

By differentiating the equation \nn{metr} and computing the
consequences of solutions we find that solutions to \nn{metr} prolong
to distinguished sections of $S^2\cT$ as summarised in the following
theorem of \cite{EM} (given here with the conventions of Section
\ref{trS} above).
\begin{theorem} \label{EMthm}
The solutions to \nn{metr} are in one-to-one correspondence with solutions
of the following system:
\begin{equation}\label{psys}
\nabla_a\left( \begin{array}{c}
\si^{bc}\\
\mu^b\\
\rho
\end{array}
\right) + \frac{1}{n}
\left( \begin{array}{c}
0 \\
W^b{}_{dac}\si^{cd}  \\
- 2 C_{cab} \si^{bc}
\end{array}
\right) =0.
\end{equation}
\end{theorem}
\noindent Note the left-hand side of \nn{psys} may be considered as
the formula for a connection on $S^2\cT$. For convenience we shall
call this the {\em prolongation connection}.

For a solution $\si^{bc}$ of the metrisability equation \nn{eq1} $\nabla_a
\si^{bc} + \delta^b_a \mu^c + \delta^c_a \mu^b =0$ we have that the
variable $\mu^a$ satisfies $\mu^a=-\frac{1}{n+1}\nabla_b\si^{ba}$ and 
from \nn{s2conn-orig} 
$n \rho= \si^{bc}\P_{bc}-\nabla_a\mu^a$, since $W^a{}_{bcd}$ is trace-free. 
These formulae determine a differential splitting operator 
\begin{equation}\label{Lop}
\si^{ab}\mapsto L(\si):= 
\left( \begin{array}{c}
\si^{bc}\\
-\frac{1}{n+1}\nabla_b\si^{ba}\\
\frac{1}{n}(\nabla_b\nabla_c\si^{bc}+ \P_{bc}\si^{bc})
\end{array}
\right),
\end{equation}
and, upon restriction to solutions, this is the 1-1 mapping taking
solutions of \nn{metr} to tractors satisfying \nn{psys}.  By standard
theory (see \cite{CGMac}), and it is easily verified directly, this
{\em differential splitting operator} is projectively invariant as a
linear operator $L: \ce^{(ab)}(-2)\to \ce^{(AB)}$.  Using this
 we have an immediate corollary of Theorem
\ref{EMthm}: 
\begin{corollary}\label{spott}
If $\si$ is a non-trivial solution of \nn{metr} then $L(\si)$ is
nowhere zero, and in particular $\si$ is non-zero on an open dense
set.
\end{corollary}
\noindent Note the contrapositive statement, to that here, is also useful:
 If $\si$ is a solution of the metrisability equation \nn{metr} such that
$L(\si)=0$ at some $x\in M$ then $\si$ is zero everywhere.

\begin{remark}\label{eintr}
It is natural to ask what is the meaning of the system \nn{psys} if
the second term is omitted; that is if tractor field
$(\si^{bc},~\mu^b,~\rho)$ (in $S^2\cT$) is required to be parallel for
the normal tractor connection \nn{s2conn-orig}. By definition,
$\si^{bc}$ is then a normal solution of the metrisability equation.
This is treated in \cite{CGMac}. 
It is shown there that $\si^{bc}$, if non-degenerate, is equivalent to
an Einstein metric. Furthermore the converse is also true. See also
\cite{GMac} where this equivalence with the Einstein condition 
 is derived in different way, and
\cite{Armstrong} where a slightly weaker result was given.
\end{remark}

\subsection{Metric projective structures, prolongation and tractor connections}\label{mpsub}
Here we first show that, in the case of nullity 
on a metric projective structure, \nn{metr} implies the following:
\begin{proposition}\label{altmid} Let $\sigma$ be a solution of \nn{metr} 
on a metric projective manifold $(M,\bm)$ with projective Weyl
nullity.  Then, in the notation above, and in a scale of $g\in \bm$,
we have
$$ 
\operatorname{trace-free}(\nabla_a\mu^b-\si^{bc}g_{ac}B) =0 .
$$
\end{proposition}
\begin{proof} We fix a choice of $g\in\bm$ and calculate in that scale; so 
$\nabla$ denotes  the Levi-Civita connection of $g$.
From Theorem \ref{EMthm} and the formula \nn{s2conn-orig} for the
normal tractor connection we have
$$
\nabla_a\mu^b+\delta^b_a\rho-\tau^{-1}\bar{g}^{bc}\P_{ac} +\frac{1}{n}\tau^{-1}W^b{}_{dac}\bar{g}^{cd} =0,
$$ 
where we have used $\bar{g}^{ab}:=\tau\si^{ab}$. So using \nn{tfphi2} we have 
$$
\nabla_a\mu^b+\delta^b_a\rho-\tau^{-1}\bar{g}^{bc}(\P_{ac} -\mathring{\bar{\phi}}_{ac}) =0,
$$ where $\mathring{\bar{\phi}}_{ac}:=
\phi_{ac}-\frac{1}{n}\bar{g}_{ac}\bar{g}^{bd}\phi_{bd}$. (Note that
$\phi=\bar\phi$.)  Using the last display with the identity
$\P_{ac}-\phi_{ac}=Bg_{ac}$, of \nn{phidef}, we have
\begin{equation}\label{newmid}
\nabla_a\mu^b-\si^{bc}g_{ac}B + \delta^b_a(\rho- 
\frac{1}{n}\si^{cd} \phi_{cd})=0
\end{equation}
as required.
\end{proof}

We now have an immediate consequence of this, using also Corollary \ref{spott}.
\begin{corollary}\label{Bs}
Let $(M,\bm)$ be a metric projective manifold with projective Weyl
nullity.  If $g,\bar{g}\in \bm$, and these two metrics are such that
at $x\in M$ they are non-proportional, then there is no open set of
$M$ on which $g$ and $\bar{g}$ agree up to constant dilation.
Furthermore there is an open dense set $U\subseteq M$ on which $B^g$ (and hence
$B^{g'}$ for all $g'\in \bm$) is smooth.
\end{corollary}
\begin{proof}
The metrics $g$ and $\bar{g}$ each determine solutions, respectively
$\tilde{\si}$ and ${\si}$, of the metrisability equation \nn{metr}. If the
metrics agree up to constant dilation on any set then there is a
constant $c$ such that the difference $\tilde{\si}-c\cdot {\si}$
vanishes on the same set. In particular if the set is open then, by
Corollary \ref{spott}, it must be that $\tilde{\si}-c\cdot {\si}$
vanishes everywhere and so $g=\bar{c}\bar{g}$, everywhere on $M$, for
some constant $\bar{c}$.

The final statement then follows from \nn{newmid} since the trace-free
part of $\si^{bc}g_{ac}$ is smooth and vanishes on an open set if and
only if $\si^{bc}$ is conformal to $g^{bc}$. But, by a classical result of Weyl \cite{Weyl}, see also  \cite[Lemma 4]{pseudosymmetric},  on an open set,
conformally related metrics can only lie in the same projective class
$\bm$ only if they are related by constant dilation.
\end{proof}

\begin{remark}\label{gradient}
For a solution $\si^{bc}$ of the metrisability equation \nn{eq1} $\nabla_a
\si^{bc} + \delta^b_a \mu^c + \delta^c_a \mu^b =0$ we have that the
variable $\mu^a$ satisfies $\mu^a=-\frac{1}{n+1}\nabla_b\si^{ba}$. But
on a metric projective manifold $(M,\bm)$ and calculating in the
scale of a metric $g\in \bm$ (so $\nabla$ is the Levi-Civita connection for
$g$) we have
\begin{equation}\label{mug}
\mu_a:=g_{ab}\mu^b =-\frac{1}{2}\nabla_a (g_{bc}\si^{bc}), 
\end{equation}
from \nn{ef}, and so $\mu^a\in \Gamma(\ce^a(-2))$ is a gradient. 
 \end{remark}

The result \nn{newmid} suggests that we define a change of variable
$$
\rho':= \rho -\frac{1}{n}\si^{cd} \phi_{cd} .
$$ 
Note that $\rho'$ transforms in the same way as $\rho$ (see e.g.\ \cite{CapGoCrelle} for the latter), under the projective transformation associated with  a
change of background metric from $\bm$, since $\phi_{bc} $ is an
invariant of $(M,\bm)$. Thus we have the following:
 \begin{lemma}
On a metric projective structure $(M,\bm)$ with Weyl nullity, there is
a well-defined and canonical bundle isomorphism
$$
S^2\cT\to S^2 \cT
$$
given in any scale $g\in \bm$ by 
$$
(\si^{ab},\mu^b,\rho)\mapsto (\si^{ab},\mu^b,\rho -\frac{1}{n}\si^{cd} \phi_{cd}) .
$$
This is smooth on any open set where $B$ is smooth.
\end{lemma}

Now for metric projective structures $(M,\bm)$  with nullity we want to
construct a new and simple connection on tractors fields in $S^2\cT$
with solutions that are in agreement with those for the prolongation
connection of Theorem \ref{EMthm} (or at least this should be the case
for non-degenerate solutions).  This is linked to three equations that
together give the parallel transport.  The first equation we take from
the normal tractor connection \nn{s2conn-orig}:
\begin{equation}\label{eq1}
\nabla_a \si^{bc} + \delta^b_a \mu^c + \delta^c_a \mu^b =0.
\end{equation}
This is the metrisability equation $D\si=0$ of Theorem \ref{EM1}, see \nn{ef}.
The second is the equation \nn{newmid}
\begin{equation}\label{prekey}
\nabla_a\mu^b-\si^{bc}g_{ac}B + \delta^b_a \rho' =0,
\end{equation}
where we retain the notation $\rho'$ to record manifestly a
distinction from the variable $\rho$.

It remains to
treat the last equation. Here we assume that $B$ is smooth.
From the prolongation connection we have
$$
\nabla_a\rho-2\P_{ab}\mu^b -\frac{2}{n}C_{bad}\si^{bd}=0 ,
$$
where have continued our notation from above and used \nn{s2conn-orig} and \nn{psys}.
Now $\P_{ab}=Bg_{ab}+\phi_{ab}$, so we come to 
$$
\nabla_a\rho-2B\mu_a -2\phi_{ab}\mu^b -\frac{2}{n}C_{bad}\si^{bd}=0 .
$$ 
Now using that $\rho'=\rho-\frac{1}{n}\si^{ab}\phi_{ab}$, and
assuming $\si^{ab}$ solves \nn{metr}, we have 
$\nabla_a\rho= \nabla_a\rho' - \frac{2}{n} \phi_{ab}\mu^b+ \frac{1}{n}\si^{bc}\nabla_a\phi_{bc}$ and so
the display is equivalent to 
$$
\nabla_a\rho' -2B\mu_a + \frac{1}{n} \si^{bc}\nabla_a\phi_{bc} -2\frac{n+1}{n}\phi_{ab}\mu^b -\frac{2}{n}C_{bad}\si^{bd}=0 .
$$
Thus we have the following result.
\begin{proposition} \label{Newprop}
On a metric projective structure $(M,\bm)$,  with Weyl nullity,
the solutions to \nn{metr} are in one-to-one correspondence with solutions
of the following system on $S^2\cT$:
\begin{equation}\label{npsys}
\nabla^{\cT_1}_a\left( \begin{array}{c}
\si^{bc}\\
\mu^b\\
\rho'
\end{array}
\right) + \frac{1}{n}
\left( \begin{array}{c}
0 \\
0  \\
\si^{bc}\nabla_a\phi_{bc} -2(n+1)\phi_{ab}\mu^b - 2 C_{cab} \si^{bc}
\end{array}
\right) =0,
\end{equation}
where we calculate in a scale $g\in \bm$ and $\nabla^{\cT_1}$ is the $\ttt=1$ 
connection given in Proposition \ref{alphaf} (i.e. \nn{s2conn}). For solutions 
the section of $S^2\cT$ is in the image of the invariant operator $L^\phi$ of Proposition \ref{splitB}.
\end{proposition}
\begin{proof}
The first statement is proved above for solutions where $\si^{ab}$ is
non-degenerate, but from this and linearity the main result
follows. The final statement is immediate from formula \nn{npsys}.
\end{proof}
\begin{remark}
Since the connection $\nabla^{\cT_1}$ is 
invariant on $(M,\bm)$ it is evident that
$$
\si^{bc}\nabla_a\phi_{bc} -2(n+1)\phi_{ab}\mu^b - 2 C_{cab} \si^{bc}
$$ 
is also invariant for solutions.  
\end{remark}

\begin{remark}\label{nnewcon}
Since the second term of \nn{npsys} is linear in the variables
$(\si,\mu, \rho )$ it follows that the (total) system on the left-hand-side of
 \nn{npsys} defines a linear
connection on $S^2\cT$.
\end{remark}

\subsection{A simpler connection}\label{newconn}
We shall use further integrability conditions of solutions to improve
\nn{npsys} to a simpler and more elegant system. Again we assume that
$B$ is smooth here, and until Section \ref{Vsec}.

Differentiating \nn{newmid} yields
$$
 \nabla_a\nabla_b\mu^c +\delta^c_b\nabla_a \rho'-\si^c{}_b\nabla_a B
+ B \delta^c_a\mu_b +B g_{ab}\mu^c =0 ,
$$
and hence
\begin{equation}\label{int1}
 \mu^d R^c{}_{dab}+ \delta^c_b (\nabla_a \rho'-B\mu_a )- \delta^c_a
(\nabla_b \rho'-B\mu_b )= \si^c{}_b \nabla_a B -\si^c{}_a \nabla_b B .
\end{equation}
This contains key algebraic data for our system.

Contracting \nn{int1} with $\mu_c$ annihilates the first term on the
left-hand-side (since the Riemannian curvature is alternating on its
first arguments) and what remains is the identity
\begin{equation}\label{star}
\rho'_{[a} \mu_{b]}= B_{[a}\tilde{\mu}_{b]}. 
\end{equation}
Here, for the clarity of algebraic manipulations, we have introduced
the shorthand notations: $\rho'_{a}:= \nabla_a \rho'$, $
B_{a}:=\nabla_a B $, and $\tilde{\mu}_{b}:= \si^c{}_b\mu_c$.
Now this implies 
\begin{equation}\label{useful}
\rho'_a=\alpha \mu_a+\beta B_a , \quad \mbox{and} \quad \tilde{\mu}_b=\beta\mu_b+\gamma B_b
\end{equation}
for some functions $\alpha$, $\beta$, and $\gamma$.

We now
divide our discussion into the cases of $B$ constant or not. 

\subsubsection{$B$ constant} \label{Bc}
If $\nabla_a B=0$ then \nn{star}, equivalently \nn{useful}, simplify further and we obtain 
$$
\nabla_a \rho'= \alpha \mu_a. 
$$ 
Inserting this in \nn{int1}, and noticing that the right-hand-side
of this is zero by dint of $\nabla_a B=0$, we obtain
$$
 \mu^d R^c{}_{dab}+ \delta^c_b (\alpha-B )\mu_a - \delta^c_a
(\alpha-B)\mu_b =0. 
$$
 So by Theorem \ref{WeqK} $\mu^d$ is in the Weyl nullity, and then
 by Proposition \ref{B!} we conclude that $\alpha =2B$, at least on
 the open set where $\mu^a$ is non-zero.  So by Proposition
 \ref{Newprop} this implies that $ \si^{bc}\nabla_a\phi_{bc}
 -2(n+1)\phi_{ab}\mu^b - 2 C_{cab} \si^{bc}=0$ on the same set.  Now,
 Theorem \ref{Tae} implies that the set where $\mu^a$ is not zero is
 an open dense set in $M$.

In summary then, we have the
following result.
\begin{theorem}\label{4.7} 
Suppose that $(M,\bm)$ is a metric projective structure with Weyl
nullity and $B$ constant.
Then 
the solutions to \nn{metr} are in one-to-one correspondence with solutions
of the system:
\begin{equation}\label{npsys-BC}
\nabla^{\cT_1}_a\left( \begin{array}{c}
\si^{bc}\\
\mu^b\\
\rho'
\end{array}
\right) =0 .
\end{equation}
\end{theorem}
Note that not only does the system here invariantly describe solutions to
\nn{metr}, but recall (from Proposition \ref{alphaf}) that the
connection $\nabla^{\cT_1}$ itself is invariant on any metric
projective structure $(M,\bm)$ with Weyl nullity.

\medskip

As an immediate application let us pause to observe that Theorem \ref{GOTthm} and Theorem \ref{4.7} enable us to efficiently
relate solutions of the Gallot-Obata-Tanno equation \nn{one} to
solutions of the metrisability equation \nn{metr}.  Let us say that
solutions of \nn{one} (respectively \nn{metr}) are {\em algebraically
  generic} if the corresponding section of $S^2\cT^*$
(resp.\ $S^2\cT$) is everywhere of rank $(n+1)$.
Then we have the following result:
\begin{corollary}\label{GOT=m}
On a metric projective manifold $(M,\bm)$ with Weyl nullity and $g\in
\bm$ such that $B^g$ is constant, algebraically generic solutions of
 the Gallot-Obata-Tanno equation \nn{one} are equivalent to
 algebraically generic solutions of the metrisability equation
 \ref{metr}.
 \end{corollary}
\begin{proof}
Given a section $H$ of $S^2\cT^*$ which is parallel and of maximal
rank, its inverse $H^{-1}$ is a section of $S^2\cT^*$ which is
parallel and of maximal rank. The converse is also true. 
\end{proof}
\noindent The proof and result here follows a similar idea for
\cite[Theorem 4.3]{CGMac}.

Consider the splitting operator $L^\phi$ of  \nn{Lphi-op} taking solutions of
the the metrisability equation to the corresponding section of $S^2\cT$ (which is parallel in the
setting $B=$constant, of the Corollary above). If $\si$ is the solution
corresponding to the metric $g$ then $\nabla \si=0$, and we have
$L^\phi(\si)=(\si^{bc},~0,~\frac{1}{n}Bg_{bc}\si^{bc})$. Thus we see
the following:
\begin{proposition}\label{rankP}
On a metric projective manifold $(M,\bm)$ with Weyl nullity and $g\in
\bm$,
let $\si$ be the corresponding
solution of the metrisability equation. Then $ \operatorname{rank}
(L^\phi(\si) )=n+1 $ if and only if  $B_g\neq 0$.
\end{proposition}

\noindent From this and the Corollary \ref{GOT=m} it follows that, on
a manifold $(M,g)$, if $f$ is an algebraically generic solution of the
Gallot-Obata-Tanno equation then the constant $B^g$ is not zero.
However the converse is false. For example on the standard sphere we
have $B=1$, but for each solution $f$ of \nn{obata1} it follows that
$f^2$ is a solution of the Gallot-Obata-Tanno equation such that
$\bar{L}(\tau^{f^2})$ has rank 1.

Corollary \ref{GOT=m} above gives a non-linear map
  equating certain solutions of the Gallot-Obata-Tanno equation to
  corresponding solutions of the metrisability equation. This uses the
  {\em existence} of $g\in \bm$ such that $B^g$ is constant, but does
  not otherwise directly use $g$. However is we allow the metric $g$
  to be used directly then a stronger result is available as follows:
\begin{proposition}\label{GOT=met}
Consider a metric projective manifold $(M,\bm)$ with Weyl nullity and
$g\in \bm$ such that $B^g\neq 0$ is constant.  Then solutions of the
metrisability equation \nn{metr} are in 1-1 correspondence with
solutions of the Gallot-Obata-Tanno equation.

Furthermore in the scale of the fixed metric $g\in \bm$, a solution
$\si^{ab}$ of the metrisability equation \nn{metr} with $\mu^a_x$ not
zero, at a given point $x\in M$, is equivalent to a solution $f$ of
the Gallot-Obata-Tanno equation \nn{one} with $df_x$ not zero.
\end{proposition}
\begin{proof}
Since $B_g\neq 0$ and constant, it follows from Theorem \ref{4.7} and
Proposition \ref{rankP} that there is a non-degenerate metric $H^{-1}$
on the bundle $S^2\cT^*$ that is parallel for $\nabla^{\cT_1}$ (namely
$H^{-1}= L^\phi(\si^g)$ where $\si^g$ is the solution of \nn{metr}
corresponding to $g$). This and its inverse enable us to identify
$\cT$ with its dual $\cT^*$ in a way that preserves
$\nabla^{\cT_1}$. Applying this to tensor powers we see that, in
particular, we can identify parallel sections of $S^2\cT^*$ with
parallel sections of $S^2\cT$. Thus the the first result follows immediately
from Theorem \ref{GOTthm} and Theorem \ref{4.7}.

In the scale $g$ the tractor metric $H$ and its inverse are block
diagonal. Thus the final claim follows from the formulae \nn{Lphi-op}
for $L^\phi(\si)$ and $\bar{L}^\phi(\tau^f)$ The latter is the ``matrix composition''' $HL^\phi(\si)H$,  so  one is not block
diagonal, then neither is the other.
\end{proof}

In one direction, the last result extends to the case that $B^g=0$. See Proposition \ref{GOT=met} below.
\begin{proposition}\label{mettoGOT}
Consider a metric projective manifold $(M,\bm)$ with Weyl nullity and
$g\in \bm$ such that $B^g= 0$.  Then non-parallel (with respect to the
Levi-Civita $\nabla^g$) solutions of the metrisability equation
\nn{metr} determine non-constant solutions of the Gallot-Obata-Tanno
equation \nn{one}.
\end{proposition}
\begin{proof}
Recall that, the equation \nn{metr} is given explicitly  by \nn{ef},
$$
\nabla_a \si^{bc} + \delta^b_a \mu^c + \delta^c_a \mu^b =0,
$$ 
in the scale $g$, and we work in this scale now.  As shown above
$\mu^a$ satisfies equation \ref{prekey}, which with $B^g=0$ simplifies
to
\begin{equation}\label{prekeyB0}
\nabla_a\mu^b+ \delta^b_a \rho' =0.
\end{equation} 
This is given in the scale $g$, and we work in this scale. From \nn{npsys-BC}
and $B^g=0$, we have $\nabla \rho'=0$. So 
$$
\nabla_a\nabla_b\mu_c=0, \quad
\mbox{and, in particular} 
\quad \nabla_{(a}\nabla_b\mu_{c)}=0. 
$$ 

But $\mu^c$ is a gradient, and, trivialising densities using
  $\tau^g$, we have 
\begin{equation}\label{mudf}
\mu_c= \nabla_c f
\end{equation} 
where
  $f:=-\frac{1}{2}g_{bc}\si^{bc}$. Then $f$ is a solution of the
  equation \nn{one} as claimed. Moreover the last claim is immediate from \nn{mudf}.
\end{proof}

\begin{theorem}\label{muae}
Consider a metric projective manifold $(M,\bm)$ with Weyl nullity and
$g\in \bm$ such that $B^g$ is constant. Let $\si^{ab}$ denote a solution  of the metrisability equation  
equation \nn{ef}
$$
\nabla_a \si^{bc} + \delta^b_a \mu^c + \delta^c_a \mu^b =0, \quad \mbox{in the scale~} g. 
$$ 
Then either $\mu^a$ is zero on $M$, or $\mu^a$ is non-zero on an
open dense set. 
\end{theorem}
\begin{proof}
This follows immediately from the corresponding result Theorem
\ref{Tae} for the Gallot-Tanno-Obata equation, given Proposition
\ref{GOT=met} and Proposition \ref{mettoGOT}.
\end{proof}
\begin{remark}
In part Theorem \ref{muae} can be seen more directly. For example in
the case of $B^g=0$ it follows easily from the form of equation
\ref{prekeyB0}. This is an overdetermined finite type PDE among the
variety considered in \cite{BCEG}, thus non-trivial solutions
$\mu^a$ cannot vanish on an open set. 
\end{remark}

\subsubsection{Further refinements} 
We next want to understand the case where $B$ is not constant. Here we
observe that na\"{\i}ve considerations reveal considerable
information. 

Suppose that $v^a$ is a vector field in the Weyl
nullity.  Contracting $v_c:=g_{ca}v^a$ into $\mu^dR^c{}_{dab}$ we have
$$
v_c\mu^dR^c{}_{dab}=- \mu_dv^cR^d{}_{cab}=\mu_dB(\delta^d_bv_a-\delta^d_av_b)= B(\mu_bv_a-\mu_av_b).
$$ 
So contracting $v_c$ into \nn{int1} yields:
$$
B(\mu_bv_a-\mu_av_b) +v_b(\rho'_a-B\mu_a)-v_a(\rho'_b-B\mu_b)=
\tilde{v}_bB_a-\tilde{v}_aB_b ,
$$ 
where $\tilde{v}_a:=\si_a{}^bv_b$. That is 
\begin{equation}\label{key-0}
v_b(\rho'_a-2B\mu_a)-v_a(\rho'_b-2B\mu_b)=
\tilde{v}_bB_a-\tilde{v}_aB_b .
\end{equation}
When $B$ is constant we have from above (e.g. Theorem \ref{npsys-BC}) 
that $\rho'_a-2B\mu_a=0$,
and so the display gives no restriction on $v^a$. Otherwise, if $B_a$
is not zero, we may substitute from \nn{useful} to find the strong constraint
$$
(\alpha-2B)v_{[b}\mu_{a]}= \tilde{v}_{[b}B_{a]} -\beta v_{[b}B_{a]}.
$$ 
Thus we have $\tilde{v}_a=\beta'\mu_b+\gamma' B_b$. Furthermore if $\beta=0$ then, using \nn{useful} and arguing as in Section \ref{Bc}, we again conclude that \nn{npsys-BC} holds. Otherwise 
if $\beta\neq 0$, 
$v_a=\alpha'\mu_a+\delta' B_a$  for some functions $\alpha'$,
$\beta'$, $\gamma'$ and $\delta'$. 
 Thus, where $B_a\neq 0$, the possibilities for vectors
in the nullity are seriously restricted.

\smallskip

Next observe that 
$$
v^bR^c{}_{dab}= v^bR_{ab}{}^c{}_{d}=-B(v^cg_{ad}-v_d\delta^c_a).
$$
So $\mu^d v^bR^c{}_{dab}= -B v^c\mu_a+ Bv^b\mu_b\delta^c_a $ 
and contracting $v^b$ into \nn{int1} gives 
\begin{equation}\label{key-1}
\delta^c_a(2Bv^b\mu_b- v^b\rho'_b)+ v^c(\rho'_a-2B\mu_a)=\tilde{v}^cB_a-\si^c{}_a v^bB_b .
\end{equation}
This shows that if $B_cv^c$ is non-zero at a point $x$ then, at $x$,
$\si^c{}_a$ is necessarily a low rank adjustment of a (density)
multiple of $\delta^c_a$.

Contracting the last display 
with $\tilde{v}^a$ gives
$$
\tilde{v}^c(2Bv^b\mu_b- v^b\rho'_b)- v^c(2B\tilde{v}^b\mu_b- \tilde{v}^b\rho'_b)=\tilde{v}^c(\tilde{v}^bB_b)-v^c (v^bB_b) ,
$$  
or equivalently
$$
\tilde{v}^c(2Bv^b\mu_b- v^b\rho'_b - \tilde{v}^bB_b)= v^c(2B\tilde{v}^b\mu_b- \tilde{v}^b\rho'_b- v^bB_b)
$$
Now, by working locally if required, let us suppose that $v^a$ is nowhere zero.
The last display shows that, at a point $x$, either 
\begin{equation}\label{unl}
2Bv^b\mu_b- v^b\rho'_b - \tilde{v}^bB_b =0
\end{equation}
or 
$$
\tilde{v}^a= f_xv^a \quad \mbox{for some number $f_x$}.
$$ 
Let us first assume that \nn{unl} does not vanish at some
point and hence in an open neighbourhood, and work in that
neighbourhood. We have there $\tilde{v}^a= fv^a$ for some function
$f$, and putting this into \nn{key-1} we see that 
$$
\delta^c_a(v^b\rho'_b-2Bv^b\mu_b)+ v^c(2B\mu_a-fB_a-\rho'_a)=\si^c{}_a (v^bB_b ).
$$
By symmetry we have that 
\begin{equation}\label{null-f}
2B\mu_a-fB_a-\rho'_a = \tilde{\kappa} v_c ,
\end{equation}
for some function $\tilde{\kappa}$ and, by our assumption on \nn{unl},
$\tilde{\kappa} \neq 0$ at $x$. It follows that $v^bB_b \neq 0$ at $x$,   
and we have that
\begin{equation}\label{wow}
\si^c{}_a = \nu \delta^c_a + \kappa v^cv_a,
\end{equation}
in a neighbourhood of $x$, for some functions $\nu$ and $\kappa$.
Thus we see that with the assumption that \nn{unl} is non-vanishing we
have a very strong restriction \nn{wow}. In fact we will see below
that we can strengthen this result.

\vspace{1cm}

\subsubsection{$B$ non-constant}
We derived \nn{newmid} assuming that the metric projective structure
$(M,\bm)$ has nullity. We saw in Section \ref{Bc} above that in this case and if $B$ is constant then the vector field $\mu^a$ lies in the nullity. In fact the later holds without the assumption that $B$ is constant:
 The following result is critical for our
subsequent discussion.
\begin{theorem}\label{mid2n}
Suppose that $(M,\bm)$ is a metric projective structure with Weyl
nullity at every point and that the metrisability equation \nn{eq1}
holds:
$$
\nabla_a \si^{bc} + \delta^b_a \mu^c + \delta^c_a \mu^b =0.
$$
 Then, the vector field $\mu^a$ satisfies
 \begin{equation} \label{mu_0} \mu^b  W^a{}_{bcd}=0
\end{equation}
at every point.
\end{theorem} 
 This
result is critical for our subsequent discussion but to obtain it in
this generality takes some work, so we postpone the proof of this
until the next section (see Theorem \ref{lambda_nullity}). Let us
first observe some useful consequences.  First we use it to compute an
alternative formula for $\nabla_a \rho'$ that yields a variant of the
result in Theorem \ref{npsys-BC}.

Using that $\mu^d$ lies in the nullity we have
$$
\mu^d R^c{}_{dab}=B\delta^c_a \mu_b- B\delta^c_b \mu_a ,
$$
and so \nn{int1} simplifies to
\begin{equation}\label{key}
\delta^c_b (\nabla_a \rho'-2 B\mu_a )- \delta^c_a
(\nabla_b \rho'- 2 B\mu_b )= \si^c{}_b \nabla_a B -\si^c{}_a \nabla_b B .
\end{equation}
Contracting with $\delta^b_c$ we obtain 
$$\textstyle
\nabla_a\rho'-2B\mu_a +\frac{1}{n-1}\si^{bc}(g_{ab}\nabla_c B 
- g_{bc}\nabla_a B)=0
 .
$$
In summary:
\begin{theorem}\label{mconnthm} 
On a metric projective structure $(M,\bm)$ with Weyl nullity almost
everywhere and $B$ smooth, the solutions to \nn{metr} are in
one-to-one correspondence with solutions of the following system:
\begin{equation}\label{npsys-s}
\nabla^{\cT_1}_a\left( \begin{array}{c}
\si^{bc}\\
\mu^b\\
\rho'
\end{array}
\right) + \frac{1}{n-1}
\left( \begin{array}{c}
0 \\
0  \\
\si^{bc}(g_{ab}\nabla_c B 
- g_{bc}\nabla_a B )
\end{array}
\right) =0.
\end{equation}
\end{theorem}
\begin{remark}
Note that we derived the system \nn{npsys-s} assuming that $B$ was not
constant, but the result in any case generalises the $B$ constant
case. So the system \nn{npsys-s} applies without any assumptions on the
constancy  of $B$. 
\end{remark}

\begin{remark} \label{liu}
Observe that since the second term of \nn{npsys-s} is
  linear in the variables $(\si,\mu, \rho )$, (depending on just $\si$
  thereof) it follows that the total system defines a linear
  connection on $S^2\cT$. By construction this is invariant on
  solutions of the metrisability equation on $(M,\bm)$: 
  It is derived
  from the projectively invariant system \nn{psys} using only that
  $\si$ is a solution of the projectively invariant equation
  \nn{metr}, in the case that $g$ is a metric in $\bm$. It is therefore expected,
  that this linear connection is metric 
  projectively invariant, or, which is the same, that  the (0,3)-tensor field $L_{abc}:= g_{ab}\nabla_c
  B - g_{bc}\nabla_a B$ is metric projectively invariant: if we take another metric $\bar g$ in the same projective class and the corresponding $\bar B:= B_{\bar g}$, then  $ \bar L_{abc}:= \bar g_{ab}\nabla_c
  \bar B - g_{bc}\nabla_a \bar B= L_{abc}$. 
  
  We claim here, and explain the proof of this claim in Remark
  \ref{liouville}, that this expectation is true, i.e., {\it $L_{abc}$
    is indeed a metric projective invariant.}  Of course it exists on
  metric projective manifolds with nullity only.  For the case of
  dimension 2 (where we always have nullity so this tensor is always
  defined), this projectively invariant tensor is very well known, is
  essentially due to \cite{liouville} and is often called the
  \emph{Liouville invariant}, see e.g. \cite{bryant}, or the
  projective Cotton tensor (see \ref{Cotton}).  It is the obstruction for a
  two-dimensional projective structure to be flat and in fact it can
  be constructed for any, not necessary metric, projective structure.
  For higher dimensions, it seems to be new, though of course it
  exists only if the projective structure is metric and has Weyl
  nullity. 
\end{remark}

\begin{remark} Since $\mu^a$ everywhere lies in the
 Weyl nullity we also have that $\mu^a\phi_{ab}=0$. Thus putting
 together \nn{npsys} and \nn{npsys-s} we conclude that for solutions
 of \nn{metr} (on a manifold $(M,\bm)$ with Weyl nullity) we have the
 identity
$$
n\si^{bc}(g_{ab}\nabla_c B 
- g_{bc}\nabla_a B )=(n-1) (\si^{bc}\nabla_a\phi_{bc} - 2 C_{cab} \si^{bc}) . 
$$
\end{remark}

\subsection{Consequences for the solution $\si$} \label{solS}

First we observe an immediate implication of the Theorem \ref{mconnthm}.
\begin{theorem}\label{Btypefix}
On a metric projective geometry $(M,\bm)$, if $B$ is constant for one
metric in $\bm$ then it is constant for all metrics in $\bm$.
\end{theorem}
\begin{proof}
First observe that for both of the systems, \nn{npsys-BC} and
\nn{npsys-s}, a solution $(\si,\mu,\rho')$ must be in the image of
$L^\phi$, of \nn{Lphi-op}. 

Now let $g\in \bm$ and suppose that in the scale $g$ we have that
$\nabla_a B\neq 0$ at some point $x\in M$. The metric $g$ is
equivalent to a solution, that we will denote $\si$, of the equation
\nn{metr}. Thus by Theorem \ref{mconnthm} the prolongation of $\si$
given by $L^\phi$ solves \nn{npsys-s}. But for this solution it is
evident that the second term in the display \nn{npsys-s} is not zero
at $x$. Thus $\nabla^{\cT_1}L^\phi(\si)\neq 0$. The result now follows
immediately from Theorem \ref{4.7}, as the result
$\nabla^{\cT_1}L^\phi(\si)\neq 0$ is not dependent on the choice of
any metric from $\bm$.
\end{proof}

At this point we rather easily obtain consequences for the nature of
the solution $\si$ of \nn{metr}.
\begin{theorem}\label{mformthm}
Suppose that $(M,[g])$ is a metric projective structure with nullity
almost everywhere and $B$ non-constant almost everywhere. Then,
locally,  
$$
\si^a{}_b=\nu\delta^a_b+ \epsilon B^aB_b
$$ 
for suitable densities $\nu$ and $\epsilon$, and where $B_a:=\nabla_a B$.
\end{theorem}
\begin{proof} 
If $\nabla_a B\neq 0$ at some point,
  equivalently in a neighbourhood, then \nn{key} implies that
  pointwise $\si^c{}_b$ is a linear combination of $\delta^a_b$ and
  $B^aB_b$.
\end{proof}

 \subsection{The vector $\mu^j$ lies in the nullity of $W^i_{\ j k\ell}$} \label{Vsec}
 
 We will work on a metric projective manifold $(M, \bm)$. We take a
 metric $g\in {\bf m}$ and assume that there exists a solution
 $\si^{ij}$ of the metrisability equation. For convenience we will
 work in the scale of the metric $g$, using this also to trivialise
 density bundles via the volume density it determines.  We then write
 $a^{ij}$ for the unweighted $(2,0)$ symmetric tensor equivalent to
 $\si^{ij}$ in the given trivialisation; for convenience we shall then
 write the metrisability equation on $a$ as
  \begin{equation} \label{Mbasic} 
 a^{ij}_{\ \ ,k}= \lambda^i \delta^j_k + \lambda^j \delta^i_k.
 \end{equation} 
so that the (unweighted) vector field $\lambda^i$ corresponds (using
the given trivialisation of density bundles) to $-\mu^i$ above. The
choice of sign is to make our discussion in this section closely
compatible with some of the related existing literature (see
e.g.\ \cite{einstein, pseudosymmetric, mikesSur}).

Recall from \nn{mug} of Remark \ref{gradient} that $\mu^i$ is a
gradient, so we have the same for $\lambda^i$, it is the gradient of a
function $\lambda$:
\begin{equation}\label{lam} 
\lambda:= \tfrac{1}{2} g_{p q }a^{p q}  = \tfrac{1}{2} \trace(a^i_{\ j}) , 
\end{equation}  
and $\lambda^i=\nabla^i \lambda$.  In particular, the covariant
derivative of $\lambda_i$ is symmetric: $\lambda_{i ,j} =
\lambda_{j,i}$. Here, as usual indices are raised and lowered using
the metric.

 The main goal of this section is to prove Theorem \ref{mid2n} above,
 which we repeat here for convenience, in our current notation.
 \begin{theorem} \label{lambda_nullity} Assume the projective Weyl tensor of  $g$   has a nullity at every point. Then, the vector $\lambda^i$ satisfies 
 \begin{equation} \label{lambda_0} \lambda^s  W^i_{ \ s jk}=0\end{equation}
  at every point.
 \end{theorem}
 
 Since $\lambda^i$ is a smooth section of $TM$ it follows that if
 \eqref{lambda_0} holds at every generic point, then it holds
 everywhere.  (Generic here means that the multiplicities of all
 eigenvalues of $A$ are locally constant.)
 Thus we will work in the neighborhood of a generic point,
 and this will suffice.

  The proof of Theorem \ref{lambda_nullity} will require additional
  technical results which we formulate as separate statements.  Most
  of these do {\em not} use Weyl nullity. So until further notice we
  will not assume that there is (non-trivial) Weyl nullity.

 Without loss of
  generality we may assume that, in a small neighborhood that we are
  working in, the eigenfunctions of $a^i_{\ j}$ are smooth (possibly
  complex-valued) functions, and the rank of $a^i_{\ j}$ is
  constant. In the places below where we use index-free notations, we
  will denote the (1,1)-tensor $a^{i}_{\ j}$ by $A$ and view it as a
  tensor field of endomorphisms of $TM$.

 We consider the (point dependent) eigenvalue $\rho$ of $A$, 
and assume first
 that it is real-valued.  We consider (smooth) vector fields
 $\overset{1}\xi$, $\overset{2}\xi$,..., $\overset{m}\xi$ from the
 generalized eigenspace of $\rho$ such that
\begin{equation} \label{condition} 
g(\overset{\alpha}\xi,   \overset{\beta}\xi)
=\left\{ \begin{array}{ll} \varepsilon & \textrm{ if $\alpha+\beta=m+1$ }\\   
0 & \textrm{ otherwise, }\end{array}\right. \quad  \textrm{ and }\quad  A\overset{\alpha}\xi  = \rho\overset{\alpha}\xi +  \overset{\alpha-1}\xi,     
\end{equation} 
where $\overset{0}\xi:=0$, 
and $\varepsilon$ is plus or minus one (and is the same for all $\alpha,
\beta =1,...,m$).
The existence of such vector fields is
well-known, see for example \cite[Theorem 12.2]{lancaster}.
 We now consider 
a basis such that the first $m$ vectors are
$\overset{\alpha \geq 1}\xi$, in their given order, and the remainder are chosen to be  orthogonal to these. In
this basis the first $(m\times m)$-blocks of $g$ and of $A$ are given
by
  $$ \begin{pmatrix}
 &  & &   & \varepsilon \\
 &  & & \varepsilon &   \\
 &  & \iddots & & \\
 & \varepsilon &      &  & \\
\varepsilon &  &      &  &
\end{pmatrix}\ , \ \ 
\begin{pmatrix}
\  \rho & 1 & &  \\
                        &\rho & \ddots &  \\
                        &                        & \ddots & 1  \\
                        & & &  \rho \
\end{pmatrix}.$$

\begin{remark} \label{rem:eigenvalues}   The choice of the vector fields  $\overset{\alpha}{\xi}$ is not unique. More precisely, if the eigenfunction $\rho$ has geometric multiplicity equal to $1$,  then the vector fields are unique up to a sign (in a small connected simply-connected neighborhood).
 But if $a^i_{\ j}$ has a bigger $\rho$-eigenspace, for each
 eigenvector $v$ there exists a number $m$ and the vectors
 $\overset{1}{\xi},..., \overset{m}{\xi}$ satisfying \eqref{condition}
 such that $\overset{1}{\xi}$ is proportional to $v$.  For 
$\rho$-eigenvectors $v$ and $u$ that are non-proportional,
 at some point, the corresponding linear spaces generated by
 $\underbrace{\overset{1}{\xi_v},...,\overset{m_v}{\xi_v}}_{\textrm{constructed
     by $v$}}$ and by
 $\underbrace{\overset{1}{\xi_u},...,\overset{m_u}{\xi_u}}_{\textrm{constructed
     by $u$}}$ have trivial intersection. An analogous statement is
 also true for the vectors $\overset{1}{\xi},..., \overset{m}{\xi}$
 satisfying \eqref{conditionsC1} and \eqref{conditionsC2} below.
\end{remark}

We will need  two technical statements. These are Lemma
\ref{eigenvalues} and Lemma \ref{eigenvaluesC}.

\begin{lemma}\label{eigenvalues}
In the notation above, the gradient ${\rho_,}^i$ of $\rho$ is,
at each point, a linear combination of the vectors $\overset{1}{\xi},...,
\overset{m}{\xi}$.
\end{lemma}
\begin{proof}
Consider the defined equation $A\overset{\alpha}\xi =
\rho\overset{\alpha}\xi + \overset{\alpha-1}\xi$, which in the 
tensor notation reads
$$ 
a_{ji} \overset{\alpha}{\xi^i}- \rho \overset{\alpha}\xi_j -
 {\overset{\alpha-1}\xi}_j=0,
$$  
differentiate it  covariantly  and substitute the derivatives of $a_{ij}$ 
given by \eqref{Mbasic} to obtain 

\begin{equation}\label{T0}
\overset{\alpha}\xi_k \lambda_j + g_{jk} \overset{\alpha}{\xi^i}
\lambda_i - \rho_{,k} \overset{\alpha}{\xi_j} + (a_{ij}- \rho g_{ij})
\overset{\alpha}{\xi^i}_{\ ,k} -  \overset{\alpha-
  1}\xi_{j,k}=0.\end{equation} 
 We now contract
$\overset{\beta}{\xi^j}$ into this equation;
in view of
$\overset{\beta}{\xi^j}(a_{ij}- \rho g_{ij}) = 
\overset{\beta-1}{\xi}_i $ we obtain

\begin{equation}\label{T} 
\underbrace{\overset{\alpha}{\xi_k} \overset{\beta}{\xi^j} \lambda_j +
  \overset{\beta}{\xi_k} \overset{\alpha}{\xi^i} \lambda_i}_{(A)} -
\underbrace{\rho_{,k} \overset{\beta}{\xi^j}
  \overset{\alpha}{\xi_j}}_{(B)} + \underbrace{
  \overset{\beta-1}{\xi^j} \overset{\alpha}{\xi_{j,k}} -
  \overset{\beta}{\xi^j}\overset{\alpha- 1}{\xi_{j,k}}}_{(C)}=0.
\end{equation}
Now we  denote the left hand side of \eqref{T} by $\overset{(\alpha, \beta)}T$ and  consider the  sum 
\begin{equation} \label{T1}
\sum_{\begin{array}{c} \alpha + \beta=m+1 \\ 1\le \alpha \le m\end{array} } \overset{(\alpha, \beta)}T. 
\end{equation}
This sum  is of course $0$ since each $\overset{(\alpha, \beta)}T$ is zero. 

From the other side, the sum of the $(C)$-terms of \eqref{T1} is
zero. In order to see this observe that
the sums  
 $$\sum_{\begin{array}{c} \alpha + \beta=m+1 \\ 1\le \alpha \le m-1\end{array} }  
   \overset{\beta-1}{\xi^j} \overset{\alpha}{\xi_{j,k}} \ \ \textrm{   and  }  \  
   \sum_{\begin{array}{c} \alpha + \beta=m+1 \\ 2\le \alpha \le m \end{array}}  \overset{\beta}{\xi^j}\overset{\alpha- 1}{\xi_{j,k}}$$  are equal  and,  since they come with different signs, cancel each other. 
 The two
remaining terms are $\overset{0}{\xi^j}
\overset{m}{\xi_{j,k}}$
 and $ \overset{m}{\xi^j}\overset{0}{\xi_{j,k}}$ and vanish because
$ \overset{0}{\xi^j}=0 $.

   The sum of the  $(A)$  terms is  
$$\sum_{\begin{array}{c} \alpha + \beta=m+1 \\ 1\le \alpha \le
       m\end{array} }{\overset{\alpha}{\xi_k} \overset{\beta}{\xi^j}
     \lambda_j + \overset{\beta}{\xi_k} \overset{\alpha}{\xi^i}
     \lambda_i}$$ and is manifestly a linear combination of the vectors
   $\overset{1}{\xi},..., \overset{m}{\xi}$.  The sum of the remaining
   (B)-terms
$$
\sum_{\begin{array}{c} \alpha + \beta=m+1 \\ 1\le \alpha \le m\end{array} } \rho_{,k} \overset{\beta}{\xi^j} \overset{\alpha}{\xi_j}
$$    is equal, in view of condition \eqref{condition},  to  $m\varepsilon \rho_{,k}$. Putting these results together
 we obtain that $ \rho_{,k}$ is a linear combinations of  $\overset{\alpha}{\xi_k}$ and $\overset{\beta}{\xi_k}$ which was our goal. Lemma \ref{eigenvalues}
is proved. 

\end{proof}

Let us now assume that the eigenvalue is complex-valued: $\rho= a +
\cplxi b$, where $a, b$ are real-valued functions and $\cplxi$ is the
imaginary unit and $b\ne 0$ at some point, and hence in a neighbourhood that we now work in. 
 We consider real vector fields
$\overset{\alpha} x,\overset{\alpha}y$, $\alpha= 1,...,m$ such that
\begin{equation}\label{conditionsC1} g(\overset{\alpha}x, \overset{\beta}x)= 0, \  g(\overset{\alpha}y, \overset{\beta}y)= 0,   \  g(\overset{\alpha}x, \overset{\beta}y)= \left\{\begin{array}{cc}  1 & \textrm{ if $\alpha+\beta= m+1$} \\ 0  & \textrm{ if $\alpha+\beta\ne  m+1$} \end{array}\right. \end{equation}
and such that for the complex-valued vectors $\overset{\alpha}
\xi:=\overset{\alpha}x + \cplxi \overset{\alpha}y$ we have
\begin{equation} A\overset{\alpha}\xi  = \rho\overset{\alpha}\xi +  \overset{\alpha-1}\xi \quad \textrm{ where }\quad  \overset{0}\xi:=0.
\label{conditionsC2} \end{equation} 
This is the natural complex analogue of the condition
\eqref{condition}.  The existence of such vector fields follows again
from \cite[Theorem 12.2]{lancaster}. It is easy to check, by repeating
the arguments from the proof for the real eigenvalue, that the
equations \eqref{T0} and \eqref{T1} holds: the fact that $\rho$ and
$\overset{\alpha}\xi$ are complex-valued changes nothing. We again
consider the sum $$\sum_{\begin{array}{c} \alpha + \beta=m+1 \\ 1\le
    \alpha \le m\end{array} } \overset{(\alpha,\beta)}T$$ which is
zero.

 The sum of $(A)$ terms with a raised index is a linear combination of
 the (possibly complex-valued) vectors from the generalized eigenspace
 of $\rho$ with possibly complex-valued coefficients. For example, the
 first term of $\overset{(\alpha,\beta)}T$ is proportional to the
 (complex valued) $\overset{\alpha}\xi_k $ with the coefficient
 $\overset{\beta}{\xi^j}\lambda_j$, which is a complex function.

The sum of the $(C)$ terms is zero by the same argument as when the eigenvalue was real. Now, the sum of the $(B)$-terms is, in view of 
$$
\overset\alpha{\xi^j}\overset\beta\xi_j = g(\overset{\alpha}x + \cplxi \overset{\alpha}y , \overset{\beta}x + \cplxi \overset{\beta}y)=g(\overset{\alpha}x, \overset{\beta}x)+\cplxi g(\overset{\alpha}x  ,   \overset{\beta}y)+
 \cplxi g(  \overset{\alpha}y , \overset{\beta}x ) - g(\overset{\alpha}y , \overset{\beta}y) \stackrel{\textrm{\tiny $\alpha+ \beta= m+1$}}= 2\cplxi  
$$ equal to $ 2 m \cplxi \rho_{,k}$ as we want. Thus, the following analog of Lemma \ref{eigenvalues} for complex-valued $\rho$ 
is proved:

\begin{lemma}\label{eigenvaluesC}
In the notation above, the gradient ${\rho_,}^i$ (which is now a
complex-valued vector) of $\rho= a + \cplxi b$ is, pointwise, a linear
combination of the vectors $\overset{1}{\xi},..., \overset{m}{\xi}$.
\end{lemma}

\begin{lemma}\label{drhozero}
At a generic point, $d\rho=0$ for eigenvalues $\rho$ of $A$ with
geometric multiplicity $\ge 2$. 
\end{lemma}
\begin{proof}
 We work at a generic point of $M$. Suppose an 
eigenvalue $\rho$ has geometric multiplicity $\ge 2$. Then, one finds
two nonproportional $\rho$-eigenvectors $v, u$. Combining Lemma
\eqref{eigenvalues} (resp. Lemma \eqref{eigenvaluesC}, if $\rho$ is a
complex-valued) with Remark \ref{rem:eigenvalues}, we see that the
gradient ${\rho_,}^i$ lies in each of two eigenspaces whose
intersection is trivial; thus, as claimed, $d\rho=0$ for eigenvalues $\rho$ of
geometric multiplicity $\ge 2$.
\end{proof}

\begin{corollary} \label{cor:eigenvalues}
Suppose that on a pseudo-Riemannian manifold $(M,g)$ the metrisability
equation \nn{Mbasic} holds. Then, at a generic point, $\lambda^i$ lies
in the direct sum of the generalized eigenspaces whose geometric
multiplicity is one.
\end{corollary}

\begin{proof} We work at a generic point of $M$. 
Since, as we explained at the beginning of this section, $\lambda^i$
is the gradient of the function $\lambda= \trace{A}$, $\lambda^i$ is
therefore a linear combination of the gradients of nonconstant
eigenvalues of $A$, which by Lemmas \ref{eigenvalues},
\ref{eigenvaluesC} and \ref{drhozero} lie in the
direct sum of the generalized eigenspaces whose geometric multiplicity
is one. Corollary \ref{cor:eigenvalues} is proved.
\end{proof}

We need one further result that does not use Weyl nullity. Namely we
want integrability conditions for the equation \eqref{Mbasic}. One
obtains these by substituting the derivatives of $a^{ij}$, given by
\eqref{Mbasic}, into the Ricci identity $a^{ij}_{\ \ ,\ell k}-
a^{ij}_{\ \ ,k\ell}= a^{i p }R^{j}_{\ pk\ell} + a^{p
  j}R^{i}_{\ pk\ell}$ (which of course holds for every $(2,0)-$tensor
$a^{ij}$) to obtain:
 \begin{equation}  a^{i p }R^{j}_{\ pk\ell} +  a^{p j}R^{i}_{\ pk\ell} 
=\lambda^i_{\ ,  \ell} \delta^{j}_{k}+\lambda^j_{\ , \ell} \delta^i_{k}
-\lambda^i_{\ ,k} \delta^j_{\ell}-\lambda^{j}_{\ ,k} \delta^i_{\ell}. 
\label{int2} 
\end{equation}  
 The integrability condition in this form was obtained by Sinjukov
 \cite{Sinjukov}; in an equivalent form, it was known to Solodovnikov
 \cite{s1}.

\smallskip

Now let again consider the situation of possible Weyl nullity and 
the proof of Theorem \ref{lambda_nullity}. It will be
convenient to actually prove that
\begin{equation}\label{nullity3}\lambda^jZ^i_{\ jkm}=0,\end{equation} where $Z$ is given by \eqref{nullity2}; as we explained in Remark \ref{rem:nullity2} this condition is equivalent to the condition $
\lambda^s W^i_{\ sjk}=0.  
$

First let us observe that using \nn{int2} we recover Proposition \ref{altmid}:
If we contract \eqref{int2} with any nowhere zero vector $v^k $ in
the Weyl nullity then, using the symmetries of the curvature tensor, we
obtain
\begin{equation}\label{trace}
(\lambda^i_{\ , \ell } + B a^i_{\ \ell})v^j + (\lambda^j_{\ , \ell } + B a^j_{\ \ell})v^i= U^i \delta^j_\ell + U^j \delta^i_\ell 
\end{equation}
for $U^i=\frac{1}{n+1}( \lambda^i_{\ ,k} v^k + B a^{ik}v_k)$.  Now, the equation
\eqref{trace} immediately implies that the trace-free part of
$(\lambda^i_{\ , \ell } + B a^i_{\ \ell})$ is zero, so we have recovered Proposition \ref{altmid}. 
This implies that
the covariant derivative $\lambda^i_{\ , j}$ satisfies the equation
\begin{equation}\label{vnb}
 \lambda^i_{\ , j} = \rho' \delta^i_j - B a^{i}_{j},
\end{equation}
for some function $\rho'$ that absorbs the trace terms.

\begin{remark} \label{secondstep} The equation
    \eqref{vnb} is the essential result that we need from the Weyl
    nullity. First yields the critical algebraic equation \nn{acom}
    below, but more than this where $\lambda^i$ is nowhere zero it
    {\em implies Weyl nullity}, see Proposition \ref{con2null} below.
\end{remark}

Now substituting \eqref{vnb} in \eqref{int2}, we see that all
terms with $B$ vanish and obtain
\begin{equation} \label{acom}  a^{i p }Z^{j}_{\ pk\ell} +  a^{p j}Z^{i}_{\ pk\ell}= 0.\end{equation}

We will now establish the following linear algebraic result:
\begin{lemma}\label{keys} 
For $x\in M$, suppose that $A\in \End(T_x M)$ satisfies \nn{acom}, for
$Z^{j}_{\ pk\ell}$ at $x$, and that $v^i$ is a vector from a
generalised eigenspace of $A$ corresponding to an eigenvalue $\rho$ of
geometric multiplicity  one.  Then $v^s Z^i_{\ sjk}=0$.
\end{lemma}
\begin{proof}
 For simplicity we assume that the eigenvalue $\rho$ is real, the
 proof for complex-valued eigenvalues is essentially the same and will
 be left to the reader.  We take a point $x\in M$.  Without loss of
 generality we may assume that the eigenvalue $\rho$ is actually equal
 to $0$ at $x$, since adding a constant multiple of $\delta^i_j$ to
 $A= a^i_{\ j}$ does affect \eqref{acom}.  Let $m$ be the dimension of
 the generalized eigenspace corresponding to the eigenvalue $0$, and
 let us denote this generalized eigenspace by $V\subseteq T_xM$.  We
 consider a basis \{$\overset{1}\xi$, $\overset{2}\xi$,...,
 $\overset{m}\xi\}$ in $V$ such that
\begin{equation} \label{condition0} 
g(\overset{\alpha}\xi, \overset{\beta}\xi)=\left\{ \begin{array}{c}
  \varepsilon, \textrm{ if $\alpha+\beta=m+1$ }\\ 0, \textrm{
    otherwise }\end{array}\right.  \textrm{ and } A\overset{\alpha}\xi
=  \overset{\alpha-1}\xi,
\end{equation} 
where $\varepsilon$ is plus or minus one (and is the same for all
$\alpha, \beta =1,...,m$), and $\overset{0}\xi:=0$.
The
existence of such a basis follows again from \cite[Theorem
  12.2]{lancaster} in view of the condition that the geometric
multiplicity of the eigenvalue $0$ is one.

Consider now two arbitrary vectors $X, Y\in T_xM$ 
 and the endomorphism $$\tilde Z:=  Z^i_{\ j k\ell } X^kY^\ell : T_xM\to T_x M.
$$ 
Let us note first that \eqref{acom} implies that $\tilde Z$
 commutes with $A$, i.e.,
 $$
 A \tilde Z = \tilde Z A.
 $$
 Now, from the definition of $Z$ we see that $\tilde Z$ is $g$-skew symmetric, in the sense that the bilinear form  
$g(\tilde Z \cdot , \cdot)$ is  
skew-symmetric. Then, for any integer $r\ge 0$, we have 
$$
 g(A^r\tilde Z \cdot , \cdot)=  g(\tilde Z \cdot ,A^r \cdot)=-g(\cdot, \tilde ZA^r \cdot)= -g(\cdot, A^r\tilde Z \cdot ), 
$$ 
  so the bilinear form $g(A^r\tilde Z \cdot , \cdot)$ is skew symmetric and in particular 
  $$g(A^r\tilde Z U , U)=0 \textrm{ \ \ for any $U\in T_xM$}.$$
  Then, for every $\alpha$ and $\beta \in \{1,..., m\} $ such that $\alpha\ne m$ and $\beta\ge \alpha$, we have 
  $$ g(\tilde Z \overset{\alpha}\xi, \overset{\beta}\xi)= g(\tilde Z A
  \overset{\alpha+1}\xi, \overset{\beta}\xi)= g(\tilde Z
  A^{\beta-\alpha} \overset{\beta}\xi, \overset{\beta}\xi)= 0.$$ On
  the other hand, for any vector $\eta\in T_xM$ orthogonal to $V$, we
  have
  $$ g(\tilde Z \overset{\alpha}\xi, \eta)= 0, $$ since $\tilde Z
  \overset{\alpha}\xi \in V$ because $\tilde Z$ and $A$ commute.
Thus, the
  1-form $g(\tilde Z \xi , \cdot)$ vanishes for any $\xi \in V$, which
  implies that any vector $v=v^i$ of $V$ lies in the nullity of $Z$,
  as we claimed.

 In summary, we have shown that   every vector $v^i$ from a  generalized eigenspace of $A$ such that  the 
  geometric multiplicity is one satisfies 
$v^s Z^i_{\ sjk}=0$, so we are done.
\end{proof}

We are now ready to prove the main theorem from this section: 
\begin{proof}[Proof of Theorem \ref{lambda_nullity}]
From Lemma \ref{cor:eigenvalues} we have that $\lambda^i$ is a linear
combination of vectors from generalized eigenspaces of $A$ of
geometric multiplicity $1$.  Thus from Lemma \ref{keys} $\lambda^s
Z^i_{\ sjk}=0,$ and Theorem \ref{lambda_nullity} is proved.
\end{proof}

To close this part let us observe some consequences of the equation \nn{vnb} (i.e.\ \nn{prekey}),
$$
 \lambda^i_{\ , j} = \rho' \delta^i_j - B a^{i}_{j}. 
$$
First, near the points where $g$ is not proportional to (its projectively
equivalent) $\bar g$ we have that $\delta^i_j$ and $a^i_j$ are
linearly independent, at each point $x\in M$. Thus the coefficients
$B$ and $\rho$ evidently are smoothly point dependent (which we did
not assume a priori) as $\lambda^i$ is smooth and $B$, $\rho$ are the
coefficients of two smooth tensor fields which are linearly
independent at each point. 

Finally here we note that the equation \nn{vnb} is intimately related to
Weyl nullity. It was obtained by assuming that on a metric projective
manifold with Weyl nullity there is a solution $a^{ij}$ to
metrisability equation. On the other hand there is also a converse: 
\begin{proposition}\label{con2null} On a pseudo-Riemannian manifold $(M,g)$, 
suppose that \nn{Mbasic} holds, that is 
$$ 
a^{ij}_{\ \ ,k}= \lambda^i \delta^j_k + \lambda^j \delta^i_k, 
$$
and that the $\lambda^i$ here satisfies 
\nn{vnb}.  
Then
$$
\lambda^pW^i{}_{pjk}=0, \quad \mbox{at } x.
$$ 
So if $\lambda^i (x)\neq 0$, for some $x\in M$, then there is
projective Weyl nullity at $x$.
\end{proposition} 
\begin{proof}
This follows at once from the proof of Theorem \ref{lambda_nullity}
above, as in that proof Weyl nullity was only used to obtain \nn{vnb}.
\end{proof}

\begin{remark} 
The tensor $Z=Z^{i}_{\ jk\ell}$ played important role in this section. 
 Let us explain its geometric sense. Consider the  projectively invariant 
connection $\nabla^{\cT_1}$ (e.g. from \eqref{tconn1}). Its curvature, naturally projected to the manifold, is precisely the tensor $Z$. For example, if our metric 
 has constant sectional curvature, then $\nabla^{\cT_1}$ is flat, so its curvature vanishes which implies that  $Z$ vanishes -- which of course follows trivially from the definition of $Z$. 
\end{remark}

\subsection{Strictly nonproportional projectively equivalent metrics with Weyl nullity have constant curvature.} 
As a byproduct of the technical results obtained in section \ref{Vsec}
we obtain the following result. We say two metrics $g$ and $\bar g$
{\em strictly non-proportional} at a point, if the minimal polynomial
of the $(1,1)$-tensor $g^{is}\bar g_{sj} $ has degree $n=
\operatorname{dim} M$, at the given point.  In the case that one of $g$ or $\bar g$ has Riemannian signature,
strict non-proportionality of $g$ and $\bar g$ is
equivalent to the existence of $n$ different eigenvalues of
$g^{is}\bar g_{sj} $. In any signature, it is equivalent to the
property of each eigenvalue to have geometric multiplicity one.

\begin{theorem} \label{new}   On an connected manifold metric projective structure $(M,\bm) $ of dimension $n\ge 3$, suppose that  $g,\bar g\in \bm$
are strictly non-proportional, at least at one point. If $\bm$ has Weyl
nullity, then $g$ and $\bar g$ have constant
curvature.
\end{theorem} 

{\bf Proof.} By \cite[Proposition 2.1]{unpublished}, the metrics are
strictly non-proportional at almost every point. 
Let us work in the scale $g$ and
let the tensor $a^{ij}$ (satisfying \nn{Mbasic}) correspond to the
metric $\bar g$.  As proved in Section \ref{Vsec}, the existence of a
nullity implies that the tensor $a_{ij}$ satisfies \eqref{acom}.
Furthermore, by the assumptions of the Theorem, at almost every point
each eigenvalue of $A=(a^i_{\ j})$ has geometric multiplicity one. It
then follows easily from Lemma \ref{keys} that the tensor $Z^i_{
  \ jk\ell}$ vanishes identically. Thus, the curvature tensor of $g$
is constant.  Theorem \ref{new} is proved.

\subsection{ If $\phi$ is the same  for two non-affinely projectively equivalent metrics, then there exists Weyl nullity} 

In Theorem \ref{fth} we have proved that if a metric projective
structure $(M,\bm)$ has Weyl nullity, then the tensor $\phi_{ij}$ is
an invariant of $(M,\bm)$. In particular, for two projectively equivalent
metrics $g$ and $\bar g$ we have
\begin{equation} \label{fedo}  
 \P_{ij} - B g_{ij}=  \bar \P_{ij} - \bar B \bar g_{ij},
\end{equation} 
in the setting of Weyl nullity.
The goal of this section is to prove the converse (assuming smoothness
of $B$ and that the projective equivalence is non-affine).
\begin{theorem}  \label{converse} 
Suppose $g$ and $ \bar g$ are projectively equivalent metrics on a manifold 
$M^n$ of dimension $n\ge 2$.  If \eqref{fedo} holds (for a smooth function
$B$), then \eqref{lambda_0} holds, that is,
$$
\lambda^s  W^i_{ \ s jk}=0
$$
 where, in local coordinates, 
$$
\lambda^s=-e^{2\Upsilon} \bar g^{s i}\frac{\partial}{\partial x^i} \Upsilon, \qquad  \Upsilon= \frac{1}{2(n+1)} \log\left(\left|\frac{\det(\bar g)}{\det( g)}\right|\right).
$$ 
If, in particular, the metrics are non-affinely projectively
related then there is Weyl nullity on the open set where $\Upsilon$ is
non-constant.
\end{theorem} 
\noindent In the Theorem here, and below, $\det( g)$ denotes the
determinant of the metric component matrix $(g_{ij})$, in the given
coordinates. Note that the ratio of determinants $\frac{\det(\bar
  g)}{\det( g)}$ is coordinate independent.

For the purposes of our calculations here we will calculate in the
scale of the metric $g$ which we will regard as the background metric
and denote  by ``comma'' the covariant differentiation with respect
to the Levi-Civita connection of $g$.  As preparation for proving
Theorem \ref{converse} let us describe an equation the covariant
derivative $\bar g_{ij, k} $ satisfies, and also identify the objects
of equation \ref{Mbasic} in our current terms.

As we recalled in section \ref{pg}, if two affine connections $\nabla$ and $\bar
\nabla$ are projectively equivalent, then they are related by
\eqref{ptrans}. In terms of the connection coefficients $\Gamma^i
_{jk}:=dx^i(\nabla_{\frac{\partial}{\partial
    x^k}}\frac{\partial}{\partial x^j} )$ this reads
 \begin{equation} \label{c1} 
 \bar \Gamma_{jk}^i  = \Gamma_{jk}^i + \delta_{\ k}^i\Upsilon_{j} + \delta_{\ j}^i\Upsilon_{k}.    
   \end{equation} 
 
  If $\nabla$ and $\bar \nabla$ are Levi-Civita connections of metrics
  $g$ and $\bar g$ respectively, then one can find explicitly
  (following Levi-Civita \cite{Levi-Civita}) a function $\Upsilon$ on
  the manifold such that its differential $\Upsilon_{,i}$ coincides
  with the $(0,1)$-tensor $\Upsilon_i$: indeed, contracting \eqref{c1}
  with respect to $i$ and $j$, we obtain $\bar \Gamma_{s i}^s =
  \Gamma_{s i}^s + (n+1) \Upsilon_{i}$.  From the other side, for the
  Levi-Civita connection $\Gamma$ of a metric $g$ we have $ \Gamma_{s
    k}^s = \tfrac{1}{2} \frac{\partial \log(|\det(g)|)}{\partial x_k}
  $.  Thus,
 \begin{equation} \label{c1,5}  \Upsilon_{i}= \Upsilon_{,i} \end{equation}
  for the function $\Upsilon:M\to \mathbb{R}$ given by 
  \begin{equation} \label{phi-2}  \Upsilon:= \frac{1}{2(n+1)} \log\left(\left|\frac{\det(\bar g)}{\det( g)}\right|\right). \end{equation}  In particular, the derivative of $\Upsilon_i$ is  symmetric, i.e., $\Upsilon_{i,j}= \Upsilon_{j,i}$. 
 
We can now use this to characterise projectively equivalent metrics:
The formula \eqref{c1} implies that two metrics $g$ and $\bar g$ are
geodesically equivalent if and only if for $\Upsilon_{i}$, the
differential of $\Upsilon$ given in \eqref{phi-2}, we have
\begin{equation}\label{LC}
    \bar g_{ij, k} - 2 \bar g_{ij} \Upsilon_{k}- \bar
    g_{ik}\Upsilon_{j} - \bar g_{jk}\Upsilon_{i}= 0.
\end{equation}

Next, note that the function $\Upsilon$ and the determinant of the (1,1)-tensor 
$a^i_{\ j}$ constructed by $g$  and $\bar g$ 
by the formula \begin{equation} \label{a} 
a_{ij}: =   e^{2\Upsilon} \bar g^{s q} g_{s i} g_{q j}\end{equation}
 are closely related, namely 
 $\exp(-2\Upsilon)= \det(a^i_{\ j})$. 
    Differentiating \eqref{a} and using \nn{LC}, we obtain \nn{Mbasic}, with 
  \begin{equation} \label{lambda} 
\lambda_{i} = - \Upsilon_s a^s{}_{i}= -e^{2\Upsilon}\Upsilon_s \bar
g^{s p} g_{p i}.
\end{equation} 
We are now ready to prove the Theorem. 

\begin{proof}[Proof of Theorem \ref{converse}] Combining \eqref{fedo} with \eqref{Btrans},
 we note that 
\begin{equation}\label{f11} 
 \Upsilon_{i,j} - \Upsilon_{i} \Upsilon_{j} = B g_{ij}-\bar B \bar g_{ij}. 
 \end{equation} 

Now we first covariantly differentiate \eqref{lambda}, then we
use the expression \eqref{LC} for $\bar g_{ij,k}$, and finally
we substitute \eqref{f11} to obtain
 \begin{equation} \label{f2} \begin{array}{ccl}
 \lambda_{i,j} &=& -2 e^{2\Upsilon}\Upsilon_{j} \Upsilon_s \bar g^{s p} g_{p i}-e^{2\Upsilon}\Upsilon_{s,j} \bar g^{s p} g_{p i}+e^{2\Upsilon}\Upsilon_s  \bar g^{s q} \bar g_{q \ell,j} \bar g^{\ell p} g_{p i} \\
 &\stackrel{\eqref{LC}}{=}&   -e^{2\Upsilon} \bar g^{s p} g_{p i}(\Upsilon_{s,j} -  \Upsilon_s\Upsilon_j)+e^{2\Upsilon}\Upsilon_s \Upsilon_p \bar g^{s p}  g_{ i j }    \\
 &\stackrel{\eqref{f11}}{=}& -e^{2\Upsilon} \bar g^{s p} g_{p i}
 \left(B g_{s j}-\bar B \bar g_{s j}\right)
 +e^{2\Upsilon}\Upsilon_s \Upsilon_p \bar g^{s p}   g_{ i j }, 
    \end{array}     
 \end{equation} 

Finally we use \eqref{a}  to re-express this as 
\begin{equation}\label{vb} 
 \lambda_{i,j}= \rho' g_{ij}- B a_{ij}   
 \end{equation}
where $\rho'= e^{2\Upsilon}(\bar B+ \Upsilon_s \Upsilon_p \bar g^{s p} )$, 
is smooth. 
 This  is clearly equivalent to
\eqref{prekey} and \nn{vnb}.

Now, by Proposition \ref{con2null}, we have \eqref{lambda_0}. Theorem
\ref{converse} is proved.
\end{proof}

\section{Local and global structure for $B$ non-constant}\label{BnC}

 In this section we consider metric projective
  structures $(M,\bm)$.  We assume that the manifold $M$ is connected,
  and that $n=\dim M\ge 3$ (though some results trivially hold for
  $n=2$).  We also assume the existence of metrics $g, \bar g\in \bm$
  which are not affinely equivalent; as we know, from the previous
  section, this corresponds to the existence of a solution
  $\sigma^{ab}$ of the metrisability equation such that it is not
  parallel w.r.t.  the Levi-Civita connection of $g$.

 Our goal is to describe such $( M,\bm)$ locally (i.e.\ in an
 neighborhood of almost every point) and globally (assuming the
 manifold is closed, i.e., compact and without boundary) assuming the
 existence of a nullity of the Weyl tensor such that $B$ (constructed
 by $g$) is not constant.  We will see that near almost every point in
 a certain coordinate system the metric $g$ (up to a multiplication by
 a constant) has the warped product form
 \begin{equation}\label{warped}  
 g=  (dt)^2 + f(t) \sum_{i,j=1}^{n-1} h(x^1,...,x^{n-1})_{ij}dx^idx^j. 
 \end{equation} 
 We will also obtain a description, up to an isometry, 
  of all possible metrics $g$ (Riemannian, with nullity, admitting a nonparallel solution of the metrisability equation, with nonconstant $B$)  on  closed manifolds (of dimension $\ge 3$).

  \subsection{ Local theory if $B\ne \textrm{const}$}   \label{Bnichtcon}

We will work on the scale of the metric $g$ and use the Levi-Civita connection of $g$ for covariant differentiation;  then the equation \eqref{key} reads 
\begin{equation}\label{key1}
\delta^i_k ( \rho'_{,j} +2 B\lambda_j )- \delta^i_j
( \rho'_{, k}+ 2 B\lambda_k )= a^i{}_j  B_{,k}-a^i{}_k  B_{j}  .
\end{equation} 
Here we, as  usually,  denote by $a^{ij}$ the tensor  obtained when we  multiply $\sigma^{ab}$ by the weight parallel w.r.t. the volume form of $g$ (so $a^{ij}$ satisfies \eqref{Mbasic}),  and by $\lambda^i$ the 
 tensor  obtained when we  multiply $\mu^a $ by the weight parallel w.r.t. the volume form of $g$. We keep the notation $\rho'$ but now it is a function and is not a weight.

At the points such that $B_{,a}\ne 0$, the equation \eqref{key1}
immediately implies
\begin{equation}  
a^i{}_j=\nu\delta^i_j+ \varepsilon {B_,}^iB_{,j}\label{mm} 
\end{equation} 
for certain $\nu$ and $\varepsilon$ which are now functions (and not
weights). We will assume later that $\varepsilon \ne 0$; this is a
generic condition, since if $\varepsilon\equiv 0$ in a neighborhood
then $a_{ij}$ is proportional to $g_{ij}$ in this neighborhood which
implies that it is proportional to $g_{ij}$ on the whole manifold, which 
we assume to be connected.  Then, the tensor ${a^i}_j$ has (at most) two
eigenvalues at every point.  One of these eigenvalues  is $\nu$, it
has geometric multiplicity $n-1$. Indeed, any vector orthogonal to
${B_,}^i$ is an eigenvector of $ a^i{}_j$. 

   Let us now observe that the case when $ B_{,}^{\ i}$ is lightlike
   and nonzero on some open nonempty subset is not possible. Indeed,
   in this case $\nu$ is an eigenvalue whose geometric multiplicity is
   $n-1$ and algebraic multiplicity is $n$.  Then, by
   \cite[Proposition 2.1]{unpublished}, the function $\nu$ is
   constant.  Moreover, at every point of the manifold the constant
   $\nu$ is an eigenvalue of $a^i_j$ of geometric multiplicity at
   least two.  Then, the trace of $a^i{}_j$ is constant which implies
   that $a^i{}_j$ is parallel.  But then $\lambda^i{}_{,j}$ in
   \eqref{vnb} equals to zero which implies that $a^i{}_j$ is
   proportional to $\delta^i{}_j$ which contradicts \eqref{mm}
   (assuming $\varepsilon\ne 0$).

Now, at the points where ${B_,}^s B_{,s}\ne 0$ we have that
$B_{,}^{\ i}$ is an eigenvector with eigenvalue $\nu + \varepsilon
{B_,}^s B_{,s}$ of algebraic and geometric multiplicity $1$.

The case when ${B_,}^s B_{,s}\ne 0$ was considered in \cite{KM2014}. 
 By  \cite[Lemma 2]{KM2014}  the metric $g$ has (in a certain local coordinate system defined almost everywhere)    the warped form \eqref{warped}   and the solution $a^i_{ \ j}$ is given by the diagonal matrix  
   \begin{equation}\label{diag}
   \operatorname{diag}\bigl(\textrm{const} \cdot  f(x^1)+ \nu, \underbrace{\nu,...,\nu }_{n-1} \bigr). 
   \end{equation}

    Note that in the Riemannian case ${B_,}^s B_{,s}\ne 0$ for $
    B_{,}^{\ i}\ne 0$ so the (1,1)-tensor $a^i{}_j$ has precisely two
    eigenvalues in our neighborhood: $0$ of multiplicity $n-1$, and
    $\nu + \varepsilon {B_,}^s B_{,s}$ of multiplicity $1$.  By
    \cite[Corollary 1]{hyperbolic}, the metric has two eigenvalues  (one of multiplicity $1$ and another of multiplicity $n-1$) at
    almost every point, which implies that the metric has warped
    product structure at almost every point, which implies that it has
    nullity at almost every point  and hence at every point. Actually, this observation holds
    for metrics of arbitrary signature, but we do not prove it here.

   \begin{remark} \label{forfuture} 
Combining \eqref{mm}, the condition that $\nu$ is a constant,
\eqref{warped} and \eqref{diag} we see that the differentials $dB$ and
$df$ are linearly dependent.
   \end{remark}

\subsection{  Global theory if $B\ne \textrm{const}$}  \label{glth} 
 We again assume that $B$ is not constant in a neighborhood of $M$
 which is now assumed to be closed, and we will work in the notation
 of the previous section. Our goal is to describe all closed Riemannian
 manifolds admitting simultaneously both nonaffine projective
 equivalence and Weyl nullity with nonconstant $B$.
We begin by constructing two large classes of such Riemannian manifolds.

Take any $n-1$-dimensional Riemannian manifold $(N,h)$ equipped with a
positive function $f$ on $\mathbb{R}$, periodic with period $1$, and
consider the product $\mathbb{R}\times N$ with the warped product
metric
    \begin{equation} \label{inherit} dt^2 
   + f(t) \sum_{i,j=1}^{n-1} h_{ij}dx^idx^j\end{equation} (where
   $x^{1},...,x^{n-1}$ denote local coordinates on $N$ and $t$ is the standard coordinate on $\mathbb{R}$).  Next, take
   an isometry $I:N\to N$ and consider the action of the group
   $\mathbb{Z}$ generated by the isometry $(t, x)\mapsto (t+ 1,
   I(x))$. The quotient will be denoted by $M$, it is clearly a closed manifold. 
   For example, one can take $I= \operatorname{Id}$; in this case the manifold $M$ is topologically  the direct product $S^1\times N$. 
     Since the group
   $\mathbb{Z}$ acts by isometries, the metric \eqref{inherit} induces
   a metric on $M$ which we denote by $g$. The metric $g$ has Weyl  nullity
   at every point and admits non-trivial projective equivalence. Note
   that if $f\ne \operatorname{const}$,  there must exist a neighborhood such that
   $B_{i}\ne 0$.

 Let us now construct the next class of examples. Take the standard polar
 coordinates on the standard sphere. This coordinate system has two
 singularities that are traditionally called the north and south
 poles; the standard sphere metric has the warped product structure
 $dt^2 + \sin^2(t) \sum_{i,j=1}^{n-1} h_{ij}dx^idx^j$, where $h$ is
 the standard metric of the $n-1$-dimensional sphere and $t\in
 [0,\pi]$ is the altitudinal polar coordinate.  Now, replace the
 function `$\sin$' in this formula by any other smooth function f(t)
 such that it is positive outside of $0,\pi$, vanishes at $0,\pi$, and
 such that its derivative at $0$ is $1$ and at $\pi$ equals $-1$,
 i.e., consider the metric $dt^2 + f(t)^2 \sum_{i,j=1}^{n-1}
 h_{ij}dx^idx^j$. It is a smooth Riemannian metric on the sphere which
 is a warped product metric everywhere, except  possibly at the poles.

\begin{remark}
Note that, if $f\ne \operatorname{const}$, the degree of mobility (i.e.\ the dimension of the space of metrics projectively equivalent metrics) of
the metric on $(\mathbb{R}\times N)_{/\mathbb{Z}}$ constructed above
is precisely two, so any solution $a^i_{\ j}$ of the metrisability
equation has the form \eqref{diag} is this coordinate system.  The
degree of mobility of the metric on the sphere $S^n$ constructed above
is also precisely two if the function $f(t)$ is not equal to $\sin(t)$
(which would imply that the metric has constant sectional curvature),
and any solution $a^i_{\ j}$ of the metrisability equation has the
form \eqref{diag} is this coordinate system.
\end{remark}

The next theorem shows that the two classes of examples above effectively
capture all cases:

\begin{theorem} \label{globalBnC}
Suppose $g$ is a Riemannian metric on a closed connected manifold
$M$. Assume $g$ has Weyl nullity in all points of a certain
neighborhood, and assume that $B$ is not constant. Suppose there exists a
metric $\bar g$ that is projectively equivalent to $g$ and is not proportional to $g$.

 Then, for
a certain positive constant $C$, a finite  (at most, double) cover of $M$
equipped with a metric that is $C$ times  the lift of $g$  is isometric to one of
the examples above.
  \end{theorem}

\begin{proof}
We consider the solution
$a^{ij}$ of \eqref{Mbasic} corresponding to the metric $\bar g$.
Since the metric $g^{ij}$ itself satisfies \eqref{Mbasic}, for some constant $\nu$ we can
subtract $\nu\cdot \delta^i_j$ from $a^i_{\ j} $ so that, without loss
of generality, at every point of the manifold the
solution $a^i_{\ j}$ has, in a certain basis, the form
 \begin{equation}\label{diag0}
   \operatorname{diag}\bigl(f, \underbrace{0,...,0 }_{n-1} \bigr)  
   \end{equation}
  for some function $f$ which is nonzero almost everywhere. This uses
  Theorem \ref{mformthm}. Moreover, if the manifold is closed, the
  function $f$ is either non-positive everywhere or nonnegative
  everywhere by \cite[Corollary 1]{hyperbolic} and we may assume
  without loss of generality that it is nonnegative almost everywhere.

Let us show that one can construct, at least on the 2-cover of the
manifold, a smooth vector field $v^i$ such that
$$a^{ij} = v^iv^j.$$ This vector field is defined up to sign and
vanishes at the points where $f=0$.

Near the points where $f\ne 0$, the existence of such a vector field
is evident: at every point we take an $f$-eigenvector of ${a^i}_j$
normalized such that its length is $\sqrt{f}$. There are precisely two
choices for it; it is clear that, lifting to a 2-cover if required,
one can make the choices so that the resulting vector field is smooth
near every point where $f\ne 0$.

  In order to understand that one can extend the vector field to all
  the manifold, let us first explain that the points such that $f=0$
  are isolated.  We call such points {\em singular points}, and denote
  the set of such points by ${\bf Sing}$.
  
We will use that for every $t$ the function 
  \begin{equation} \label{integral} 
  I_t:TM\to \mathbb{R}, \ \  I(\xi)= g(\textrm{comatrix}(A- t\operatorname{Id} )\xi, \xi),  
  \end{equation}
  where we denote by $A$ the (1,1)-tensor  ${a^{i}}_{j}$ viewed as endomorphism, is an integral of the geodesic  flow of $g$, see \cite[Theorem 4]{hyperbolic} and the references inside (the fact is actually due to  \cite{MT} but is written in other notation there). Recall that a function $I$ on the tangent bundle 
   is an integral of the geodesic flow, if for any geodesic $\gamma(s)$ the function $t\mapsto I(\gamma'(s))$ does not depend on $s$. 
   
     Clearly, the   family of the   functions  $I_t$ is polynomial in $t$ of degree $n-1$.  Then, the function 
 $$
 \tilde I= \tfrac{1}{(n-2)!}\tfrac{d^{n-2}}{dt^{n-2}}_{|t=0} I_t: TM\to \mathbb{R}
 $$ is also an integral. 
 
  In an orthonormal  basis such that $A$ is given by \eqref{diag0}, the values of the 
  functions $I_t$ and $\tilde I$ on a tangent vector $\xi= (\xi^1,...,\xi^n)$,  are given  by 
  $$
  I_t(\xi)= t^{n-1} (\xi^1)^2 + (f-t)t^{n-2}\bigl((\xi^2)^2  +...+(\xi^n)^2  \bigr), \ \   \tilde I(\xi)=
   f \cdot \bigl((\xi^2)^2 +...+(\xi^n )^2\bigr). 
  $$
 We see that at the points such that $f=0$, the function $\tilde I$ vanishes for all tangent vectors. 
  We also see that at the points such that   $f\ne 0$, the vanishing of the function on a vector $\xi$ implies that 
  $\xi^2 =...= \xi^n=0$  implying that the vector $\xi$ is an $f$-eigenvector of $A$.

  Then, the existence of two points  $x_1$, $x_2\in {\bf Sing}$ 
   in a small neighborhood 
   implies the following contradiction: if we take a generic point $x$ of this neighborhood 
   (such that $f(x)\ne 0$ and such that this point does not lie on the  geodesic connecting $x_1$ and $x_2$) 
    and connect it by a geodesic $\gamma_1$ with $x_1$ and $\gamma_2$ with $x_2$, then the value of the integral $\tilde I$ on the velocity vectors of these  geodesics is zero because the geodesics contain  points such that $f=0$. Then, at the point $x$,  the velocity vectors of these 
     geodesics are $f$-eigenvectors of $A$ which is impossible since by assumption they are not proportional and the $f$-eigenspace  of $A$ is one-dimensional. 
   
   The contradiction shows that the points such that $f=0$ are isolated. 

\begin{remark} \label{byproduct}  
As a byproduct we obtained, that for geodesics passing through a singular point,  the velocity vector is an $f$-eigenvector of $a^i_{\ j}$. 
\end{remark}

Since the dimension of our manifold is at least 3, the compliment
$M\setminus {\bf Sing}$ is locally simply-connected, and in a
sufficiently small neighborhood $U$ of every point there are precisely
two possibilities for the choice of vector fields $v^i$ on $U\setminus
{\bf Sing}$ such that $v^iv^j= a^{ij}$. Then, at least on the 2-cover
of the manifold we can construct a smooth vector field $v^i$ on the
compliment to the singular set. We will think that the vector field is
constructed already on the $M\setminus {\bf Sing}$, and show that the
manifold $(M,g)$ is as in examples above.

First let us show that the vector field $v^i$ can be smoothly extended
to the points of ${\bf Sing}$. Of course, there is no problem at all
to extend it to ${\bf Sing}$ continuously, in order to do it we simply
define $v^i=0$ at the points of ${\bf Sing }$, but we would like to
have a smooth and not a merely continuous vector field $v^i$, so our
goal to show that this continuous extension is actually smooth.

In order to do it, let us first observe that in the coordinate system
where the metric has the form \eqref{warped} and $a^i_{ \ j} $ has the
form \eqref{diag0}, the vector field $v^i$ is given, up to sign, by
$\sqrt{f} \tfrac{\partial }{\partial t}$. Then, the orthogonal
distribution to $v^i$ is integrable and the function $f$ is constant
along it.

We consider now a singular point $p$, the geodesics passing through
the point, and spheres of small radii (in the distance function
corresponding to the metric $g$) around this point. These spheres are
orthogonal to these geodesics. But by Remark \ref{byproduct}, the
velocity vectors of such geodesics are proportional to $v^i$. So the
function $f$ is constant on the sphere.  Thus, in a neighborhood of
$p$, $f$ is a function of the distance to the point $p$, 
which we denote by $t$, i.e., $f= f(t)$.  This notation is compatible
with \eqref{warped}, since in the the exponential polar coordinates
$(r, x^1,...,x^{n-1})$ ( where $x^1,..., x^{n-1}$ are local
coordinates on the unit sphere in $T_pM$) the metric has the warped
product form \eqref{warped}.  
Then, the function $f(t)$ is a an smooth function of $t\geq 0$ and
since it is nonnegative and vanishes only at $t=0$ (at least, for
small $t$), it follows that the vector field $\sqrt{f(t) }
\tfrac{\partial }{\partial t}$ is a smooth vector field as we claimed.

 Let us now find out an equation the vector field satisfies.  
  In order to do it, we observe that 
  the covariant derivative of 
  $a^{ij}=v^iv^j$ is equal to \begin{equation} \label{bas1}  (v^i v^j)_{, k}=v^i_{\ , k} v^j+ v^j_{\ , k} v^i \end{equation} 
  but should, in view of \eqref{Mbasic}, be equal to 
  $
  \lambda^j \delta^i_k + \lambda^i \delta^j_k
  $  for some vector field $\lambda^i$. Comparing this with \eqref{bas1}, we see that 
 $v^i_{\ , k}$ is proportional to $\delta^i_k$: there exists a smooth function $\eta$ such that    
  \begin{eqnarray}
v^i_{\ ,j}&=& \eta \delta^i_j , \label{first}
\end{eqnarray} 
  i.e.\ it satisfies the equation studied in Section \ref{BGG}. 
The vector fields satisfying \eqref{first} were extensively studied in
the literature under different names, see \cite{KR} for references.
By the result of Tashiro \cite[Lemma 2.2]{Tashiro} (who called such
vector fields {\em concircular vector fields}), a compact manifold
admitting such a vector field is as we claimed.  Theorem
\ref{globalBnC} is proved.
\end{proof}

 Note that the equation \eqref{first} is equivalent to
 \eqref{psys1}. Thus by Proposition \ref{psysth} a warped product
 metric has nullity at every point.

 \subsection{If $B\ne \const$ in a neighborhood, then Weyl  nullity exists on the whole manifold} 
 
Here we prove the following statement. 
 
 \begin{theorem} \label{lllast}Let $g$ be a metric of arbitrary signature  on 
a connected $M$, and suppose $\sigma$ is a nonparallel solution of the
metrisation equation. Assume $g$ has a Weyl nullity in a certain
neighborhood, and suppose the corresponding $B=B^g$ is not
constant. Then, $g$ has a Weyl nullity on the whole manifold.
 \end{theorem} 
 
Examples show that the assumption that $B$ is not constant is
essential.

Recall that \emph{ Killing}  $(0,2)$ tensors are symmetric tensors  $Q_{ij}$  satisfying the Killing equation
$$
Q_{ij,k}+ Q_{jk,i}+ Q_{ki,j}=0.
$$ It is well known that Killing $(0,2)$ tensors are essentially the
same as integrals for the geodesic flow that are quadratic in
velocities: for a Killing tensor $Q$ the quadratic in velocities
function $\xi\mapsto Q(\xi, \xi)$ is an integral.

We will consider the Killing tensors corresponding to the integrals
$I_t$ given by \eqref{integral}, and denote them by $Q_t$. The tensors
$Q_t$ are given by the formula
\begin{equation}\label{Q} Q_t(\cdot,\cdot)= g(\mathrm{comatrix}(A- t\mathrm{Id})\cdot,\cdot).\end{equation}

In the proof of Theorem \ref{lllast} the main role is played by the
following result:

  \begin{lemma} \label{lem2} 
  The number of linear independent Killing tensors among the $Q_t$, in a
  small neighborhood of a generic point, is equal to the degree of the
  minimal polynomial of $A$ (i.e., the nonzero polynomial of the
  smallest degree that annihilate $A$).
 \end{lemma}

 In the language of integrals (linearly independent Killing tensors
 correspond to functionally independent integrals) this statement is
 known, \cite[Theorem 2 and Proposition 3]{Topalov}.

\begin{proof}[Proof of Lemma \ref{lem2}]  We show first that the degree of
 the minimal polynomial is greater than or equal to the number of
 linearly independent Killing tensors from the family $Q_t$. Because
 of \eqref{Q}, it is sufficient to consider (1,1)-tensors
 $$A_t:=\mathrm{comatrix}(A-t \mathrm{Id})$$ instead of $Q_t$, and show that
 the number of linear independent tensors among $A_t$ is at most the degree of the minimal polynomial of $A$.

   Let us first note that the family $A_t$ is polynomial in $t$ of
   degree $n-1$ (the coefficients of the polynomial are (1,1)-tensors,
   if we fix a point $x\in M$ and a basis in $T_xM$, $A_t$ is a
   polynomial in $t$ of degree $n-1$ whose coefficients are matrices).
   
   We call $t_0\in \mathbb{C}$ a \emph{zero} of the polynomial at
   point $x$, if the tensor $A_{t_0}(x)={\bf 0}$, and a \emph{zero} of
   the polynomial of order $k\ge 1$, if at the point $x$ we have that
   $\tfrac{d^{\ell}}{dt^{\ell}}_{|t=t_0}\bigl(A_t\bigr)={\bf 0}$ for
   all $\ell=0,...,k-1$.  Let us observe that if the geometric
   multiplicity of an eigenvalue $\rho$ of $A$ is greater than $1$,
   then $\rho$ is a zero of the polynomial $A_{t} $; on the way we
   will also see what is the order of the zero.

     Indeed, suppose an eigenvalue $\rho$ of $A$ has algebraic
     multiplicity $m_\rho $ and the maximal height of the Jordan block
     corresponding to $\rho$ is $h_\rho$. The assumption that the
     geometric multiplicity is at least two implies that $m_\rho-h_\rho\ge
     1$.  Then, $\rho$ is a zero of ${\det(A_t)}$ of multiplicity
     $m_\rho$, and is a pole of $(A_t)^{-1}$ (considered as a
     matrix-valued function) of multiplicity at most $h_\rho$. Then,
     $t_0=\rho$ is a zero of $ {\det(A_t)}
     (A_t)^{-1}=\mathrm{comatrix}(A-t \mathrm{Id}) $ of multiplicity
     $m_\rho-h_\rho$. Hence, $t_0=\rho$ is a zero of the multiplicity
     $m_\rho-h_\rho\ge 1$ of $A_t$.

       Note that by Lemma \ref{drhozero} for  $m_\rho> h_\rho$ we have that  $\rho$ is a constant. 
       
   Since the sum of $h_\rho$ over    all   eigenvalues $\rho$ is  the degree of the minimal polynomial of $A$ which we denote by $\deg_{\min}$, and the sum  
   of $m_\rho$ over    all   eigenvalues $\rho$ is $n$, we obtain the existence 
   of      $n-\deg_{\min}$  constant\footnote{in the sense that in all points $x$ of a small neighborhood  the same numbers are zeros}   zeros, counted with the multiplicities,     of  the polynomial $\mathrm{comatrix}(A-t \mathrm{Id}) $. Then, there exists a decomposition 
   $$
   A_t = P_{\textrm{MAT}} P_{\const}.
    $$
  Here   $P_{\textrm{MAT}}$ is a polynomial in $t$ of degree $\deg_{\min}-1$ whose coefficients are (1,1)-tensors, and $P_{\const}$ is a polynomial in $t$ of degree $n-\deg_{\min}$ whose coefficients are (constant real) numbers. In fact, the polynomial  $P_{\const}$ is the product of $(t- \rho)^{m_\rho - h_\rho}$ over all eigenvalues $\rho$ of $A$. 
  
  The proof of the existence of such decomposition is more or less the
  standard proof of the known statement that if a polynomial $P$ has
  zeros $\rho_1,..., \rho_\ell$ it is divisible by
  $(t-\rho_1)...(t-\rho_\ell)$, and the fact that in our case our
  polynomial has matrix coefficients does not really affect the proof,
  since the proof only needs the polynomial  remainder theorem.

   Then, each (1,1)-tensor among $A_t$ is a linear combination of  
    the   coefficients of the polynomial $P_{\textrm{MAT}}$, which implies that there is at most  $\deg_{\min}$ linearly independent $Q_t$.

   Let us now explain that the degree of the minimal polynomial is
   less than or equal to the number of linearly independent
   $Q_t$. Actually, it is a simple exercise in the linear algebra: we
   need to show that the number of linearly independent matrices among
   the matrices of the form $\textrm{comatrix}(A - t \textrm{ Id})$ is
   at least (in fact, precisely, since above we explained the ``at
   most'' direction) $\deg_{\min} A$.  We leave this exercise to the
   reader, and recommend to do calculations in the basis such that $A$
   has Jordan normal form.  Lemma \ref{lem2} is proved. \end{proof}

   \begin{corollary} \label{llll}
   The minimal polynomial of $A$ has the same degree at each point of an open everywhere dense subset of $M$. 
   \end{corollary} 
   
   \begin{proof}
 It is known, see e.g. \cite{Wolf}, that   the Killing equation  is of finite type. Then, the dimension of the space of Killing tensors $Q_t$ is the same in each neighborhood and the claim follows from  Lemma \ref{lem2}. 
   \end{proof} 
   
   \vspace{1ex} 
   {\bf Proof of Theorem \ref{lllast}.}   In dimension $n$ of the manifold is  two, it is nothing to show. We assume $n\ge 3$. 
   
    If $B$ is not constant in some neighborhood, as we explained in \S \ref{Bnichtcon}, the tensor $A$ is given by \eqref{diag}, and ${\deg}_{\min}(A)= 2$.   Then, by Corollary \ref{llll},  the degree  of the  minimal polynomial of $A$  is $2$ at almost every point, which implies that $A$ is as in \eqref{diag}, so the metrics has warped product form and there exists a nullity, or  $a^{ij} = v^iv^j + g^{ij},$  where $v^i$ is a light like vector field. But in the last case, as explained in  \S \ref{Bnichtcon}, the tensor $a^{ij}$ is parallel which implies that $v^i$ is parallel. Then, it lies in the nullity of the curvature tensor. Theorem \ref{lllast} is proved.

 Let us also note that under the assumption that there exists a nullity in some  neighborhood and a nonparallel solution of the metrisation equation we obtain the local description of the metric almost everywhere: it is either warped product metric, or has a parallel light like vector field; recall that 
   the 
 descriptions of metrics admitting a parallel vector 
 is know since at least \cite{eiparallel}.

 \section{ Projectively equivalent metrics on closed manifolds  if $B=\const$ and applications.}\label{BC} 
 
Here we assume that $\lb g \rb$ has Weyl nullity at almost every (and
therefore, every) point and that $\bar g\in \lb g \rb$ is non-affinely
(projectively) equivalent to $g$. We will also assume that the
function $B$ is constant.  Then, by Theorem \ref{4.7}, the equations
\eqref{npsys-BC} hold, which in the notation $(a,\lambda, \rho')$ take the form
 \begin{equation}\label{vnB}\begin{array}{lcr} 
 a_{ij,k}& = &  \lambda_i g_{jk} + \lambda_jg_{ik}\\
 \lambda_{i,j}& =&\rho' g_{ij} -B a_{ij}\\ 
 \rho'_{,k}&= &-2B\lambda_k.
 \end{array}.\end{equation}

 {Combining this with Proposition \ref{GOT=met}, Proposition \ref{mettoGOT} and result \cite[Theorem 1]{mounoud},  we obtain the
 following result:}
  \begin{theorem}  \label{Bcon} Assume $(M, g)$ 
(where $g$ has arbitrary signature) 
is  of dimension $n \ge 3$ and is closed. Assume that $(M,\lb g\rb )$ has Weyl  nullity 
 and that the coefficient $B^g$ is constant.  Let $\bar g$ be a metric
  that is projectively equivalent to $g$, but is not affinely equivalent
 to $g$.  Then, after multiplication by a constant, $g$ is a Riemannian metric of
 constant positive sectional curvature.
\end{theorem}

\subsection{Sasakian manifolds are geodesically rigid.}

Sasakian manifolds have Weyl nullity and $B=1$. 
Thus we immediately have the following result.
\begin{corollary}\label{sasy}
On a closed  Sasakian manifold $(M,g )$ of arbitrary signature, any metric in
$\lb g \rb$ is affinely equivalent to $g$,  unless for a certain constant $c\ne 0$ the metric $c g$ is the Riemannian metric of constant sectional  curvature equal to 1.
\end{corollary}

 Note that the statement of Corollary \ref{sasy} does not hold
 locally, as it follows from \cite[\S3]{conification} that there exist
 local Sasakian manifolds, of any odd dimension $\ge 3$ and of
 nonconstant curvature, admitting projectively but not affinely
 equivalent metrics.  More precisely, it was shown there these
 Sasakian manifolds are such that the cone over them admits
 nontrivial parallel symmetric $(0,2)$ tensors.

\subsection{Closed K\"ahler  manifolds  do not admit nontrivial projective equivalence} 
\begin{theorem}\label{kahler}
On a closed  K\"ahler  manifold $(M,g )$ of arbitrary signature, any metric in
$\lb g \rb$ is affinely equivalent to $g$. 
\end{theorem}

Note that locally there are K\"ahler metrics that are projectively equivalent, 
but not affinely equivalent. A simple example is the flat
metric. Examples with nonconstant sectional curvature also exist.

Note that there exist closed K\"ahler manifolds admitting an affinely
 equivalent nonproportional metric. Indeed, take two compact
K\"ahler manifolds $(M_1, g_1)$ and $(M_2, g_2)$. The metrics $g_1 +
g_2$ and $g_1 + 2 g_2$ on the direct product $M_1\times M_2$ are
affinely equivalent.
 
Theorem \ref{kahler} is an easy corollary of the following proposition:
\begin{proposition} \label{mainkahler} Let $(M^{n} , g, J)$, $n=2m\ge 4$,   be a connected K\"ahler manifold admitting a solution $a^{ij}$ 
 of the metrisability equation    which is not parallel. Then, at every point there exists a nullity and $B=0$. 
\end{proposition}

The assumption that there exists a solution of the metrisability
equation which is not parallel is important: a generic K\"ahler
metric, and even the Fubini-Study metric, does not have nullity.

Let us now explain why Proposition \ref{mainkahler} implies Theorem
\ref{kahler}: Assume that on a closed K\"ahler manifold $(M^{2m},g,J)$
we have a non-parallel solution of the metrisability equation.
Combining Proposition \ref{mainkahler} with Theorem \ref{Bcon} we that
there is a constant $C$ such that $Cg$ is the Riemannian metric of
constant sectional curvature $+1$. Thus the local holonomy group of
the metric $g$ is the whole $SO(2n)$, which is impossible since it
should preserve the complex structure.

\begin{proof}[Proof of Proposition \ref{mainkahler}.]  We will work in the
scale of the metric $g$. Let $a^{ij}$ be a solution of the metrisability
equation \eqref{Mbasic}.  Then, as it was shown in \cite[page
  133]{Sinjukov} (see alternatively \cite[Eq. (1.4)]{mikesSur}), there
exists a function $\mu$ such that the following equation holds:

\begin{equation} \label{mik1}
n\lambda^i_{ \ , j}= \mu \delta^i_{\ j}- a^{is} R_{sj} - a^{sp} R^i_{\ sp j}.  
\end{equation}

Let us now consider the (2,0)-tensor $\hat a^{ij}:= a^{ij} +J^j_{\ j'} a^{i'j'}J^i_{\ i'}.$ It is hermitian by construction. 
Since the complex structure $J$  is parallel by the definition of K\"ahler manifolds, the    tensor $\hat a$ clearly satisfies 
\begin{equation}  \label{mik}
\hat a^{ij}_{\ \ , k}= \lambda^ig^{jk} + \lambda^jg^{ik} +\bigl(\lambda^{i'}g^{j'k} + \lambda^{j'}g^{i'k} \bigr)J^i_{\ i'} J^j_{ \ j'}.
\end{equation} 

Hermitian tensors satisfying \eqref{mik} were actively studied in the context of  the
 so-called  h- or c-projectively equivalent metrics and of so-called Hamiltonian 2-forms. It is known (see e.g. \cite[Lemma 1]{FKMR} or more classical references given there) that the (1,1)-tensor 
  $\lambda^i_{\ , j}$ commutes   with the complex structure $J$. It is also known (see e.g. 
  \cite[page 216]{Sinjukov} or \cite[page 1336]{mikes1}) that  if a hermitian $\hat{a}^{ij}$ satisfies \eqref{mik} 
    there exists a function $\hat \mu$ such that the following equation holds:
    \begin{equation} \label{mik2}
n\lambda^i_{ \ , j}= \hat \mu \delta^i_{\ j}- \hat a^{is} R_{sj} - \hat a^{sp} R^i_{\ sp j}.  
\end{equation}

Let us now multiply \eqref{mik1}  by $J^{i'}_{ \ i} J^{j}_{\  j' }$; using the symmetries of the curvature tensor for K\"ahler manifolds  
we obtain 
after renaming the indexes  $i' \to i$ and $j'\to j$:
\begin{equation} \label{mik3}
n\lambda^i_{ \ , j}= \mu \delta^i_{\ j}- J^i_{\ i'} J^s_{\ s'}a^{i's'} R_{sj} - J^p_{\ p'} J^s_{\ s'} a^{s'p'} R^i_{\ sp j}.  
\end{equation}
Now, adding \eqref{mik1} and \eqref{mik3} and subtracting \eqref{mik2} we obtain  
$$
n\lambda^i_{ \ , j} = \tilde \mu \delta^i_{\ j}. 
$$
We  see that $(a^{ij}, \lambda^i)$ satisfies \eqref{vnb} with $B=0$; in view of Proposition \ref{con2null} our metric has nullity with $B=0$. Proposition \ref{mainkahler} is proved.  \end{proof}

\subsection{Projectively equivalent metrics with the same trace-free Ricci} 
\label{finalS}

Our goal is to prove the following Theorem. 

\begin{theorem}  \label{last} 
On a connected $M$ of dimension $n\ge 3$, suppose that $g$ and $\bar
g$ are non-affinely projectively equivalent and have the same
trace-free Ricci tensor, that is
\begin{equation} \label{fedo1} 
R_{ij}- \tfrac{R}{n} g_{ij}=\bar  R_{ij}- \tfrac{\bar R}{n} \bar{g}_{ij},\end{equation} 
 where $R_{ij}$ (resp. $\bar R_{ij}$) is the Ricci-curvature tensor
 and $R= R_{ij}g^{ij}$ (resp. $\bar R$) is the scalar curvature for
 $g$ (resp. $\bar g$).  Then, the metric-projective class $\lb g \rb $ has 
projective Weyl
 nullity at each  point $p$ with constant  $B$. 
 \end{theorem}

 \begin{proof} We see that the condition \eqref{fedo1} is a special
case of the condition \eqref{fedo}. So, by Theorem \ref{converse},
$\lambda^i$ lies in a projective Weyl nullity.  Let $U$ denote the
open set of points $p$ where $\lambda^i_p\ne 0$. On the closure
$\overline{U}$ of $U$ we have therefore projective Weyl nullity, with
$B= \tfrac{R}{2n(n-1)} $ on $U$. 

On the open complement of $\overline{U}$, if it is non-empty,
$\lambda^i$ is everywhere zero and so the metrics $g$ and $\bar{g}$
are affinely related. This implies $R_{ij}= \bar R_{ij}$ and then, by
\eqref{fedo1}, that the metrics are related by $R g_{ij}=\bar{R}
\bar{g}_{ij}$. If $R\neq 0$ at some point, and hence on some
neighbourhood, of $M\setminus \ol{U}$ then on that neighbourhood $g$
and $\bar{g}$ are both affinely and conformally related. Then, it
follows that $g$ and $\bar g$ are related by constant dilation; this
result of Weyl \cite{Weyl} is easily verified. But then, by Corollary
\ref{Bs}, the metrics are related by constant dilation on $M$. This
contradicts our assumptions in the Theorem here. So the only
possibility is that $R=0$ everywhere on $M\setminus \ol{U}$, and hence
also on its closure. In particular $R$ is constant there.

Suppose that the scalar curvature $R$ is constant on $\ol{U}$ then, using the
observations just made, it is constant on $M$.  Note then, on all of
$M$, we have a solution of the Gallot-Obata-Tanno equation \nn{one}
with $B_{\circ}:= \frac{R}{2n(n-1)}$ and
$f=\frac{1}{2}g_{ij}a^{ij}$. This holds trivially on $M\setminus
U$, as $\lambda_i=0$ there, while on $U$ it follows from
Proposition \ref{GOT=met} (and its proof which shows that $\mu=-\lambda$ and $df$ agree up to a constant factor) 
and Proposition \ref{mettoGOT}. By our
assumptions in the Theorem, this solution is not constant.  So, by
Theorem \ref{Tae}, the metric has the nullity at every point,
$B^g=B_\circ$ and we are done.

 It remains to show that $dR=0$ at every point; we will do it by
contradiction.  Suppose now we have a point such that $dR\ne 0$;    in a
neighborhood of such a point we also have  then $dB\ne 0$.  Then the
metric has (in some open nonempty subset of this neighborhood, and up
to multiplication of the  metric by a constant) the form
\eqref{warped}, which, for cosmetic reasons we rewrite as
 \begin{equation}\label{warped2}  
 g=  (dt)^2 + f(t)^2 \sum_{i,j=1}^{n-1} h(x^1,...,x^{n-1})_{ij}dx^idx^j,  
 \end{equation} 
 and the solution $a^i_{\ j}$ 
has a diagonal form
  (cf \eqref{diag}) which, after multiplication  by an
  appropriate constant, is given by
$$
a^i{}_j = \operatorname{diag}\left( f^2(t) + C, \underbrace{C, ... , C}_{n-1}\right),$$ where $C$ is a constant. 
Then, in view of \eqref{phi-2} and \eqref{a}, in these coordinates the metric $\bar g$, up to  multiplication by a constant,  is given by 
\begin{equation} \label{bargwarped} 
\bar g=  \tfrac{C}{(C+ f(t)^2)^2} dt^2 +  \tfrac{f(t)^2}{f(t)^2 + C}\sum_{i,j=1}^{n-1} h(x^1,...,x^{n-1})_{ij}dx^idx^j .
\end{equation}
For the (warped product)  metrics $g$ and $\bar g$, one may explicitly calculate the Ricci and the scalar curvatures and therefore the equation \eqref{fedo1}.  
Both metrics are actually warped product metrics, and their curvatures were calculated many times in the literature and easily can be done by computer algebra software; let us explain the idea we used in our calculations, since it will be used below and also in the next section.

  We will use that the conformally equivalent metric $\frac{1}{f^2} g
   $ is a direct product metric so its Ricci tensor has the form
 \begin{equation}  \left(\begin{array}{c|ccc}
          0 &&&\\
        \hline
        &&& \\
        &&\overset{0}{R}_{ij} &  \\
        &&&
      \end{array}\right)\label{curvH} \end{equation}
   where  $\overset{0}{R}_{ij}$ is the Ricci-tensor of the $(n-1)$-dimensional metric $h_{ij}$, and its scalar curvature is simply the scalar curvature of $h_{ij}$.  
   Now, it is well known (see e.g. \cite{wiki}) that the Ricci-tensors and the scalar curvatures  of any the conformally equivalent metrics  $g $ and 
   $ \hat g:= e^{2\psi} g$ 
   are related by 
   \begin{equation}\label{conf-2} 
   \begin{array}{ccl}  \hat R_{ij} &=&  R_{ij} - (n-2)(\psi_{,ij}- \psi_{,i} \psi_{,j}) - (\Delta_2  + (n-2)\Delta_1)g_{ij} , \\ \hat R&=& -e ^{-2\psi} ( R+ 2(n-1)\Delta_2 + (n-1)(n-2)\Delta_1 ), 
  \end{array} \end{equation}
   where $\Delta_2$ is the Laplacian of $\psi$, $\Delta_2= \psi_{,ij}
   g^{ij}$, and $\Delta_1$ is the square of the length of $\psi_{,i}$ in
   $g$, $\Delta_1:= g^{ij}\psi_{,i} \psi_{,j}$.  We apply these
   formulae with the metric $g$ in \eqref{conf-2} replaced by the direct
   product metric $\frac{1}{f^2} g $ and with $\psi=  \log f .$
   After some relatively simple calculations we obtain $R_{ij}$ as an
   algebraic expression in $\overset{0}{R}_{ij}$, $h_{ij}$, $f$, $f'$
   and $f''$, and also $\overset{0}{R} $ as an algebraic expression in
   $\overset{0}{R} $, $f$, $f'$ and $f''$.

   Similarly, the metric $\frac{  f^2+ C } {f^2}\bar g $ which is
   conformally equivalent to the metric $\bar g$ is also the direct
   product metric so its Ricci curvature also takes the form given in
   \eqref{curvH}. We again combine it with \eqref{conf-2} and calculate
   the scalar and the Ricci curvatures of $\bar g$.  Substituting the
   result of the calculation into \eqref{clc}, we obtain that   
   the matrix of 
\begin{equation} \label{clc} R_{ij}-  \tfrac{R}{n} g_{ij}-  \bar R_{ij}- \tfrac{\bar R}{n} g_{ij}\end{equation}   
is given by
\begin{equation} \label{dclc} 
\operatorname{diag}\left( \left( \tfrac{1}{f^2(f^2+C)}\right) \left(\tfrac{(n-1)(n-2)}{n}(ff''-  (f')^2)+ \tfrac{1}{n} \overset{0}{R}\right),  \underbrace{0,..,0}_{n-1}\right), 
\end{equation}
where $\overset{0}{R}$ is the scalar curvature of the $(n-1)$-dimensional metric $h_{ij}$. So only the ($i=1,j=1$)-component of the obtained matrix may be nonzero. 
We see that the condition \eqref{fedo1} is reduced to one ODE, namely to the ODE 
\begin{equation} \label{ODEf}
 {(n-1)(n-2)} (ff''-  (f')^2)+   \overset{0}{R}=0.
\end{equation} Note that the function $f$ depends only on the variable $t$ and the function   $\overset{0}{R}$  on the variables $x^1,..., x^{n-1}$, this implies that the scalar curvature of the $(n-1)$-dimensional metric $h_{ij}$ is locally a constant.

Let us now show that $B$ is constant.  We need  to calculate    calculate $\psi_{,ij}$ first: the only  Christoffel symbol we need is $\Gamma^0_{00}$  (we think that the index $0$ corresponds to the variable $t$) and it is given by  
$$
\Gamma^0_{00} = \tfrac{1}{2} {f^2} \tfrac{d}{dt}\left(\tfrac{1}{f^2}\right)= -\tfrac{f'}{f}. 
$$
Then, $$ \psi_{,00}= \tfrac{f''}{f} \  \textrm{and} \    \psi_{,0} = \tfrac{ f'}{f}  \ \textrm{so} \  \psi_{,00}- \psi_{,0}\psi_{,0}= \tfrac{ff''-(f')^2}{f^2}.  $$
All other components of $\psi_{,ij}- \psi_{,i}\psi_{,j}$ are zero. This implies 
$$
\Delta_2= ff'' \ \textrm{and } \ \Delta_1= (f')^2.  
$$ Consider the vector field $T^i:= \tfrac{\partial } {\partial t}$;
we know that it lies in the nullity and is therefore an eigenvector of
the Ricci tensor with eigenvalue $ (n-1)B$.  Clearly, $T^jR_{ij}=0$,
 so substituting
\eqref{conf-2} in the condition that $T^i$ is an eigenvector of the
Ricci tensor $\hat R_{ij}$ (of $g$) with eigenvalue $(n-1)B$ we obtain
$$  -(n-2) \tfrac{ f'' f-(f')^2}{f^2}  - \tfrac{1}{f^2}(ff'' + (n-2)(f')^2) =  (n-1)B,$$ 
which after simplifications gives us 
\begin{equation} \label{lastB}
B= -\tfrac{f''}{f}. 
\end{equation}

Now it is an easy exercise to show that for any solution of \eqref{ODEf} the function $B$ given by \eqref{lastB} is a constant. In order to do it, we combine 
 \eqref{ODEf} with \eqref{lastB} to obtain 
$$ B= -\tfrac{\overset{0}R}{(n-2)(n-1)}\tfrac{1}{f^2}+ \left( \tfrac{f'}{f}\right)^2.$$
 Then, 
 $$
\begin{aligned}
\tfrac{d}{dt}B &=  -2\tfrac{\overset{0}{R}}{(n-2)(n-1)}\tfrac{f'}{f^3}+ 2 \tfrac{f'}{f} \tfrac{ f'' f-(f')^2}{f^2}\\& = \tfrac{2}{(n-2)(n-1)}\tfrac{f'}{f^3} \bigl((n-2)(n-1) (f''f- (f')^2)+ \overset{0}R \bigr)\stackrel{\eqref{ODEf}}{=}0 .
\end{aligned}
$$
This shows that $B$ is constant. This is a contradiction as the result was derived by assuming $B$ not constant.  
\end{proof}  

\begin{corollary} \label{ccccc}
On a closed connected $M$ of dimension $n\ge 3$, suppose that $g$ and $\bar
g$ are non-affinely projectively equivalent and have the same
trace-free Ricci tensor.  Then,  for a certain constant $C\ne 0$ the metric $C g$ is the Riemannian metric of  constant
 positive sectional curvature. 
 \end{corollary}

\begin{remark} \label{liouville} Above, in \eqref{lastB}, we have a formula for $B$ for the metric \eqref{warped2}. By essentially the same calculations  one can obtain the formula for $\bar B= B_{\bar g}$ for $\bar g$ given by \eqref{bargwarped}: 
$$
\bar B= -{\frac {  \left(f^{2}+ C\right)f''-f   \left( f'
  \right)^{2}  }{f   }}. 
$$ 
By direct calculation we see that the Liouville tensors $L_{ijk}$ defined in Remark \ref{liu} constructed for $g$ and for $\bar g$ coincide, which proves that 
 $L_{ijk}$ is a metric projective invariant as  we claimed in Remark \ref{liu}. 
\end{remark}

\begin{remark} \label{liouville2} As we recalled  in Remark \ref{liu}, in dimension 2, vanishing of $L_{ijk}$ implies that the metric has constant curvature. For higher dimensions, it is not the case anymore locally, but still since vanishing of $L_{ijk}$ implies that $B=\operatorname{const}$, the following  
 global analogue is true: 

{\it Let a closed connected $(M, g)$  have Weyl nullity. Assume   $L_{ijk}\equiv 0$ on $M$. If    $\bar
g$ is  non-affinely projectively equivalent to $g$, then, 
 for a certain constant $C\ne 0$,  the metric $C g$ is the Riemannian metric of  constant
 positive sectional curvature. 
}  
\end{remark}

\subsection{Metrics with two dimensional nullity and a nonparallel 
solution of the metrisability equation have $B= \operatorname{const}$} \label{kathi}

As an easy by-product of calculations in the previous section we
obtain the following statement:

\begin{lemma}  Assume dimension is  $n\ge 3$. Suppose $g$  has nullity. Assume $dB\ne 0$ at a point. Then,  there exists an open set  $U$ containing this point and an open $U'\subset U$  such that  $U'$ is everywhere dense in $U$ and such that at each point  of $U'$  the nullity space is precisely  1-dimensional
\end{lemma}

An extension of this result to include dimension 2  is evidently impossible. In the case of dimension 2 the nullity space is 2-dimensional at every point, but $B$ is not necessary constant. 

\begin{proof} In \S \ref{Bnichtcon} we have seen that if $dB\ne 0$,  then  near this point  the metric is given by the warped product formula 
\eqref{warped2}. As in the previous section, consider the conformally   equivalent metric $\tfrac{1}{f^2} g$. It is the product metric, 
  $$\tfrac{1}{f^2} g = \tfrac{1}{f^2} dt^2 + \sum_{i,j=1}^{n-1} h(x)_{ij}dx^i dx^j.$$ 
  Its Ricci curvature   and the Ricci curvature of the initial metric $g$ are related   by \eqref{conf-2}  (note that in \eqref{conf-2} the Ricci curvature of    $\tfrac{1}{f^2} g$ is called $R_{ij}$ and  the Ricci curvature of    $ g$ is called $\hat R_{ij}$).  The function  $\psi:=  \log f$ depends on the variable $t$ only, which implies that    $\Delta_2 + (n-2)\Delta_1$ also depends on the variable $t$ only, if fact below we derive formulas for $\Delta_2$ and $\Delta_1$.

  Should the nullity space be at least two-dimensional (at every point
  of the neighborhood we are working in), then there would exist a
  vector field $X^i$ in the nullity such that it is orthogonal to
  $\tfrac{\partial } {\partial t}$. We know from Proposition
  \ref{mweyl} that this vector is an eigenvector of the Ricci tensor
  $\hat R_{ij}$ (of $g$) with eigenvalue $(n-1)B$. Since $g$ is a
  product metric, for any vector $Y^i$ orthogonal to $\tfrac{\partial
  } {\partial t}$ we have $\psi_{,ij}Y^i= \psi_i\psi_{,j} Y^i=0$, we
  obtain that 
  $$
 (n-1)B X^i=\hat R^{i}{}_jX^j \stackrel{\eqref{conf-2}}{=} \tfrac{1}{f^2}\overset{0}R{}^{i}{}_jX^j-  \tfrac{1}{f^2}(\Delta_2  + (n-2)\Delta_1)X^j  .   
  $$
  The factor $\tfrac{1}{f^2}$ appeared in the right hand side   because of the conformal coefficient $ \tfrac{1}{f^2} $.  By $\overset{0}R{}^{i}{}_j$ we understand the Ricci tensor of  the metric $\tfrac{1}{f^2} g$ with an index raised with the help of the metric $\tfrac{1}{f^2} g$; the notation is compatible with that one in the previous section. 
  
We see that $X^i$ is an eigenvector of $\overset{0}R{}^{i}{}_j$ with
eigenvalue $ (\Delta_2 + (n-2)\Delta_1)+{(n-1)}{f^2} B$.  By
construction the components of $\overset{0}R{}^i{}_j$ depend on the
coordinates $ x_1,...,x_{n-1}$. We also know that $ f^2 $, $ (\Delta_2
+ (n-2)\Delta_1) $ depend on the coordinate $t$ only. Further the same
is true of $B$,  see Remark \ref{forfuture} in Section \ref{Bnichtcon}. Thus,
\begin{equation} \label{cer}  (\Delta_2  + (n-2)\Delta_1) +(n-1)f^2  B = C  \end{equation}  for a  certain  constant $C$.

This is essentially the same equation as \eqref{ODEf}, and we already
know that for any solution of this equation the function $B$ is
actually a constant, see the last part of the proof of Theorem
\ref{last}. We see that also in this case $B$ is constant.
\end{proof}

\begin{corollary} \label{cccccc}
On a closed connected $M$ of dimension $n\ge 3$, suppose that $g$ and
$\bar g$ are non-affinely projectively equivalent. Assume that Weyl
nullity space is at least two-dimensional at every point.  Then, for a
certain constant $C\ne 0$ the metric $C g$ is the Riemannian metric of
constant positive sectional curvature.
 \end{corollary}

\end{document}